\documentclass[12pt,lot, lof]{amsart}
\author[]{Mohammad  Farajzadeh Tehrani\\ Part 2 joint with Aleksey Zinger}
\address{Simons Center for Geometry and Physics}
\email{tehrani@scgp.stonybrook.edu}
\title[]{Counting genus zero real curves\\ in symplectic manifolds}
\usepackage{epsfig,pstricks}
\usepackage{amsfonts,amssymb,amsmath}   
\usepackage{amsthm}
\usepackage{latexsym}
\usepackage{graphicx}
\usepackage[all]{xy}
\usepackage{verbatim}
\usepackage{multirow}
\usepackage{longtable}
\usepackage{booktabs}
\usepackage{enumitem}

\numberwithin{equation}{section}
\newtheorem{theorem}{Theorem}[section]
\newtheorem{lemma}[theorem]{Lemma}
\newtheorem{corollary}[theorem]{Corollary}
\newtheorem{proposition}[theorem]{Proposition}

\theoremstyle{definition}
\newtheorem{definition}[theorem]{Definition}
\newtheorem{remark}[theorem]{Remark}
\newtheorem{example}[theorem]{Example}

\def\b{\mathbf b}
\def\A{\mathbf A}

\def\C{\mathbb C}
\def\H{\mathbb H}
\def\I{\mathbb I}

\def\P{\mathbb P}
\def\Q{\mathbb Q}
\def\R{\mathbb R}
\def\T{\mathbb T}
\def\Z{\mathbb Z}

\def\z{\mathbf z}
\def\fL{\mathfrak L}

\def\mc{\mathcal}
\def\cL{\mc{L}}
\def\cN{\mc{N}}

\def\De{\Delta}
\def\La{\Lambda}
\def\Si{\Sigma}
\def\Th{\Theta}

\def\de{\delta}
\def\ep{\epsilon}
\def\la{\lambda}
\def\om{\omega}
\def\si{\sigma}
\def\th{\theta}
\def\vt{\vartheta}
\def\ze{\zeta}

\def\fa{\mathfrak a}
\def\fc{\mathfrak c}
\def\fI{\mathfrak i}

\def\lra{\longrightarrow}
\def\ra{\rightarrow}

\def\ov#1{\overline{#1}}
\def\tn#1{\textnormal{#1}}
\def\mf#1{\mathfrak{#1}}
\def\wt#1{\widetilde{#1}}
\def\wh#1{\widehat{#1}}
\def\flr#1{\lfloor{#1}\rfloor}
\def\lr#1{\langle{#1}\rangle}
\def\Eq#1{\begin{equation}\label{#1}}
\def\EEq{\end{equation}}
\def\sEq#1{\begin{equation*}\label{#1}}
\def\sEEq{\end{equation*}}
\def\Thm#1{\begin{theorem}\label{#1}}
\def\EThm{\end{theorem}}
\def\Lem#1{\begin{lemma}\label{#1}}
\def\ELem{\end{lemma}}
\def\Prop#1{\begin{proposition}\label{#1}}
\def\EProp{\end{proposition}}
\def\Fix{\tn{Fix}}
\def\ev{\tn{ev}}
\def\id{\tn{id}}
\def\vir{\tn{vir}}

\def\top{\tn{top}}
\def\bot{\tn{bot}}
\def\down{\tn{bot}}

\def\disk{\tn{disk}}
\def\FS{\tn{FS}}

\def\dec{\tn{dec}}
\def\Def{\tn{Def}}
\def\Aut{\tn{Aut}}

\def\half{\tn{half}}
\def\val{\tn{val}}
\def\mov{\tn{mov}}
\def\Res{\tn{Res}}
\def\nd{\tn{d}}
\def\mfi{\mf{i}}

\def\bu{\bullet}

\def\ne{\tn{e}}
\begin{document}
\date{\small Dec~4, 2012. Updated: \today}
\begin{abstract}
There are two types of $J$-holomorphic spheres in a symplectic manifold invariant under an anti-symplectic involution:
those that have a fixed point locus and those that do not. The former are described by moduli spaces of $J$-holomorphic disks, which are well studied 
in the literature. 
In this paper, we first study moduli spaces describing the latter and then combine the two types of moduli spaces to get a well-defined theory of counting real curves of genus 0. We use equivariant localization to show that these invariants (unlike the disk invariants) are essentially the same 
for the two (standard) involutions on $\P^{4n-1}$. 
\end{abstract}
\maketitle
\tableofcontents

\section{Introduction and main results\label{ch:introduction}}

Let $(X,\om,\phi)$ be a symplectic manifold, which we will assume to be connected throughout this paper, 
with a real structure $\phi$, i.e a diffeomorphism $\phi\colon X\to X$ such that $\phi^2=\id_X$ and  
$\phi^*\om=-\om$. Let $L=\Fix(\phi) \subset X$ be the fixed point locus of $\phi$; $L$ is a Lagrangian submanifold of $(X,\om)$ which can be empty. 
In the simplest case of $(X,\omega)=(\P^1,\om_{\FS})$, where $\om_{\FS}$ is the Fubini-Study symplectic form, there are involutions of both types. An almost complex structure $J$ on $TX$ is called \textsf{$(\om,\phi)$-compatible} if $\phi^*J=-J$ and $\om(\cdot,J\cdot)$ is a metric. Denote the set of such almost complex structures by $\mc{J}_{\om,\phi}$ or simply $\mc{J}_{\phi}$.

Fix a compatible almost complex structure $J$. Let $u\colon \P^1 \to X$ be an $n$-marked somewhere injective $J$-holomorphic sphere, i.e.
\begin{equation}\label{swi}
du+J\circ du\circ j=0,\qquad u^{-1}(u(z))=\{z\} \quad \textnormal{for almost every }z\in \P^1,
\end{equation}
where $j$ is the complex structure of $\P^1$. We call such a $J$-holomorphic map \textsf{real} if its image (as a marked curve) is invariant under the action of $\phi$. In this case, pulling back $\phi$ to $\P^1$, we get an involution on $\P^1$, which may or may not have fixed points and preserves the set of marked points. After a change of coordinates, an anti-symplectic involution with fixed points can be written as 
$$\tau\colon \P^1 \to \P^1, \qquad \tau([z,w])=[\bar{w},\bar{z}],$$
while a fixed point free involution  can be written as 
\begin{equation}\label{equ:etainv}
\eta\colon \P^1 \to \P^1, \qquad \eta([z,w])=[\bar{w},-\bar{z}].
\end{equation}

For $k,l\in \Z^{\geq 0}$ and $A\in H_2(X)$, we define $\mc{M}_{k,l}(X,A)^{\phi,\tau}$ and $\mc{M}_{l}(X,A)^{\phi,\eta}$ to be the moduli spaces of degree $A$ genus zero $J$-holomorphic curves $u\colon \P^1 \to X$ satisfying
\begin{equation}\label{equ:taumaps}
u=\phi\circ u\circ \tau\qquad \hbox{and}\qquad  u=\phi\circ u\circ \eta,
\end{equation}
respectively, with $l$ disjoint ordered conjugate pairs of marked points, along with $k$ real ($\tau$-fixed) marked points in the first case.
Similar to \cite[Appendix C]{MS2}, these moduli spaces 
have real virtual dimension
\begin{equation}\label{dimension}
\begin{split}
\dim^{\vir}\mc{M}_{k,l}(X,A)^{\phi,\tau}&= \dim_{\C}X+c_1(A)+2l+k-3,\\
\dim^{\vir}\mc{M}_l(X,A)^{\phi,\eta}   & = \dim_{\C}X+c_1(A)+2l-3.
\end{split}
\end{equation}
Every $J$-holomorphic map $u\colon \P^1 \to X$ in $\mc{M}_{k,l}(X,A)^{\phi,\tau}$ corresponds to two $J$ holomorphic disks $u\colon (D^2,S^1) \to (X,L)$ with $k$ boundary marked points and $l$ $(\pm)$-decorated\footnote{Decorated moduli spaces are studied in \cite{G2}.} interior marked points, representing $\beta, -\phi_*\beta \in H_2(X,L)$; the $j$-th decoration is $(+)$ if the first point of the conjugate pair $(z_j,\tau(z_j))$ lies on the chosen disk and is $(-)$ otherwise. We define $\mc{M}_{k,l}^{\disk}(X,L,\beta)_{\dec}$ and $\mc{M}_{k,l}^{\disk}(X,L,\beta)$ to be the moduli space of such $J$-holomorphic disks with and without decorations, respectively.  
Let $\ov{\mc{M}}_{l}(X,A)^{\phi,\eta}$ and $\ov{\mc{M}}_{k,l}(X,A)^{\phi,\tau}$ be the stable map compactifications of $\mc{M}_{l}(X,A)^{\phi,\eta}$ and $\mc{M}_{k,l}(X,A)^{\phi,\tau}$, respectively. Let 
\begin{equation}\label{evaluation-maps}
\begin{aligned}
& \ev_i    \colon \ov{\mc{M}}_{l}(X,A)^{\phi,\eta}   \to X,  &\ev_i([u,\Si,(z_j,\eta(z_j))_{j=1}^{l}])= u(z_i), \\
& \ev_i^{B}\colon \ov{\mc{M}}_{k,l}(X,A)^{\phi,\tau} \to L,  &\ev_i^{B}([u,\Si,(w_j)_{j=1}^k,(z_j,\tau(z_j))_{j=1}^{l}])= u(w_i),\\
& \ev_i    \colon \ov{\mc{M}}_{k,l}(X,A)^{\phi,\tau} \to X,  &\ev_i([u,\Si,(w_j)_{j=1}^k,(z_j,\tau(z_j))_{j=1}^{l}])= u(z_i),
\end{aligned}
\end{equation}
be the natural evaluation maps.

For the classic moduli space $\ov{\mc{M}}_{n}(X,A)$ of $J$-holomorphic spheres in a homology class $A$, genus $0$  \textsf{Gromov-Witten invariants} are defined via integrals of the form
\begin{equation}\label{GW-invariants}
\left\langle \theta_1,\cdots,\theta_n\right\rangle_{A}= \int_{[\ov{\mc{M}}_{n}(X,A)]^{\vir}} \ev_1^*(\theta_1)\wedge \cdots \wedge \ev_n^*(\theta_n),
\end{equation}
where $\theta_i$'s are cohomology classes on $X$; see \cite{FO,LT,RT}. 
These integrals make sense and are independent of $J$, because $\ov{\mc{M}}_{n}(X,A)$ has a (virtually) orientable fundamental cycle without real codimension one boundary. One would like to define similar invariants for the moduli spaces $\ov{\mc{M}}_{k,l}^{\disk}(X,L,\beta)$ and the evaluation maps in (\ref{evaluation-maps}). 
The existence of such invariants is predicted by physicists \cite{F1,F2,F3,F4}), but there are obstacles to defining such invariants mathematically. In addition to the transversality issues (which are also present in the classical case), issues concerning orientability and codimension one boundary arise.

\subsection{Disk or $\tau$-invariants}\label{sec:open}

Whereas moduli spaces of closed curves have a canonical orientation induced by $J$, $\mc{M}_{k,l}^{\disk}(X,L,\beta)$ is not necessarily orientable. Moreover, if it is orientable, there is no canonical orientation. If $L$ has a spin (or relative spin) structure, then $\mc{M}_{k,l}^{\disk}(X,L,\beta)$ is orientable and a choice of spin structure canonically  determines an orientation on $\mc{M}_{k,l}^{\disk}(X,L,\beta)$; see \cite[Theorem 8.1.1]{FOOO}.

Let $\iota:H_2(X)\to H_2(X,L)$ be the inclusion homomorphism. 
The union of moduli spaces $\ov{\mc{M}}_{k,l}^{\disk}(X,L,\beta)_{\dec}$ over all $\beta\in H_2(X,L)$ such that $\iota(A)=\beta-\phi_*\beta$ is an \'etale\footnote{\'Etale double covering means that over the main stratum $\mc{M}_{k,l}^{\disk}(X,L,\beta)$, it is double covering; however, over the boundary strata it has higher degrees. } double covering of $\ov{\mc{M}}_{k,l}(X,A)^{\phi,\tau}$, with the deck transformation 
\begin{equation}\label{equ:tauM}
\tau_{\mc{M}}: [u,(w_j)_{j=1}^k,(z_j,\ep_j)_{j=1}^{l}] \to [\phi \circ u \circ c,(w_j)_{j=k}^1,(c(z_j),-\ep_j)_{j=1}^{l}],
\end{equation}
where $\ep_j=\pm$ is the decoration and $c(z)=\bar{z}$; see \cite[Theorem 1.1.(3)]{G2}, if $l>0$, and \cite[Section 1.3.4]{PSW}, if $l=0$, for a more detailed description of this covering map. At several points in the paper, we go back and forth between the two descriptions to relate the known results for $J$-holomorphic disks with the corresponding statements for the $(\phi,\tau)$-real maps. 

For every $\beta\in H_2(X,L)$, $k,l\geq 0$, a choice of spin structure on $L=\Fix(\phi)$ 
determines an orientation on $\mc{M}_{k,l}^{\disk}(X,L,\beta)_{\dec}$,
with the anti-complex orientation imposed on the tangent spaces at the $(-)$ marked points, 
as in \cite[Section 4]{G2}. 
By \cite[Theorem 1.3]{FO-3} and \cite[Corollary 5.4]{G}, 
$\tau_{\mc{M}}$ is orientation-preserving if and only if 
$\frac{\mu(\beta)}{2}+k$ is even, where $\mu(\beta)\in 2\Z$ is the Maslov index of $\beta$. 
In particular, if $L$ is spin and $4|c_1(TX)$ (i.e. $4|c_1(A)$ for every $A\in H_2(X)$), 
then $\mc{M}_{0,l}(X,A)^{\phi,\tau}$ or simply $\mc{M}_{l}(X,A)^{\phi,\tau}$ is orientable, 
while $\tau_{\mc{M}}$ is orientation-reversing on $\mc{M}_{1,l}^{\disk}(X,L,\beta)_{\dec}$.

The boundary of $\ov{\mc{M}}_{k,l}(X,A)^{\phi,\tau}$, i.e. the subspace of maps with at least one node, has two types of real (virtual) codimension one strata; see Figure~\ref{Fig:boundaryterms}. 
The first type, called \textsf{disk bubbling}, consists of maps from two spheres with a real point in common. 
This strata breaks into unions of components isomorphic to 
\begin{equation}\label{boundary-equ}
\mc{M}_{k_1+1,l_1}(X,A_1)^{\phi,\tau} \times_{(\ev^B_1,\ev^B_1)} \mc{M}_{k_2+1,l_2}(X,A_2)^{\phi,\tau}/G,
\end{equation}
where 
$$l_1+l_2=l,\quad k_1+k_2=k,\quad A_1+A_2=A, \quad
G=\begin{cases}\Z_2,& \hbox{\tn{if}}~k,l=0,~A_1=A_2;\\
\{1\},& \hbox{\tn{otherwise}}.
\end{cases}$$

\begin{figure}
\begin{pspicture}(-3,-2.5)(11,2.4)
\psset{unit=.3cm}
\pscircle*(0,-2){.3}
\psellipse[linewidth=.07](-4,-2)(4,1)
\psellipse[linewidth=.07](4,-2)(4,1)
\psarc[linewidth=0.07](-4,-2){4}{0}{180}
\psarc[linewidth=0.07](4,-2){4}{0}{180}
\rput(4,0){$u_2$}\rput(-4,0){$u_1$}\rput(0,-8){Disk bubbling}\rput(0,-4){node}\rput(-10,-4){$L$}
\psline[linewidth=0.07](-12,-5)(7,-5)(12,3)(-7,3)(-12,-5)
\psline[linewidth=0.07](13,-5)(32,-5)(37,3)(29,3)
\psline[linewidth=0.07](21,3)(18,3)(13,-5)
\rput(25,-8){Sphere bubbling}\rput(25,-4){collapsed boundary}\rput(15,-4){$L$}
\pscircle*(25,-2){.3}
\pscircle[linewidth=0.07](25,2){4}
\psellipse[linewidth=.07](25,2)(4,1)
\end{pspicture}
\caption{Half of the curves in the codimension one boundary strata of $\ov{\mc{M}}_{k,l}(X,A)^{\phi,\tau}$}\label{Fig:boundaryterms}
\end{figure}
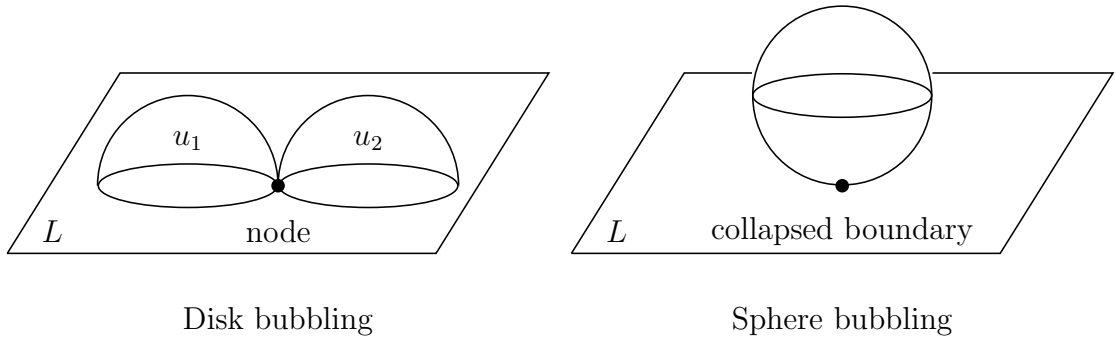

The second type, called \textsf{sphere bubbling}, appears only if $k=0$ and $A=B-\phi_*B$ for some $B\in H_2(X)$.
It consists of maps from  nodal domain $\Si=\P^1\cup_q \P^1$, taking the node $q$ to $L$. 
This strata is isomorphic to ($\Z_2$-quotient of) 
\begin{equation}\label{boundary-equ2}
 \bigsqcup_{\begin{subarray}{l} B\in H_2(X)\\ B-\phi_* B=A \end{subarray}} (\mc{M}_{1+l}(X,B)_{\dec}\times_{\ev_1} L),
\end{equation}
where the intersection point with $L$, which corresponds to the first marked point in the $1+l$ marked points, has no decoration. 
Note that the natural extension of $\tau_{\mc{M}}$ preserves each component of the domain of every map in the first case and interchanges them in the second case.

If a codimension one  strata is a boundary, an integral similar to (\ref{GW-invariants}) depends on the particular choices of the integrands (and other choices); thus, does not produce invariants. The boundary problem is present in nearly all cases. 
In the disk formulation, it has been overcome in a number of cases by either adding other terms to compensate for the effect of the boundary \cite{W1,W2,F} or 
by gluing boundary components to each other to get moduli spaces without boundary 
\cite{S,G2}. In the real curve formulation, the latter approach actually shows that the disk-bubbling strata is a hypersurface in the real moduli space across which the orientation extends; 
see Proposition~\ref{flip-orientation}.
None of these methods can address the issue of sphere bubbling; we address it in this paper.

\subsection{$\eta$-invariants}\label{sec:real}

The moduli spaces $\ov{\mc{M}}_{l}(X,A)^{\phi,\eta}$ have mostly been ignored in the literature. 
As we show, the codimension one boundary consists of maps from a wedge of two spheres taking the node to $L$. 
The restrictions of each map to the two spheres determine elements of $\mc{M}_{1+l}(X,B)_{\dec}$ and $\mc{M}_{1+l}(X,-\phi_*B)_{\dec}$ that differ by the involution
\begin{equation}
\begin{split}
\phi_{\mc{M}}\colon \mc{M}_{1+l}(X,B)_{\dec}\times_{\ev_1} L &\to \mc{M}_{1+l}(X,-\phi_*B)_{\dec}\times_{\ev_1} L ,\\
[u,z_0,(z_1,\ep_1),\cdots,(z_l,\ep_l)] &\to [\phi\circ u \circ c, c(z_0),(c(z_1),-\ep_1),\cdots,(c(z_l),-\ep_l) ],
\end{split}
\end{equation}
where $c\colon \P^1 \to \P^1$, and $c(z)=\bar{z}$. 
Thus, the codimension one boundary breaks into unions of components isomorphic to ($\Z_2$-quotient)
\begin{equation}\label{eq:sphere-bubbling}
   \mc{M}_{1+l}(X,B)_{\dec}\times_{\ev_1}L,
\end{equation}
with $B\in H_2(X)$ such that $A=B-\phi_*B$. In particular, if $\Fix(\phi)=\emptyset$, there are no codimension one boundary components, and we obtain the following result.

\begin{proposition}\label{boundary-freeness}
If $(X,\omega,\phi)$ is a symplectic manifold with a real structure $\phi$ and $\Fix(\phi)=\emptyset$, $\ov{\mc{M}}_{l}(X,A)^{\phi,\eta}$  has a topology with respect to which it is compact and Hausdorff. 
It has a Kuranishi structure without boundary of virtual real dimension
$$  d= c_1(A)+\dim_{\C}X-3+2l. $$
\end{proposition} 
This proposition and Theorem~\ref{thm:main} are proved in Section \ref{ch:realcurves} based on \cite{FOOO} by providing the adjustments to the real case.
\begin{remark}
There are many symplectic manifolds $(X,\om)$ admitting anti-symplectic involutions without fixed points. 
For example, the involution $\eta_{2m-1}$ on $\P^{2m-1}$ defined in (\ref{equ:tau-eta}) has no fixed points. 
Furthermore, the symplectic cut of \cite[Section 2]{MF1} associates to each real symplectic manifold $(X,\om,\phi)$ with $\Fix(\phi)\cong S^n,\R\P^n$ a real symplectic manifold $(X_+,\om_+,\phi_+)$ with $\Fix(\phi_{+})=\emptyset$ by ``cutting out" $\Fix(\phi)$ and replacing that with a divisor. 
\end{remark}

In order to define invariants, we also need to consider the orientation problem, which has not been studied before. 
A \textsf{real structure} on a complex vector bundle $E\to X$ is an anti-complex linear involution $\phi_E\colon E\to E$ covering $\phi$. 
A \textsf{real square root} of a complex line bundle $\cL \to X$ with real structure $\phi_\cL$ is a complex line bundle 
$\cL'\to X$ with real structure $\phi_{\cL'}$ such that 
$$ (\cL,\phi_\cL) \cong (\cL'\otimes \cL', \phi_{\cL'}\otimes \phi_{\cL'}).$$ 
The involution $\phi$ on $X$ canonically lifts to an involution $\phi_{K_X}$ on the complex line bundle $K_X=\Lambda_{\C}^{\top} T^*X$.

\begin{theorem}\label{square-root-orientation-thm}
Let $(X,\om,\phi)$ be a symplectic manifold with a real structure. 
If $(K_X,\phi_{K_X})$ admits a real square root, all moduli spaces $\mc{M}_{l}(X,A)^{\phi,\eta}$ are orientable. 
Moreover, a choice of real isomorphism $(K_X,\phi_{K_X}) \cong (\cL\otimes \cL, \phi_{\cL}\otimes \phi_{\cL})$ canonically determines the orientation.
\end{theorem} 

This theorem is proved in Section~\ref{orientation}. 
By abuse of terminology, throughout the rest of this paper, by a real square root we mean a choice of complex line bundle $\cL$ on $X$ with a real structure $\phi_\cL$, together with a choice of real bundle isomorphism $(K_X,\phi_{K_X}) \cong (\cL\otimes \cL, \phi_{\cL}\otimes \phi_{\cL})$.

\begin{remark}\label{evenness-rmk}
If $\cL\to \P^1$ is a holomorphic line bundle with a complex anti-linear involution lift $\tilde{\eta}$ of $\eta\colon \P^1 \to \P^1$, for all $k\in \Z$ there is a decomposition $$H^0(\cL\otimes (T\P^1)^{\otimes k})= H_+^0(\cL\otimes (T\P^1)^{\otimes k})\oplus H_-^0(\cL\otimes (T\P^1)^{\otimes k})$$
into the $\pm1$ eigenspaces of the endomorphism 
$$ H^0(\cL\otimes (T\P^1)^{\otimes k}) \to H^0(\cL\otimes (T\P^1)^{\otimes k}) ,\quad \xi \ra \tilde{\eta}\circ \xi \circ \eta;$$
the two eigenspaces are interchanged by the action of $\mf{i}$. 
Since the action of $\eta$ on $\P^1$ has no fixed points and $H^0(\cL\otimes (T\P^1)^{\otimes k})$ is nonzero for $k$ large enough, the zeros of every element of $H_+^0(\cL\otimes (T\P^1)^{\otimes k})$
come in pairs and thus $\deg \cL$ is even. 
Hence, if $\mc{M}_{l}(X,A)^{\phi,\eta}$ is non-empty, then $2|K_X(A)$. 
Thus, if $K_X$ has a real square root, then $4|K_X(A)$ whenever $\mc{M}_{l}(X,A)^{\phi,\eta}$ is non-empty. 
The last requirement can not be removed. For example, if $(X,\phi)=(\P^{4m+1},\tau_{4m+1})$, then $K_X$ has a real square root but $4\nmid K_X(\ell)$, where $\ell \subset H_2(\P^{4m+1})$ is the homology class of complex projective line. 
\end{remark}
 
If $(X,\om,\phi)$ is a K\"ahler manifold with an anti-holomorphic anti-symplectic involution $\phi$ and $\cL'\to X$ is a holomorphic line bundle, 
then $\cL'\otimes \ov{\phi^*\cL'}$ is a holomorphic line bundle with a real structure. 
Hence, if $\cL\ra X$ is a holomorphic line bundle, $\cL=\cL'\otimes \cL'$, and $\ov{\phi^*\cL'} \cong \cL'$, then $\cL$ admits a real structure. 
Suppose $4|K_X$, i.e. there is a divisor $D$ such that $K_X=[4D]$. 
Since $\ov{\phi^*K_X}=K_X$, it follows that $[D-\ov{\phi_*D}]$ is torsion.
\begin{proposition}\label{prop:existence}
Let $(X,\om,\phi)$ be a symplectic manifold with a real structure.
If either
\begin{enumerate}[leftmargin=*]
\item $H^1(X;\R)=0$ and $c_1(TX)=4\alpha$ for some $\alpha\in H^2(X;\Z)$ such that $\alpha=-\phi^*\alpha$, or
 \item $X$ is compact K\"ahler, $\phi$ is anti-holomorphic, and $K_X=[4D]$ for some divisor $D$ on $X$ such that $[D]=[\ov{\phi_*D}]$,
\end{enumerate}
then $(K_X,\phi_{K_X})$ admits a real square root.
\end{proposition}

We prove this proposition in Section~\ref{sec:remarks}. 
An example with $\ov{\mc{M}}_{l}(X,A)^{\phi,\eta}$ non-orientable is described in Section~\ref{sec:remarks}. 
In the simply connected case, \cite[Example 2.6]{XCapsSetup} provides an example
 where $\ov{\mc{M}}_{l}(X,A)^{\phi,\eta}$ is not orientable.
 
\subsection{Real GW invariants}\label{mixed-invariants}

If $L=\Fix(\phi)\neq \emptyset$ and the sphere bubbling is present ($k=0$ and $A=B-\phi_*B$ for some $B\in H_2(X)$), we cannot define either the $\tau$-invariants nor the $\eta$-invariants separately. It is noted in \cite[Section 1.5]{PSW}, that in order to get well-defined invariants in these case, the moduli spaces $\ov{\mc{M}}_{l}(X,A)^{\phi,\tau}$ and $\ov{\mc{M}}_{l}(X,A)^{\phi,\eta}$ need to be combined somehow. This is achieved in this paper.

As described in Sections \ref{sec:open} and \ref{sec:real}, the codimension one boundary corresponding to sphere bubbling in $\ov{\mc{M}}_{l}(X,A)^{\phi,\tau}$ is the same as the codimension one boundary of $\ov{\mc{M}}_{l}(X,A)^{\phi,\eta}$.
By attaching $\ov{\mc{M}}_{l}(X,A)^{\phi,\tau}$ and $ \ov{\mc{M}}_{l}(X,A)^{\phi,\eta}$ along their common boundary (i.e. considering all genus 0 real curves representing class $A$), we obtain a moduli space $\ov{\mc{M}}_{l}(X,A)^{\phi}$ whose only possible codimension one boundary corresponds to disk bubbling. We then use the results of \cite{S} and \cite{G2} and observe that the codimension one strata of $\ov{\mc{M}}_{l}(X,A)^{\phi}$ corresponding to disk bubbling are in fact hypersurfaces and therefore $\ov{\mc{M}}_{l}(X,A)^{\phi}$ (virtually) does not have any codimension one boundary.

If $K_X$ has a real square root, $L$ is spin, and $4|c_1(TX)$, the moduli spaces $\ov{\mc{M}}_{l}(X,A)^{\phi,\eta}$ and $\ov{\mc{M}}_{l}(X,A)^{\phi,\tau}$ are oriented. By studying the orientation along the common boundary we show that the union is also orientable. 

If $K_X$ admits a real square root, $(\cL,\phi_\cL)$, as above,
\begin{equation}\label{equ:ind-tri}
\La_\R^{\top}TL = \cL^{\phi_\cL} \otimes \cL^{\phi_\cL},
\end{equation}
thus the Lagrangian $L$ is orientable and the induced orientation is independent of the choice of real square root. 
A spin structure on $L$ is a trivialization $L[2]\times \R^{\dim_\R L}$ of $TL$ over the $2$-dimensional skeleton $L[2]$ of a triangulation of $L$. 
Given such a trivialization, by taking the determinant of that, we obtain a trivialization of 
$\La_\R^\top TL$ over $L[2]$. 
Therefore, if we know that $L$ is orientable, there is a unique choice of orientation on $L$ which is equal to the one induced by the spin structure on $(\La_\R^\top TL)|_{L[2]}$ as above. 

\begin{definition}\label{dfn:comp-spin-sqroot}
We say that a given real square root for $K_X$ and a given spin structure on $L$ are \textsf{compatible} if their induced orientations on $L$ as above are reverse of each other. 
\end{definition}

In the situation of Definition~\ref{dfn:comp-spin-sqroot}, we would orient $L$ with the induced orientation of the spin structure.

\begin{theorem}\label{thm:main}
If $(X,\om,\phi)$ is a symplectic manifold with a real structure $\phi$, $\ov{\mc{M}}_l(X,A)^{\phi}$ has a topology with respect to which it is compact and Hausdorff. It has a Kuranishi structure without boundary of virtual real dimension 
$$d=c_1(A)+dim_{\C}(X)-3+2l.$$
If in addition $4|c_1(TX)$, then a compatible pair of a real square root for $K_X$ and a spin structure on $L$ determines an orientation on $\ov{\mc{M}}_l(X,A)^{\phi}$, hence a virtual fundamental class $[\ov{\mc{M}}_l(X,A)^{\phi}]^\vir$.
\end{theorem}

We prove the first part of this theorem in Section~\ref{kuranishi-structure} and the second part in Section~\ref{ch:mixed-invariants}. 
We call the resulting invariants \textsf{real} GW invariants. 
The moduli space $\ov{\mc{M}}_{l}(X,A)^{\phi}$ provides a framework to define real GW invariants without any restriction on the topology of the image or the involution. 
If $\mc{M}_{l}(X,A)^{\phi,\eta}$ or $\mc{M}_{l}(X,A)^{\phi,\tau}$ is empty, the real invariants reduce to the disk invariants or $\eta$-invariants above. If $\ov{\mc{M}}_{l}(X,A)^{\phi}$ is not orientable, we may still consider invariants with twisted coefficients (coefficients in the orientation bundle). 
For example, \cite{G} shows that for some cases where both the Deligne-Mumford space and $\mc{M}_{k,l}(X,A)^{\phi,\tau} $ are not orientable, invariants with twisted coefficients pulled-back from the Deligne-Mumford moduli space exist. In our case, the Deligne-Mumford space is orientable, however, one may still find non-orientable geometric cycles within $\ov{\mc{M}}_{l}(X,A)^{\phi}$ that provide the necessary twisting coefficients.

For example, if $\Fix(\phi)\!=\! L\cong\! S^3$, $X$ is a real symplectic Calabi-Yau threefold\footnote{Following \cite[Section 14.2]{GHJ}, by a ``symplectic Calabi-Yau" we mean a connected symplectic manifold of vanishing first Betty number.}, and $A\in H_2(X)$ is non-trivial, then $\ov{\mc{M}}(X,A)^{\phi}$ is (virtually) zero-dimensional and orientable. In fact, $TL$ is trivializable (hence it is spin), by Proposition~\ref{prop:existence} every real symplectic Calabi-Yau threefold admits a real square root, and by Theorem~\ref{thm:main} we should choose the one which is compatible with the chosen spin structure on $L$; therefore, the orientation of $\ov{\mc{M}}(X,A)^{\phi}$ depends on the choice of spin structure on $L$. In this case we cannot define disk invariants or $\eta$-invariants separately. We define genus $0$ real GW invariants of $(X,\phi)$ by 
$$ N_A^{\phi}(X)= \# [\mc{M}(X,A)^{\phi}]^{\vir} \in \Q.$$  
By applying the degeneration technique of \cite{MF1}, we prove the following Theorem in Section~\ref{sec:zeroes}.
It implies that for some $J$, the only contribution to $N_A^{\phi}(X)$ is from $\eta$-curves. Hence, it demonstrates that  considering only $J$-holomorphic disks does not suffice to get non-trivial invariants in this set of examples.

\begin{theorem}\label{zeroness}
Let $(X,\omega,\phi)$ be a real symplectic Calabi-Yau threefold. If $\Fix(\phi)\!\cong\! S^3$, for every nonzero $A\in H_2(X,\Z)$, there exists an almost complex structure $J\in \mc{J}_{\omega,\phi}$ such that the $\tau$-moduli space $\ov{\mc{M}}(X,A,J)^{\phi,\tau}$ is empty.
\end{theorem}
 
In fact, in \cite{MF1}, we show that there is a natural Hamiltonian $S^1$-action on a neighborhood of $L$ in $X$. Applying the symplectic cut and symplectic sum procedures to this action, we build a symplectic fibration $\pi:\mc{X}\to \De$ over a disk in $\C$, where the smooth fibers are symplectomorphic to $X$ and the central fiber is normal crossing, $X_0=X_-\cup_D X_+$. We get an induced anti-symplectic involution on $\mc{X}$ which leaves $X_{\pm}$ invariant and restricted to $X_+$ has no fixed point. Moreover, we get a canonical inclusion of $H_2(X)$ in $H_2(X_+)$. Via the symplectic sum procedure, every almost complex structure $J_0$ on $X_0$, i.e. a union of two almost complex structures $J_+$ and $J_-$ on $X_+$ and $X_-$, respectively, where both preserve $D$, extends to an almost complex structure $J$ on $\mc{X}$ which is compatible with the fibration and the symplectic structure. We can think of $J_\la=J|_{X_\la}$, $X_\la=\pi^{-1}(\la)$, as a family of almost complex structures on $X$ converging to a singular almost complex structure.
Then, for any $E>0$, we show that there exists $J_0$ and $0<\la_0$ such that $\ov{\mc{M}}(X,A,J_\la)^{\phi,\tau}$, whenever $0<\la<\la_0$ and $\omega(A)<E$, is empty.

\subsection{Projective spaces  \tn{(joint with A.~Zinger)}}\label{IntroPn_subs}
We now discuss in some detail the case $X\!=\!\P^{2m-1}$. 
The involutions \hbox{$\tau,\eta\!:\P^1\to\P^1$} 
are special cases of the anti-holomorphic involutions
$$\tau_{2m-1},\eta_{2m-1}\!:\P^{2m-1}\to\P^{2m-1},$$ 
where
\begin{equation}\label{equ:tau-eta}
\aligned
\tau_{2m-1}\big([Z_1,Z_2,\ldots,Z_{2m-1},Z_{2m}]\big)
&=\big([\bar{Z}_2,\bar{Z}_1,\ldots,\bar{Z}_{2m},\bar{Z}_{2m-1}]\big),\\
\eta_{2m-1}\big([Z_1,Z_2,\ldots,Z_{2m-1},Z_{2m}]\big)
&=\big([-\bar{Z}_2,\bar{Z}_1,\ldots,-\bar{Z}_{2m},\bar{Z}_{2m-1}]\big).
\endaligned
\end{equation}
The fixed locus of~$\tau_{2m-1}$ is the real projective space~$\R\P^{2m-1}$,
while the fixed locus of~$\eta_{2m-1}$ is empty.
The latter implies 
\begin{equation}\label{etaPn_e}
\ov{\mc{M}}_l(\P^{2m-1},d)^{\eta_{2m-1},\tau}=\emptyset.
\end{equation}
The next observation is established in Section~\ref{Podd}.

\begin{lemma}\label{new-lemma}
Suppose $d,m\!\in\!\Z^{+}$ and $l\!\in\!\Z^{\geq 0}$. Then,
\begin{equation}\label{degPn_e}\begin{split}
\ov{\mc{M}}_l(\P^{2m-1},d)^{\tau_{2m-1},\eta}=\emptyset\qquad&\hbox{if}~~d\!\not\in\!2\Z,\\
\ov{\mc{M}}_l(\P^{2m-1},d)^{\eta_{2m-1},\eta}=\emptyset \qquad&\hbox{if}~~d\!\in\!2\Z.
\end{split}
\end{equation}
\end{lemma}

Since $K_{\P^{4m-1}}=[-4m\P^{4m-2}]$ and $\R\P^{4m-1}$ is spin, 
by Proposition~\ref{prop:existence} and Theorem~\ref{thm:main},
$\ov{\mc{M}}_l(\P^{4m-1},d)^{\phi}$ is orientable for $\phi\!=\!\tau_{4m-1},\eta_{4m-1}$. 
In fact, Euler's sequence of holomorphic vector bundles 
\begin{equation}\label{Pnses_e} 
0\lra \P^{n-1}\!\times\!\C \stackrel{f}{\lra} n\mc{O}_{\P^{n-1}}(1) 
\stackrel{h}{\lra} T\P^{n-1}\lra 0
\end{equation}
over $\P^{n-1}$ provides a canonical compatible pair of real square root for $K_{\P^{n-1}}$ 
and spin structure for $\R\P^{n-1}$, whenever $n\!=\!4m$; see Section~\ref{orient_subs}.
For $l,t_1,\ldots,t_l\!\in\Z^+$, we can then define 
\begin{equation}\label{numsPndfn_e}
N_d^{\phi}(t_1,\ldots,t_l)=\int_{\ov{\mc{M}}_l(\P^{4m-1},d)^{\phi}} \ev_1^*H^{t_1}\wedge\ldots\wedge\ev_l^*H^{t_l},
\end{equation}
where $H\in H^2(\P^{4m-1},\Z)$ is the hyperplane class.

\begin{theorem}\label{equalGW}
For all $m,d,l,t_1,\ldots,t_l\in\Z^+$, 
\begin{equation}\label{numsPn_e}N_d^{\eta_{4m-1}}(t_1,\ldots,t_l)=-N_d^{\tau_{4m-1}}(t_1,\ldots,t_l).
\end{equation}
Furthermore, these invariants vanish if $d\!\in\!2\Z$ or $t_k\!\in\!2\Z$ for some~$k$.
\end{theorem}

We prove this theorem in Section~\ref{applic_subs} using 
the equivariant localization theorem of~\cite{ABo}. 
While $(K_{\P^{4m+1}},\eta_{4m+1})$ does not admit a real square root
and $\Fix(\tau_{4m+1})$ does not admit a spin structure,
we show that  Theorem~\ref{equalGW} and its proof extend to~$\P^{4m+1}$ 
with the orientations on the moduli spaces explicitly constructed 
in Section~\ref{AGorient_subs}; see Remark~\ref{comporient_rmk}.

If $d$ is odd,
\begin{equation}\label{cmeq_e}\begin{split}
\ov{\mc{M}}_l(\P^{4m-1},d)^{\tau_{4m-1}}&=\ov{\mc{M}}_l(\P^{4m-1},d)^{\tau_{4m-1},\tau},\\
\ov{\mc{M}}_l(\P^{4m-1},d)^{\eta_{4m-1}}&=-\ov{\mc{M}}_l(\P^{4m-1},d)^{\eta_{4m-1},\eta},
\end{split}
\end{equation}
by the first statement in~(\ref{degPn_e}) and by~(\ref{etaPn_e}), respectively.
The sign in~(\ref{numsPn_e}) and~(\ref{cmeq_e}) occurs because we reverse the orientation of
 $\ov{\mc{M}}_l(X,A)^{\phi,\eta}$
when gluing it to $\ov{\mc{M}}_l(X,A)^{\phi,\tau}$ 
in order to make the glued moduli space oriented. 
In fact, the canonical square root and spin structure described in Section~\ref{orient_subs}
give the same orientation on $\Fix(\tau_{4m-1})$. 
Therefore, they are not compatible 
in the sense of Definition~\ref{dfn:comp-spin-sqroot} and one of the orientations has to be flipped.
As described in Sections~\ref{fixedloci_subs} and~\ref{NormBndl_subs}, 
the torus fixed loci in
$$\ov{\mc{M}}_l(\P^{4m-1},d)^{\tau_{4m-1},\tau}
\qquad\hbox{and}\qquad \ov{\mc{M}}_l(\P^{4m-1},d)^{\eta_{4m-1},\eta},$$ 
their normal bundles, and the corresponding restrictions of the cohomology classes being integrated
are the same; this confirms~(\ref{numsPn_e}) for $d$ odd. 

If $d$ is even, $N_d^{\eta_{4m-1}}(t_1,\ldots,t_l)=0$ by (\ref{etaPn_e}) and
the second statement in~(\ref{degPn_e}). 
On the other hand, in this case, the fixed loci~in 
\begin{equation}\label{mcMtau_e}
\ov{\mc{M}}_l(\P^{4m-1},d)^{\tau_{4m-1},\tau}
\quad\hbox{and}\quad 
\ov{\mc{M}}_l(\P^{4m-1},d)^{\tau_{4m-1},\eta},
\end{equation}
their normal bundles, and the corresponding restrictions of the cohomology classes being integrated
are the same.
Since the canonical orientation on the second space in~(\ref{mcMtau_e})
gets flipped when it is glued to the first, 
the contributions to $N_d^{\tau_{4m-1}}(t_1,\ldots,t_l)$ from the fixed loci cancel in pairs. 
This confirms (\ref{numsPn_e}) for $d$ even and establishes Theorem~\ref{equalGW} whenever~$2|d$. 

Whether $d$ is odd or even, if $2|t_k$, the contributions to $N_d^{\phi}(t_1,\ldots,t_l)$ 
from the fixed loci in $\ov{\mc{M}}_l(\P^{4m-1},d)^{\phi,c}$, for $c\!=\!\tau,\eta$ fixed, 
also cancel in pairs. 
This establishes the remaining vanishing statement of Theorem~\ref{equalGW}.

In Example~\ref{d1_eg}, we show~that
\begin{equation}\label{d1N_e1}
 N_1^{\tau_{4m-1}}(t_1,\ldots,t_l)=1
\end{equation}
whenever
\begin{equation}\label{d1N_e2}
t_1,\ldots,t_l\!\in\!\Z^+\!-\!2\Z \quad\hbox{and}\quad
t_1\!+\!\ldots\!+\!t_l=4m\!-\!2\!+\!l.
\end{equation}
In particular, the signed number of real lines passing through a single non-real point in 
$\P^{4m-1}$ with the standard conjugation is $+1$ 
with respect to the canonical spin structure of Section~\ref{orient_subs}. 
In Example~\ref{d3_eg}, we show~that 
\begin{equation}\label{d3N_e1}
N_3^{\tau_{4m-1}}(t_1,t_2,4m\!-\!1)=-1
\end{equation}
whenever
\begin{equation}\label{d3N_e2}
t_1,t_2\!\in\!\Z^+\!-\!2\Z, \quad t_1,t_2\!\ge\!3, \quad\hbox{and}\quad
t_1+t_2=4m\!+\!2.
\end{equation}
A similar computation shows that
$$N_5^{\tau_3}(3,3,3,3,3)=5.$$ 

\begin{remark}\label{P4m1_rmk}
All moduli spaces $\ov{\mc{M}}_l(\P^{2m-1},d)^{\phi,c}$ are given explicit orientations in
Section~\ref{AGorient_subs}.
In the $c\!=\!\tau$ case, the orientation turns out to come from a relative spin structure 
on~$\R\P^{2m-1}$ and Proposition~\ref{flip-orientation} still applies.
We show directly that so does Proposition~\ref{matching-orientation_prp}; 
see Proposition~\ref{algbnd_prp}.
Thus, we can also define the numbers $N_d^{\phi}(t_1,\ldots,t_l)$
as in~(\ref{numsPndfn_e})
using the algebraic orientations of Section~\ref{AGorient_subs}.
They can be computed using the equivariant localization data of 
Sections~\ref{fixedloci_subs} and~\ref{NormBndl_subs} with only minor changes; 
see Remark~\ref{comporient_rmk}.
The conclusions~of Theorem~\ref{equalGW} still apply.
The conclusions of~(\ref{d1N_e1}) and~(\ref{d3N_e1}) apply to the algebraic orientations 
on the moduli spaces for~$\P^{4m+1}$, 
which agree with the orientations by a canonical relative spin structure for $d$ odd;
see Remarks~\ref{comporient_rmk2a}, \ref{comporient_rmk2b}, and~\ref{comporient_rmk}.
\end{remark}

\subsection{Outline and acknowledgments}

In Section~\ref{ch:realcurves}, 
we investigate the boundary and orientation problems for  moduli spaces of real curves 
without fixed point and define $\eta$-invariants.
In Section~\ref{ch:mixed-invariants}, we combine the orientation problem of 
$\ov{\mc{M}}(X,A)^{\phi,\tau}$ and $\ov{\mc{M}}(X,A)^{\phi,\eta}$ 
and finish the proof of Theorem~\ref{thm:main}. 
Theorem~\ref{zeroness} is proved in Section~\ref{sec:zeroes}.
In Section~\ref{orientPn_sec}, we study the moduli spaces of 
real maps $\P^1\!\lra\!\P^{2m-1}$ in detail.
We provide equivariant localization data for them and 
establish Theorem~\ref{equalGW}, (\ref{d1N_e1}), and~(\ref{d3N_e1})
in Section~\ref{EquivLocal_sec}.

I would like to thank Professor G.~Tian, for his continuous encouragement and support 
and for sharing his inspiring insights, and
A.~Zinger, for his patience and help with the exposition of this paper. 
I am also grateful to P.~Georgieva and J.~Solomon for many helpful discussions. 
Finally, I would like to thank the referee for many valuable comments and suggestions.

\part{Construction of genus zero real GW invariants}
\section{Moduli spaces of real curves without fixed points\label{ch:realcurves}}

In this section, we study the moduli space of real curves of genus 0 without real points. 
As before, let 
\begin{equation}\label{equ:eta}
\eta,\tau\!:\P^1\to\P^1, \qquad \eta(z)=\frac{-1}{\bar{z}}, \quad \tau(z)=\frac1{\bar{z}}.
\end{equation}
Denote by $G_{\eta}$ the set of Mobius transformations (automorphisms of~$\P^1$),
$\rho(z)=\displaystyle{\frac{az+b}{cz+d}}$, commuting with~$\eta$.  
It acts freely and transitively on the sphere bundle $S(T\P^1)$ of $T\P^1$.
Since $S(T\P^1)\cong \R\P^3$, $G_{\eta}$ is a compact orientable Lie group.
Furthermore, the induced orientation on $S(T\P^1)$ as the boundary of the unit disk bundle $D(T\P^1)$ 
with its complex orientation, with the convention as in (\ref{equ:indori}), 
induces a canonical orientation on~$G_{\eta}$. 
For this orientation on $G_\eta$, $v_1,v_2,v_3\in T_{\tn{id}} G_\eta$, where 
$$
\aligned
&v_1+\mfi v_2= \frac{\nd}{\nd a}\!\mid_{a=0}\frac{z+a}{1-\ov{a}z},\\
& v_3=\frac{\nd}{\nd \theta}\mid_{\theta=0} \tn{e}^{\mfi \theta}z,
\endaligned
$$
is an oriented basis.

Similarly, let $G_{\tau}$ be the set of Mobius transformations commuting with~$\tau$.  
The automorphism group $G_\tau$ has two connected components, $G_\tau^0$ containing the identity and $\rho\cdot G_\tau^0$ where $\rho(z)=z^{-1}$. 
The former is the automorphism group of a disk and the latter switches the two disk components of $\P^1\setminus \Fix(\tau)$.
Fixing one of the disk components $D$ as the reference disk, $G_\tau^0$ acts freely and transitively on the sphere bundle $S(TD)$ of $TD$.
Since $S(TD)\cong D\times S^1$, $G_{\tau}^0$ inherits an induced orientation.
Let $D$ be the choice containing $z\!=\!0\! \in \P^1$, then $u_1,u_2,u_3\in T_{\tn{id}} G_\tau^0$, where 
$$
\aligned
&u_1+\mfi u_2= \frac{\nd}{\nd a}\!\mid_{a=0}\frac{z+a}{1+\ov{a}z},\\
& u_3=\frac{\nd}{\nd \theta}\mid_{\theta=0} \tn{e}^{\mfi \theta}z,
\endaligned
$$
is an oriented basis.
We use these conventions in orienting the corresponding moduli spaces 
and in the proof of Theorem~\ref{square-root-orientation-thm} and Theorem~\ref{thm:main}.
  
The involution $\phi$ on $X$ induces an involution $\tilde\phi$ on the moduli space
$\mc{M}_{2l}(X,A)$ of all degree~$A$ $2l$-marked somewhere injective $J$-holomorphic spheres:
$$\tilde\phi\big([u,z_1,z_2,\ldots,z_{2l-1},z_{2l}]\big)
=[\phi\circ u\circ\eta,\eta(z_2),\eta(z_1),\ldots,\eta(z_{2n}),\eta(z_{2l-1})].$$
For every $J$-holomorphic sphere $u\!:\P^1\to X$ in the fixed point locus of $\tilde\phi$, there exists at most one anti-holomorphic involution $\eta_{u}$ such that $\Fix(\eta_{u})=\emptyset$ and 
$u=\phi\circ u\circ\eta_{u}$; therefore, the fixed point locus of $\tilde\phi$ contains $\mc{M}_l(X,A)^{\phi,\eta}$. Intuitively, $\mc{M}_l(X,A)^{\phi,\eta}$ has half the dimension of $\mc{M}_{2l}(X,A)$.

\begin{remark}
If $X\cong \P^1$, $\phi=\tau$, $A=[1]\in H_2(\P^1)\cong \Z$, and $l=0$, $\mc{M}_0(\P^1,[1])$ is just one point on which $\tilde\phi$ acts as identity while $\mc{M}_0(\P^1,[1])^{\tau,\eta}$ is empty; therefore, $\Fix(\tilde\phi)\neq \mc{M}_0(\P^1,[1])^{\tau,\eta}$. 
\end{remark}

Let $\ov{\mc{M}}_l(X,A)^{\phi,\eta}$ denote the stable map compactification of $\mc{M}_l(X,A)^{\phi,\eta}$. 
This is a closed subset of $\ov{\mc{M}}_{2l}(X,A)$ consisting of maps $[u,\Si,z_1,\ldots,z_{2l}]$
with the property that there exists an anti-holomorphic involution $\eta_u$ 
on the domain $\Si$ of $u$ such that 
$$|\Fix(\eta_u)|\le1, \qquad u=\phi\circ u\circ\eta_u,\qquad
\eta_u(z_2)=z_1,~\ldots,~\eta_u(z_{2l})=z_{2l-1}.$$
Thus, there are two possible cases for $\eta_u\!:\Si\to\Si$:
\begin{enumerate}
\item $\Si=\Si_0\cup\bigcup_i(\Si_i\sqcup\Si_{\bar{i}})$, $\eta_u\!:\Si_0\to\Si_0$ is an anti-holomorphic involution without fixed points, and $\eta_u\!:\Si_i\to\Si_{\bar{i}}$ is an anti-holomorphic map
with inverse \hbox{$\eta_u\!:\Si_{\bar{i}}\to\Si_i$};
\item $\Si=\bigcup_i(\Si_i\cup\Si_{\bar{i}})$, $\eta_u\!:\Si_i\to\Si_{\bar{i}}$ is an 
anti-holomorphic map with inverse \hbox{$\eta_u\!:\Si_{\bar{i}}\to\Si_i$}.
\end{enumerate}
In the second case, $\eta_u$ fixes a node of $\Si$,
which must be mapped by $u$ to $\Fix(\phi)$; $\ov{\mc{M}}_l(X,A)$ contains no such
elements if $\Fix(\phi)=\emptyset$.

The virtual codimension of a boundary stratum of $\ov{\mc{M}}_l(X,A)^{\phi,\eta}$ is 
the number of nodes in the domains of the elements of the stratum.
If $\Fix(\phi)=\emptyset$, $\ov{\mc{M}}_l(X,A)^{\phi,\eta}$ contains no elements of the second type above,
and so its boundary strata have codimension at least two.
Thus, $\ov{\mc{M}}_l(X,A)^{\phi,\eta}$ is a moduli space without (virtual) codimension one boundary
if $\Fix(\phi)=\emptyset$, and there is a hope of defining GW-type invariants directly from $\ov{\mc{M}}_l(X,A)^{\phi,\eta}$.

We study the orientation problem for $\ov{\mc{M}}_l(X,A)^{\phi,\eta}$ in Section~\ref{orientation} and describe a Kuranishi structure in Section~\ref{kuranishi-structure}.

\subsection{Orientation}\label{orientation}

Let $c\!=\!\tau,\eta$. 
In the orientation problem for $\mc{M}_l(X,A)^{\phi,c}$, 
it is sufficient to consider the case $l=0$ because any pair of marked points 
$(z_i,\ov{z_i})$ increases the tangent space by $T_{z_i}\P^1$, which has a canonical orientation. 
Denote by $\mc{P}_0(X,A)^{\phi,c}$  the space of (parametrized) degree~$A$ $J$-holomorphic maps
$u\!:\P^1\!\ra\!X$  such that $\phi\!\circ\!u\!=\!u\!\circ\!c$.
The group~$G_c$ acts~on this space~by
$$G_c\times \mc{P}_0(X,A)^{\phi,c}\lra \mc{P}_0(X,A)^{\phi,c},
\qquad g\cdot u=u\circ g^{-1}\,.$$
By definition,
$$\mc{M}_0(X,A)^{\phi,c} = \mc{P}_0(X,A)^{\phi,c} /G_c .$$
For example, $\mc{P}_0(\P^1,1)^{c,c}\!=\!G_c$ and $\mc{M}_0(\P^1,1)^{c,c}$ consists of 
a single point.
The next observation is used in Section~\ref{orientPn_sec}.

\begin{lemma}\label{ptorient_lmm}
Let $c\!=\!\tau,\eta$. 
If $\mc{P}_0(\P^1,1)^{c,c}\!=\!G_c$ is oriented with the canonical orientation of~$G_c$
as at the beginning of Section~\ref{ch:realcurves}, 
then $\mc{M}_0(\P^1,1)^{c,c}$  is a single negative point. 
\end{lemma}

\begin{proof}
The group action in this case is given~by
$$G_c\times G_c\lra G_c, \qquad g\cdot h\lra h\circ g^{-1}\,,$$
The claim is thus equivalent to the statement that the differential of the map
$$G_c\lra G_c, \qquad g\lra g^{-1},$$
is orientation-reversing at the identity. 
This differential is the multiplication by $-1$.
Since the dimension of~$G_c$ is odd, it is orientation-reversing.
\end{proof}

The orientation problem for $\mc{M}_l(X,A)^{\phi,\tau}$ has a long history.
Below we focus on the orientation problem for $\mc{M}_l(X,A)^{\phi,\eta}$.
In contrast to the group~$G_{\tau}$, the group $G_{\eta}$ is connected.
In order to put an orientation on $\mc{M}_0(X,A)^{\phi,\eta}$, 
it is thus enough to orient $\mc{P}_0(X,A)^{\phi,\eta}$. 
For this, we need to orient the determinant of the index bundle
$$ \tn{det}_\R(E)\equiv \La^{\top}H^0(E)_{\R} \otimes \La^{\top} (H^1(E)_{\R})^*, $$
where $E= u^*TX$ and $H^0(E)_{\R}$ and $H^1(E)_{\R}$ are the real elements of the kernel and cokernel of a Cauchy-Riemann operator on $E$. 
Recall that $E$ admits an anti-complex linear involution $T_{\phi}$; see the left diagram in (\ref{diag:E}).

\begin{definition}\label{admissible}
Let $E\to \P^1$ be a complex vector bundle with a real structure $\phi$ covering $\eta$.
We call a trivialization of $E$ over $\C^*\subset \P^1$,
$$\xymatrix{
& E \ar[d]_{\pi} \ar[r]^{\psi} & \C^*\times \C^m \ar[d]_{\pi} \\
& \C^* \ar[r]^{\id} & \C^*}$$
\textsf{admissible} if the involution $\phi_{\psi}(z) = \psi_{\eta(z)} \circ \phi \circ \psi_{z}^{-1}$ coincides with the standard involution $C\colon (z,v) \ra (\eta(z), \bar{v})$. Admissible trivializations $\psi$ and $\psi'$ of $(E,\phi)$ over $\C^*$ are called \textsf{homotopic} if there is a family of such trivializations $\psi_t$, $t\in [0,1]$, such that $\psi_0=\psi$ and $\psi_1=\psi'$.
\end{definition}

\begin{lemma}\label{homotopy-classes}
For every complex vector bundle $E\to \P^1$ with a real structure $\phi$ covering $\eta$, there are two homotopy classes of admissible trivializations over $\C^*\subset \P^1$. Moreover, for every admissible trivialization $\psi$ and every map 
$$ R_{(\ep_i)}\colon \C^* \times \C^m \to \C^* \times \C^m ,\quad R_{(\ep_i)}(z,v)=(\ep_1 v_1 ,\cdots,\ep_m v_m), \quad \ep_{i}=\pm1 ,$$
$R_{(\ep_i)}\circ \psi$ is another admissible trivialization which is in the same homotopy class as $\psi$ if and only if $\prod \ep_i = 1$.
\end{lemma}

\begin{proof}
(1) As a complex vector bundle, $E$ is trivial over $\C^*$. Therefore, we can fix a trivialization $\psi\colon E \to \C^* \times \C^m$. The involution $\phi$ then corresponds to a map
$$\phi_{\psi}\colon \C^* \to \tn{GL}(2m,\R) $$
whose image lies in the set of anti-complex linear matrices. 
In order to obtain an admissible trivialization, we find a change of trivialization matrix 
\begin{equation}\label{previous1}
A\colon \C^* \to \tn{GL}(m,\C)\qquad \hbox{s.t.}\quad A_{\eta(z)} \circ \phi_{\psi} \circ A_{z}^{-1} = C. 
\end{equation}

Let $B_{\psi}(z)= C \circ \phi_{\psi}(z) \in \tn{GL}(m,\C)$. Since $\phi_{\psi}$ is an involution, $\ov{B_{\psi}(\eta(z))}B_{\psi}(z)=\I_{m}$. Composing on the left by $C$, we can rewrite (\ref{previous1}) as
\begin{equation}\label{changeoftrivialization}
 \ov{A_{\eta(z)}} \circ B_{\psi} \circ A_{z}^{-1} = \I_{m}. 
 \end{equation}
Let $\alpha\colon\H\setminus \{0\} \to \tn{GL}(m,\C)$, where $\H$ is the closed upper half-plane, be a family of matrices such that
$$\alpha(r)=
\left\{
	\begin{array}{ll}
		 \I_m   & \mbox{if } r \in \R^+; \\
	   \ov{B_{\psi}(\eta(r))} & \mbox{if } r\in \R^-.
	\end{array}
	\right.$$
Next define 
$$ A(z) =
\left\{
	\begin{array}{ll}
		\alpha(z) B_{\psi}(z)  & \mbox{if } z \in \H\setminus \{0\}; \\
		\ov{\alpha(\eta(z))} & \mbox{if } z \in \ov{\H}\setminus\{0\}.
	\end{array}
	\right.$$ 
It is easy to check that $A$ is continuous and satisfies (\ref{changeoftrivialization}).\\

(2) If $\psi$ is an admissible trivialization, 
any other admissible trivialization is of the form $\rho\circ \psi$, where
\begin{equation}\label{change-of-trivialization}
\rho\colon\C^*\to\tn{GL}(m,\C) \qquad\hbox{and}\quad \ov{\rho(\eta(z))} \rho(z)^{-1}=\I_m.
\end{equation}
The question is whether $\rho$ is homotopic to identity through a 
family $\rho_t$ of matrices satisfying the same equation as (\ref{change-of-trivialization}). 

Let
$$G= \left\{\gamma\colon[0,1]\to \tn{GL}(m,\C) \mid \gamma(0) = \ov{\gamma(1)} \right\}, \quad G_0= \{\gamma \in G : \gamma(0)=\I_m\};$$
the set $G$ is a group under point-wise multiplication, while $G_0$ is its subgroup. 
The restriction of $\rho$ to the upper semi-circle, $\left\{z=e^{i\pi t} \;\mid t \in [0,1]\right\}$, determines an element of $G$. 
In fact, the space of $\rho$ satisfying (\ref{change-of-trivialization}) is homotopic to $G$. 
The map
\begin{equation}\label{pi-projection}
 \pi\colon G \to \tn{GL}(m,\C),\quad \gamma \to \gamma(0), 
 \end{equation}
is a fiber bundle with fiber $G_0$. 
From the associated long exact sequence,
$$ \cdots \ra \pi_1(\tn{GL}(m,\C)) \ra \pi_0(G_0) \ra \pi_0(G) \ra \pi_0(\tn{GL}(m,\C)) \ra 0, $$
we conclude that $\pi_0(G) = \Z/2\Z$.
In fact, the connecting homomorphism 
$$
\pi_1(\tn{GL}(m,\C))\ra \pi_0(G_0)\cong\pi_1(\tn{GL}(m,\C))\cong\Z
$$
is multiplication by $2$ for the following reason. 

\vskip.1in
We start from the loop $\gamma:[0,1]\to \tn{GL}(m,\C)$ given by
$$
\gamma(s)=\left( \begin{array}{cccc}
e^{2\pi\mathfrak{i}s} & 0 & 0 & 0 \\
0 & 1 & 0 & 0 \\
0 & 0 & \ddots &0 \\
0 & 0 & 0 & 1  \end{array} \right).
$$
This loop generates $\pi_1(\tn{GL}(m,\C))$. 
With the projection map $\pi$ as in (\ref{pi-projection}),  
the restricted $S^1$-family $\pi^{-1}(\gamma)\subset G$ is a non-trivial $G_0$-bundle. 
For every $s\!\in\![0,1]$, let $\alpha_s\!\in\!\pi^{-1}(\gamma(s))\subset G$ be the path of matrices
$$
\alpha_s(t)=\left( \begin{array}{cccc}
e^{2\pi\mathfrak{i} (s(1-2t)) } & 0 & 0 & 0 \\
0 & 1 & 0 & 0 \\
0 & 0 & \ddots &0 \\
0 & 0 & 0 & 1  \end{array} \right),  \quad t\in[0,1].
$$
Note that $\alpha_0\equiv \tn{id}$ and $\alpha_1\cong \gamma^{-2}$.
Then 
\begin{equation}\label{identification}
\pi^{-1}(\gamma(s))= \alpha_s \cdot G_0 =\{ \alpha_s \delta\mid \delta \in G_0\}.
\end{equation}
Moving along the family of identifications (\ref{identification}) over $[0,1]$, 
we find that the holonomy map of $\pi^{-1}(\gamma)$ is isomorphic to
$$
h\colon G_0 \to G_0, \quad h(\de)= \alpha_1^{-2} \de =\gamma^2 \de.
$$
In other words,
\begin{equation}\label{holonomy}
\pi^{-1}(\gamma)\cong G_0\times [0,1]/ \de\times \{0\}\sim (\gamma^2\cdot \de)\times \{1\}.
\end{equation}
This implies that the connecting homomorphism takes $\gamma\in \pi_1(\tn{GL}(m,\Z))$ to 
$$\gamma^{2}\in \pi_0(G_0)\cong \pi_1(\tn{GL}(m,\Z)).$$ 

The remaining claim of the lemma is checked by chasing the maps in the long exact sequence. 
For $\rho=R_{(\ep_i)}$, the corresponding path $\gamma$ in $G$ is 
the constant path $\gamma(t)\equiv \textnormal{diag}(\ep_i)$. 
Inside $G$, via the path 
$$
\gamma_s(t)=\tn{diag}(\tn{e}^{(1-\ep_i)\pi \mfi f_s(t)}), \quad t\in[0,1],~s\in[0,1],
$$
where $f_s(t)= -st+\frac{1+s}{2}$, we can deform $\gamma=\gamma_0$ to the path $\gamma_1\!\in\!G_0$
given~by
$$\gamma_1(t)=\tn{diag}(\tn{e}^{(1-\ep_i)\pi \mfi (1-t)}).$$ 
Then $[\gamma_1]\in \pi_0(G_0)=\pi_1(G_1)\cong \Z$ is equal to $n_\gamma=-|\{i: \ep_i=-1\}|$. 
Since the homomorphism
$$\pi_1(\tn{GL}(m,\C))\ra \pi_0(G_0)\cong\pi_1(\tn{GL}(m,\C))\cong\Z$$ 
is multiplication by $-2$, $\gamma\in\pi_0(G)$ is trivial if and only if 
$2|n_\gamma$, i.e. if and only if $\prod \ep_i =1$.
\end{proof}

\begin{lemma}\label{trivialization-to-orientation}
Let $E\to \P^1$ be a complex vector bundle with a real structure $\phi$ lifting~$\eta$. Every admissible trivialization of $(E,\phi)$ over $\C^* \subset \P^1$ canonically determines  an orientation of  $ \La^{\top} H^0(E)_{\R} \otimes \La^{\top} (H^1(E)_{\R})^* $. The two orientations given by two different admissible trivializations coincide if and only if they are in the same homotopy class. 
\end{lemma}

\begin{proof}
The proof is analogous to that of \cite[Proposition 8.1.4]{FOOO}. 
Contracting each of the two circles 
$$
C_{0,r}=\left\{z\in \C^* \mid \left|z\right|=r\right\}\quad \tn{and}  \quad C_{\infty,r}= \left\{z\in \C^* \mid \left|z\right|=\frac{1}{r}\right\}
$$
to a point, we obtain a nodal curve $\Si= \Si_{\top} \cup \Si_0 \cup \Si_{\down}$ (Figure~\ref{Fig:pinching}) with an induced fixed point free involution $\eta_{\Si}$. We denote the quotient map by $ \pi\colon \P^1 \to \Si$. Denote by $q$ and $\eta_{\Si}(q)$ the nodal points of $\Si$. We may assume that $q$ and $\eta_{\Si}(q)$ are respectively $0$ and $\infty$ in $\Si_0 \cong \P^1$.

\begin{figure}[h]
\begin{pspicture}(-3,-1.5)(11,2.4)
\psset{unit=.3cm}
\pscircle*(16,0){.3}\rput(16,-2){$q$}\pscircle*(8,0){.3}
\pscircle[linewidth=0.07](12,0){4}\psellipse[linewidth=0.07](12,0)(1,4)\rput(12,-5){$\Si_0$}
\pscircle[linewidth=0.07](20,0){4}\psellipse[linewidth=0.07](20,0)(1,4)\rput(20,-5){$\Si_\top$}
\pscircle[linewidth=0.07](4,0){4}\psellipse[linewidth=0.07](4,0)(1,4)\rput(4,-5){$\Si_\bot$}
\psline[linewidth=0.07]{<->}(7,5)(17,5)\rput(12,6){conjugation}

\end{pspicture}
\caption{Nodal curve $\Si$ obtained by pinching $C_{0,r}$ and $C_{\infty,r}$}\label{Fig:pinching}
\end{figure}
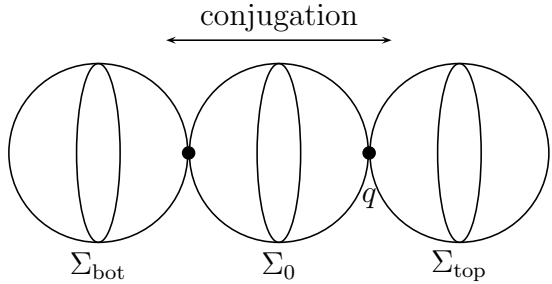

Via the given trivialization, the bundle $(E,\phi)$ descends to a bundle $(\tilde{E}, \tilde{\phi})$ over $\Si$ so that
$$ \tilde{E}\mid_{\Si_0} \cong \P^1 \times \C^m$$
and the involution $\tilde{\phi}\!\mid_{\Si_0}$ sends $(z,v)$ to $(\eta_{\Si}(z),\bar{v})$. Over $\Si_{\top} \cup \Si_{\down}$, $\tilde{\phi}$ is an anti-complex linear map of the form 
$$ \tilde{\phi} \colon \tilde{E}\!\mid_{\Si_{\top}} \to \tilde{E}\!\mid_{\Si_{\down}}.$$
A section of $(\tilde{E},\tilde{\phi})$ is of the form $\xi=(\xi_{\top},\xi_0,\xi_{\down})$, with matching conditions at the nodes. A section $\xi$ is real if and only if
$$ \xi_{\down}(\eta_{\Si}(z)) = \tilde{\phi} (\xi_{\top}(z)), \forall z\in\Si_{\top} \quad \tn{and}\quad \xi_0 \in \Gamma(\tilde{E}\mid_{\Si_0})_{\R}. $$
Therefore, it is determined by an arbitrary section of $\tilde{E}\!\mid_{\Si_{\top}}$ and a real section of $\tilde{E}\!\mid_{\Si_0}$ which match at $q$.

The matching condition at the nodes gives a short exact sequence
$$ 0 \to W^{1,p}(\tilde{E})_{\R} \to W^{1,p}(\tilde{E}\mid_{\Si_{\top}}) \oplus W^{1,p}(\tilde{E}\mid_{\Si_0})_{\R} \to \C^m_q \to 0.
$$
The associated determinant of the pair $(\tilde{E},\tilde{\phi})$ is given by
\begin{equation}\label{indexbundle} 
\tn{det}_{\R} (\tilde{E}) \cong \tn{det}_{\C} (\tilde{E}\mid_{\Si_{\top}}) \otimes \tn{det}_{\R}(\tilde{E}\mid_{\Si_0})  \otimes \tn{det}_{\C}(\C^m_{q})^*.
\end{equation}
Over $\Si_0$, the determinant bundle is canonically isomorphic (after deforming the Cauchy-Riemann operator) to 
$$ \La^{\top} H^0(\P^1 \times \C^m)_\R =\La^{\top}_{\R} \R^m \subset \La^{\top}_{\C} \C^m.$$
It inherits an orientation from the choice of trivialization. 
Since $\tn{det}_{\C} (\tilde{E}\mid_{\Si_{\top}})$ and $\tn{det}_{\C}(\C^m_q)^*$ carry orientations induced by their complex structures, they are canonically oriented. Thus, (\ref{indexbundle}) induces an orientation on $\tn{det}_{\R} (\tilde{E})$.
\end{proof}

\begin{proof}[Proof of Theorem \ref{square-root-orientation-thm}] 
By Lemma \ref{trivialization-to-orientation}, a systematic way of trivializing $u^*TX$ over $\C^*\subset \P^1$ would orient $\mc{P}_{0}(X,A)^{\phi,\eta}$. Let $K_X= \Lambda_{\C}^{\top}T^*X $ be the canonical complex line bundle over $X$. It inherits an involution $K_{\phi}\colon K_X \to K_X$ (covering $\phi$) from $T_\phi$. Therefore, it is a complex line bundle with an involution. Any admissible trivialization of $u^*TX\!\mid_{\C^*}$ canonically induces an admissible trivialization of $u^*K_X\!\mid_{\C^*}$ and changing the homotopy class of admissible trivialization of the former changes the homotopy class of the induced admissible trivialization. We can therefore reduce the orientation problem to the problem of finding a canonical way of admissibly trivializing $u^*K_X$. This is an easier problem because $K_X$ is just a line bundle and has less structure than~$TX$. 

Let $(\cL,\phi_\cL) \ra (X,\phi)$ be any complex line bundle over $X$ with an anti-complex linear involution $\phi_\cL$ covering $\phi$. The line bundle $\cL^{\otimes 2}$ inherits an involution from the one on $\cL$ by 
$$ 
\phi_{\cL^{\otimes 2}} (v_1 \otimes v_2 ) = \phi_{\cL}(v_1) \otimes \phi_{\cL}(v_2). 
$$
Every admissible trivialization of $u^*\cL\mid_{\C^*}$ induces an admissible trivialization of $u^*\cL^{\otimes 2}\mid_{\C^*}$. However, changing the homotopy class of trivialization of $L$ does not change the homotopy class of the induced trivialization on $\cL^{\otimes 2}$, since changing the trivialization of $\cL$ by the complex linear map $R_{-1}$ of Lemma \ref{homotopy-classes} changes the homotopy class of admissible trivialization of $\cL^{\otimes 2}$ by $R_{-1} \otimes R_{-1}= \id$. Thus, for the complex line bundle $(\cL^{\otimes 2}, \phi_{\cL}\otimes_{\C} \phi_\cL)$ as above $u^*\cL^{\otimes 2}$ has a canonical admissible trivialization.

We conclude that given a choice of real square root $(K_X,K_{\phi})\cong (\cL^{\otimes 2}, \phi_{\cL}\otimes_{\C} \phi_\cL)$ for $K_X$,  it provides a choice of admissible trivialization for every $u^*TX|_{\C^*}$, hence an orientation on $\mc{P}_{l}(X,A)^{\phi,\eta}$. Finally, together with the choice of orientation on $TG_\eta$ given in Section~\ref{ch:realcurves}, we obtain an orientation on $\mc{M}_{l}(X,A)^{\phi,\eta}$ such that
$$
T_u\mc{P}_{l}(X,A)^{\phi,\eta}\cong T_f\mc{M}_{l}(X,A)^{\phi,\eta}\oplus T_{\tn{id}}G_\eta
$$ 
is an oriented isomorphism of vector spaces; see Lemma~\ref{new-lemma} for the negative sign.
\end{proof}

\subsection{Complimentary remarks and examples}\label{sec:remarks}

Proposition~\ref{prop:existence}, which we prove below, provides examples of symplectic manifolds 
with the canonical bundle admitting a real square root.

\begin{lemma}\label{lem:lift1}
Let $\cL$ be a holomorphic line bundle over a compact K\"ahler manifold $X$ with an anti-holomorphic involution $\phi$.
Up to multiplication by a constant number in $U(1)\subset \C^*$,
$\cL$ admits at most one anti-holomorphic conjugation lifting $\tilde\phi$ of $\phi$.
\end{lemma}
\begin{proof}
Assuming the existence, let $\tilde\phi_1$ and $\tilde\phi_2$ be two anti-holomorphic conjugation lifts of $\phi$.
Then 
$$\tilde\phi_2=\rho \circ \tilde\phi_1,$$ 
for some holomorphic automorphism $\rho\colon X \to \C^*$. 
Since $X$ is compact $\rho \equiv e^{\mathfrak{i} \theta}$ is constant.
\end{proof}

\begin{lemma}\label{lem:lift2}
Let $\cL$ be a complex line bundle over a symplectic manifold $X$ with an anti-symplectic involution $\phi$.
Assuming $H^1(X,\R)=0$, every two anti-complex linear conjugation lifts $\tilde\phi$ of $\phi$ are equivariantly isomorphic.
\end{lemma}
\begin{proof}
Assuming the existence, as in the proof of Lemma~\ref{lem:lift1}, let $\rho\colon X \to \C^*$ be the resulting function. 
From $\phi_2^2=\tn{id}$ we conclude that 
\begin{equation}\label{equ:rho}
\rho(\phi(x))\ov{\rho(x)}=\tn{id}.
\end{equation} 
Since $H^1(X,\R)=0$, for every loop $\gamma \in \pi_1(X)$ $\int_\gamma \rho^*\tn{d} \theta =0$; therefore, $\tn{image}(\rho(\gamma))\subset \C^*$ is contractible. 
Thus, there is a well-defined square root
$$ \sqrt{\rho}\colon X \to \C^*,\quad \sqrt{\rho}^2=\rho.$$
From Equation~\ref{equ:rho} together with the identity
$$ \sqrt{\rho}(\phi(x)) \circ \phi_1(x) \circ \sqrt{\rho}(x)= \sqrt{\rho}(\phi(x)) \circ \ov{\sqrt{\rho}(x)} \circ \phi_1(x) $$ 
we conclude that $\psi=\sqrt{\rho}(\phi(x)) \circ \phi_1(x) \circ \sqrt{\rho}(x)$ is an anti-complex linear involution isomorphic to either 
$\phi_2$ or $-\phi_2$. 
If the former happens, we conclude that $(\cL,\phi_2)$ is equivariantly isomorphic to $(\cL,\phi_1)$; otherwise, changing $\sqrt{\rho}$ with $\mathfrak{i}\sqrt{\rho}$ we obtain the desired isomorphism.
\end{proof}

\begin{proof}[Proof of Proposition~\ref{prop:existence}]
If $c_1(TX)=4\alpha$ for some $\alpha$ as in the statement of the proposition, complex line bundle $\cL$ with the Chern class $2\alpha$ has a real structure given by the isomorphism 
$$\cL= \cL' \otimes \ov{\phi^*(\cL')},$$
where $\cL'$ is a complex line bundle with the Chern class $\alpha$.
From the isomorphism of complex line bundles
$$K_X \cong \cL^* \otimes \cL^*$$
and Lemma~\ref{lem:lift2} we conclude that the canonical real structure and the one induced by the above isomorphism on $K_X$ are equivariantly isomorphic; thus, $(K_X,\phi_{K_X})$ admits a real square root.

Similarly, under the assumptions of the second part, the line bundle $[2D] \cong [D] \otimes [\ov{\phi_*D}]$ admits an anti-holomorphic involution.
Thus, the line bundle  
$$K_X=[2D]\otimes [2D]$$
admits an anti-holomorphic involution and a real square root. By Lemma~\ref{lem:lift1}, the canonical real structure and the one induced by the above isomorphism on $K_X$ are equivariantly isomorphic; thus, $(K_X,\phi_{K_X})$ admits a real square root.
\end{proof}

In particular, if either $H^1(X,\R)=0$ and $c_1(TX)=0$ or $X$ is a compact K\"ahler Calabi-Yau
with anti-holomorphic involution $\phi$, then $(K_X,\phi_{K_X})$
admits a real square root. Similar but weaker result can be found in \cite[Lemma 2.9]{C}.
Through the isomorphism (\ref{equ:ind-tri}), we showed that the existence of a square root for $(K_X,\phi_{K_X})$ implies that the Lagrangian $L=\Fix(\phi)$ is orientable. 
In the case of a K\"ahler Calabi-Yau manifold, the orientability of the Lagrangian $L$ can also be seen directly as follows. 
It is possible to choose a holomorphic volume form $\Omega$ so that $\phi^*\Omega=\bar\Omega$.
On the fixed locus, it restricts to a (real-valued) volume form.

In Section~\ref{EquivLocal_sec} 
we consider $(\P^{4m-1},\eta_{4m-1})$; since $4|K_{\P^{4m-1}}$, it has a real square root. 
If $K_X$ is trivial as a complex line bundle and $H_1(X,\Z)=0$ (i.e $X$ is a symplectic Calabi-Yau manifold), then $K_X$ has a real square root; moreover, in this case we can fix an admissible trivialization of $K_X$ itself over $X$ (independent of any map $u$) and thus determine an orientation 
of the moduli space $\mc{M}_l(X,A)^{\phi,\eta}$.

As illustrated by the two examples below, there are cases where the determinant bundle is not orientable. 
The first example is similar to the non-orientable example of \cite[Section 8.1.2]{FOOO}.
\begin{example}\label{ex:1}
Let $E= S^1\times \P^1 \times \C \ra S^1 \times \P^1$. Define a family of involutions, 
 $$ 
\phi_{s}\colon E\mid_{\left\{s\right\}\times \P^1} \to E\mid_{\left\{s\right\}\times \P^1}, \quad \phi_s(z,v)= (\eta(z), \ov{e^{2\pi is}v})\quad \forall s\in S^1. $$
The real line bundle $F\ra S^1$ given by $F_s = H^0(E\mid_{\left\{s\right\}\times \P^1})_{\R}$ is then not orientable.
\end{example}

\begin{example}
Let $X=\R^2/\Z^2 \times \P^1 \times \P^1$, $A=\{\tn{pt}\}\times \P^1 \times \{\tn{pt}\} \in H_2(X,\Z)$,
\begin{equation}
\begin{split}
&\phi\colon X\to X,\qquad \phi(s,t,z,w)=(s,-t,-\frac{1}{\bar{z}},\ov{ e^{2\pi i s} w}),\\
&Y= \{(s,t,w)\in \R^2/\Z^2\times \P^1: (-t,\ov{e^{2\pi i s}w})=(t,w)\}.
\end{split}
\end{equation}
The space $Y$ is a union of two Klein bottles with double cover
$$ \R\cup\{\infty\} \times \R/\Z \times \{0,1/2\} \to Y,\quad (a,s,t) \to (2s,t,ae^{2\pi i s}).$$
Let $\pi\colon X \to \R^2/\Z^2 \times \P^1$ be the projection to the first and third factors. Since 
$$ f: \mc{M}(X,A)^{\phi,\eta} \to Y, \qquad [u] \to \pi(\tn{Im}(u)),$$
is well-defined and is a diffeomorphism, it follows that $\mc{M}(X,A)^{\phi,\eta}$ is not orientable. If $\gamma \subset \mc{M}(X,A)^{\phi,\eta}$ is the preimage of the map $S^1 \to Y$, $s\to (s,0,0)$,
$$ \gamma^* \tn{det}(T\mc{M}(X,A)^{\phi,\eta}) = \Lambda^{\top} H^0_{\R}(\gamma^*TX)\otimes (\Lambda^{\top} \tn{Lie}(G_{\eta}))^* = \R \otimes F,$$
where $F$ is the unorientable line bundle in Example~\ref{ex:1}.
\end{example}

\begin{remark}
In \cite{C}\footnote{Originally published in French about the same time as this paper 
was first published on arXiv, with an English summary \cite{CT} uploaded later.}, 
Cr\'etois approaches the orientation problem from a different point of view. 
He computes the induced sign of the action of an automorphism of a complex vector bundle with a  real structure  $(E,c_E)$ on the orientations of the determinant line bundle over the space of Cauchy-Riemann operators on $(E,c_E)$. 
His method is well suited to the case where the real locus of the underlying curve is not empty. 
The case related to our work is when $(\Si,c_\Si)=(\P^1,\eta)$ and the automorphism is the lift of either the identity automorphism or $\varphi([z,w])=[w,z]$ (with $\eta$ given by (\ref{equ:etainv})).
In this case, he uses (see \cite[Section 3.2.3]{CT}) a symplectic divisor $D$ (a polarization), invariant under the involution, which is Poincare dual of the first Chern class and finds an equation (\cite[Theorem 7]{CT}) for the first Stiefel Whitney class of the moduli of real maps that intersect $D$ transversely. 
For example, if the canonical bundle is the square of a complex line bundle admitting a real structure, and if one can find a nice section of this square root, then the first Stiefel Whitney class vanishes.
One issue with this approach is that we need to consider different polarizations to cover the entire moduli space and this increases the complexity of the calculations.     
In the cases where the ambient manifold has no real part, it is not clear how to construct a section of the canonical bundle.
Also, this approach does not provide a choice of orientating the moduli space.
\end{remark}

\subsection{Kuranishi structure}\label{kuranishi-structure}

If $\ov{\mc{M}}_l(X,A)^{\phi,\eta}$ is not an orbifold, in order to construct a virtual fundamental class, we need to put a Kuranishi structure on the moduli space. Such a construction for $\ov{\mc{M}}_{k,l}(X,A)^{\phi,\tau}$ is described in \cite[Section 7]{S}; we only describe the necessary adjustments.
For simplicity, we ignore the marked points until the end of this construction.\\
\newtheorem*{proofofboundary}{Proof of Proposition~\ref{boundary-freeness} and the first part of Theorem~\ref{thm:main}}
\begin{proofofboundary}
For $(u,(z_i,\ov{z_i})_{i=1}^{l}) \in \mc{M}_l(X,A)^{\phi,\eta}$, let 
$$E_u \equiv u^{*}TX \ra \P^1,\quad E_{u}^{0,1}\equiv (T^*\P^1)^{0,1} \otimes_{\C} E_u.$$
There are commutative diagrams
\begin{equation}\begin{split}\label{diag:E}
\xymatrix{
& E_u \ar[d]_{\pi} \ar[rr]^{T_{\phi}} && E_{u} \ar[d]_{\pi}&&&  E_u^{0,1} \ar[d]_{\pi} \ar[rr]^{T^1_{\phi}} && E_{u}^{0,1} \ar[d]_{\pi}\\
& \P^1             \ar[rr]^{\eta}     && \P^1                    &&& \P^1                     \ar[rr]^{\id}    && \P^1}
\end{split}\end{equation}
where $T_{\phi}v=d\phi(v)$ and $T^1_{\phi}\alpha=d\phi\circ\alpha\circ d\eta$.
The deformation theory of $\mc{P}_0(X,A)^{\phi,\eta}$ is described  by  the linearization of the Cauchy-Riemann operator,
\begin{equation}\label{linear}
 L_{J,u} \colon W^{k,p}(E_u) \to W^{k-1,p} (E_u^{0,1}), \quad p>2, k\geq1;   
\end{equation}
see \cite[Chapter 3]{MS2} for a similar situation. If $\nabla$ is the Levi-Civita connection of the metric $\om(\cdot,J\cdot)$, $L_{J,u}$ can be written as
$$L_{J,u} (\xi)= \frac{1}{2} (\nabla\xi + J \nabla\xi \circ j)-\frac{1}{2} J (\nabla_\xi J) \partial_J(u).$$
There is a commutative diagram
$$\xymatrix{
& W^{k,p}(E_u) \ar[d]_{\tilde{T}_{\phi}} \ar[rr]^{L_{J,u}} && W^{k-1,p}(E_u^{0,1}) \ar[d]_{\tilde{T}^1_{\phi}} \\
& W^{k,p}(E_u)                           \ar[rr]^{L_{J,u}} && W^{k-1,p}(E_u^{0,1})}                             $$
where $\{\tilde{T}_{\phi}\xi\}(z)=T_{\phi}(\xi(\eta(z)))$  and $\{\tilde{T}^1_{\phi}\alpha\}(z) = T^1_{\phi}(\alpha(z))$. 
Let 
\begin{equation}\begin{split}
&W^{k,p}(E_u)_{\R}=\{\xi\in W^{k,p}(E_u)\mid \tilde{T}_{\phi}(\xi)=\xi\},	\\
&W^{k-1,p}(E_u^{0,1})_{\R}=\{ \alpha \in W^{k-1,p}(E_u^{0,1})\mid \tilde{T}^1_{\phi}(\alpha)=\alpha\}
\end{split}\end{equation}
denote the spaces of \textsf{real} sections.
Let $H^0(E_u)_{\R}$ and $H^1(E_u)_{\R}$ be the kernel and cokernel, respectively, of the restricted operator
$$L_{J,u} \colon W^{k,p}(E_u)_{\R} \to W^{k-1,p} (E_u^{0,1})_{\R}.$$
If $H^1(E_u)_{\R}=0$, then $\mc{P}(X,A)^{\phi,\eta}$ near $u$ is a manifold $U(u)$ of real dimension
\begin{equation}
\dim_{\R} H^0(E_u)_{\R} = \tn{index}_{\R} (L_{J,u}) = c_1(A)+\dim_{\C}X;
\end{equation}
see \cite[Theorem C.1.10]{MS2}. 
In order to stabilize the domain and kill the action of the automorphism group, we need to take an slice $V(u)$ of the $G_\eta$-action on $U(u)$.
To this end, we add back few conjugate pairs of special marked points $\{(w_i,\eta(w_i))\}$, 
fix a corresponding set of (conjugate pairs of) slicing local divisors $D_i$, and take $V(u)$ to be the submanifold of the maps in $U(u)$ that intersect $D_i$ at $w_i$.
Then, by restricting the location of special marked points we will find a slice $V(u)$ of the $G_\eta$-action such that $V(u)$ is a submanifold of $U(u)$ and $\dim_\R V(u) = \dim_\R U(u)-3$. 
Each pair of ordinary conjugate marked points increases the dimension by two and we get the dimension formula (\ref{dimension}).

If $H^1(E_u)_{\R}\neq 0$, we construct a Kuranishi chart around $u$. For this aim, we choose finite-dimensional complex subspaces $\mc{E}_u \subset W^{k,p-1}(E^{0,1}_{u})$ such that
\begin{enumerate}[leftmargin=*]\label{obstruction-bundle}
	\item every $\xi \in \mc{E}_u $ is smooth and supported away from the boundary, special, and marked points;
	\item $\tilde{T}_{\phi}^1(\mc{E}_u)= \mc{E}_u$;
	\item $L_{J,u}$ modulo $\mc{E}_u$ is surjective.
\end{enumerate}
After putting enough marked points and slicing conditions to kill the automorphism group, we choose our Kuranishi neighborhood to be $V(u)= [\bar\partial^{-1}(\mc{E}_u)]_{\R}$, which is a smooth manifold of dimension 
$$c_1(A)+\dim_{\C}X-3 + 2l + \dim_{\C}(\mc{E}_u).$$

The obstruction bundle $\mc{E}(u)$ at each $f \in V(u)$ is obtained by parallel translation of $\mc{E}_{u}$ with respect to the induced metric of $J$. We thus get a Kuranishi neighborhood $(V(u),\mc{E}(u))$. The Kuranishi map in this case is just the Cauchy-Riemann operator $f \ra \bar\partial(f)$.

In order to construct Kuranishi charts for $u$ in the boundary strata of $\ov{\mc{M}}_l(X,A)^{\phi,\eta}$, we need gluing theorems as in \cite[Chapter 7]{FOOO}. The gluing theorems are identical to those for $J$-holomorphic disks; we thus omit the details and refer the reader to \cite{FOOO}.
\qed
\begin{remark}
If $(X,\om)$ is semi-positive or strongly semi-positive, \cite[Definition 6.4.5]{MS2} or \cite[Definition 7.1]{G2}, then the invariants can be defined via classical (geometric) methods of \cite{MS2} or \cite{RT}. In the strongly semi-positive case, this is for example done for the $(\phi,\tau)$-moduli space in \cite[Theorem 1.4]{G2}.
\end{remark}
\end{proofofboundary}

\section{Proof of Theorem~\ref{thm:main} and real GW invariants}\label{ch:mixed-invariants}

We continue this section with the proof of the first part of Theorem~\ref{thm:main}.
If $L=\Fix(\phi)$ is non-empty, the codimension one boundary of $\ov{\mc{M}}_l(X,A)^{\phi,\eta}$ 
might be non-empty; see (\ref{eq:sphere-bubbling}).
An element of codimension one boundary is of the form $(u,\Si= \Si_1 \cup_{q} \Si_2)$, 
where $\Si_i=\P^1$, $\eta \colon \Si_1 \to \Si_2$, and $u(q) \in L$. 
After a suitable  reparametrization, we may assume $q=0 \in \P^1$ and $\eta(z)=\ov{w}$. 
For real parameters $\ep\neq 0$, we can glue $\Si$ into a family of smooth curves 
$$ \Si_{\ep} = \{ (z,w) \in \C^2 : zw = \ep\}.$$
For $\ep\in\R$, $\Si_{\ep}$ inherits a complex conjugation from $\eta$: 
$$\eta_{\ep} \colon \Si_{\ep}\to\Si_{\ep},\qquad   \eta_{\ep}(z,w)= (\ov{w},\ov{z}).$$ 
The fixed point set of $\eta_{\ep}$ is $S^1$ if $\ep>0$ and is empty if $\ep<0$. 
Assuming regularity, by smoothing in one direction ($\ep$ negative), we get real curves without fixed points 
in $\ov{\mc{M}}_l(X,A)^{\phi,\eta}$;  by smoothing in the other direction ($\ep$ positive), 
we get real curves with fixed points in $\ov{\mc{M}}_l(X,A)^{\phi,\tau}$. 
We identify the common boundary and glue the two moduli spaces to get a new moduli 
space whose only possible boundary component comes from the disk bubbling. 
We define $\ov{\mc{M}}(X,A)^{\phi}$ to be the resulted space. 

Every disk-bubbling type nodal curve in (\ref{boundary-equ}) is of the form $(u,\Si=\Si_1\cup_q\Si_2)$, 
where $\Si_i=\P^1$, $\tau\colon\Si_i\to\Si_i$, and $q\in \Fix(\tau\!|_{\Si_i})\cong S^1$. 
After a suitable reparametrization, we may assume $q=(z_i=0)\in \C\subset \P^1$ and $\tau\!|_{\C}=c$, where $c(z_i)=\bar{z_i}$. 
For real $\ep\neq 0$, we can glue $\Si$ into a family of smooth curves
$$\Si_{\ep}=\{(z_1,z_2)\in \C^2~:~z_1z_2=\ep\}.$$
For $\ep\in\R$, $\Si_{\ep}$ inherits a complex conjugation from $\tau$: 
$$\tau_{\ep} \colon \Si_{\ep}\to\Si_{\ep},\qquad   \tau_\ep(z_1,z_2)= (\ov{z}_1,\ov{z}_2).$$ 
The fixed point set of $\tau_\ep$ is $S^1$. 
By the stability condition, for each component~$i$, 
either $l_i\neq 0$ or the map $u_i\equiv u|_{\Si_i}$ is non-trivial. 
If $l_i\neq 0$, we fix one of the marked points; 
if $u_i$ is non-trivial and somewhere injective, we fix a somewhere injective point 
of the corresponding domain. 
By tracking the image of the chosen points\footnote{In one direction, the images of these two points lie in one half disk, and in the other direction, they lie in different half disks.}, we see that gluing the map 
in positive and negative directions produce different $J$-holomorphic curves. 
If $u_i$ is multiple cover and $l_i=0$, then the obstruction bundle near
$u_i\in\mc{M}_{1,0}(X,A_i)^{\phi,\tau}$ is non-trivial and a Kuranishi neighborhood 
depends on the choice of $\mc{E}_{u_i}$ of the previous section. 
In this situation, the Cauchy-Riemann equation gives a section of the obstruction bundle.
Then we need to deform this section into close by transversal multi-sections to build a virtual fundamental class; see \cite[Section 7]{FOOO}.
By choosing (the branches of) these multi-sections non-symmetric with respect to the deck transformation of the covering map, 
we can assure that gluing in different directions produce different maps.
Therefore, the real codimension one strata (\ref{boundary-equ}) and (\ref{eq:sphere-bubbling}), corresponding 
to $\ep=0$, are indeed hypersurfaces. 
This establishes the first part of Theorem~\ref{thm:main}, i.e.~that $\ov{\mc{M}}(X,A)^{\phi}$ 
has the structure of a closed Kuranishi space;
the real codimension one strata (\ref{boundary-equ}) and (\ref{eq:sphere-bubbling}) 
are real codimension one hypersurfaces in $\ov{\mc{M}}(X,A)^{\phi}$. 
\qed

If $4|c_1(TX)$, given a compatible choice of a real square root for $K_X$ and a spin structure on $L=\Fix(\phi)$, 
the next two proposition and lemma show that $\ov{\mc{M}}_l(X,A)^{\phi}$ is orientable (in fact, oriented). 
Given such a compatible choice, the spaces $\mc{M}_l(X,A)^{\phi,\tau}$
and $\mc{M}_l(X,A)^{\phi,\eta}$ can be oriented;
see the beginning of Section~\ref{sec:open} and 
Theorem~\ref{square-root-orientation-thm}. 
The first proposition below states that the orientation of $\mc{M}_l(X,A)^{\phi,\tau}$
extends across the hypersurface~(\ref{boundary-equ}). 
Then Proposition~\ref{matching-orientation_prp} 
implies that the orientations of $\ov{\mc{M}}_l(X,A)^{\phi,\tau}$ and
$\ov{\mc{M}}_l(X,A)^{\phi,\eta}$ are compatible along the common boundary~(\ref{eq:sphere-bubbling}). 
In Proposition~\ref{matching-orientation_prp}, 
we consider the induced orientation on the boundary $\partial{M}$ of an oriented manifold $M$ to be the one given by the inward normal vector field;
i.e. 
\begin{equation}\label{equ:indori}
TM|_{\partial{M}}=T\partial{M}\oplus \R \cdot v_{\tn{in}},
\end{equation}
is an isomorphism of oriented vector spaces.
Therefore, if $M_1$ and $M_2$ are oriented and the induced orientations on $\partial M\equiv \partial M_1\cong \partial M_2$ are reverse of each other,
$M_1\cup_{\partial M} M_2$ inherits an orientation. 

\newtheorem*{proofofboundary2}{Proof of the second part of Theorem~\ref{thm:main}}
\begin{proofofboundary2}

\begin{proposition}\label{flip-orientation}
Let $(X,\om,\phi)$ be a symplectic manifold with a real structure. If $4|c_1(TX)$, then a choice of spin structure on $L=\Fix(\phi)$ determines an orientation on $\ov{\mc{M}}(X,A)^{\phi,\tau}$.
\end{proposition}

This proposition is a special case of \cite[Theorem 1.4]{G2} or \cite[Theorem 1.1]{FO-3}. 
We first lift the orientation problem to the corresponding moduli spaces of decorated $J$-holomorphic disks.
A spin structure determines an orientation on 
$$\mc{M}_{1,l}(X,L,\beta)^\disk_\dec~\tn{and}~\mc{M}_l(X,L,\beta)^\disk_\dec.$$
This orientation induces\footnote{The fiber product orientation depends on the convention.} an orientation on 
$$
\mc{M}_{1,l_1}(X,L,\beta_1)^\disk_\dec \times_{(\ev^B_1,\ev^B_1)} \mc{M}_{1,l_2}(X,L,\beta_2)^\disk_\dec 
$$
which descends to an orientation on (\ref{boundary-equ}).
Let
$$
\mc{M}_{1,l_1}(X,L,\beta_1)^\disk_\dec \times_{(\ev^B_1,\ev^B_1)}
 \mc{M}_{1,l_2}(X,L,\beta_2)^\disk_\dec \times \R^{\geq 0} \stackrel{\Psi}{\rightarrow} \ov{\mc{M}}(X,L,\beta)^\disk_\dec
$$
be the gluing map.
With a proper convention of defining the fiber product orientation, $\Psi$ is orientation preserving.
Then we observe that gluing positively or negatively in the real curve formulation corresponds to the flipping one of the disk components via $\tau_{\mc{M}}$ (the map $\tau_*$ in \cite[Theorem 1.1]{FO-3}), i.e. gluing negatively corresponds to $(u_1,u_2,\ep) \to \Psi(u_1,\tau_{\mc{M}}(u_2),-\ep)$; c.f. (\ref{equ:tauM}). 
Finally, \cite[Theorem 1.4]{G2} or \cite[Theorem 1.1]{FO-3} shows that if $4|c_1(TX)$, this flipping action is orientation reversing, hence the gluing map 
$$
(\mc{M}_{1,l_1}(X,A_1)^{\phi,\tau} \times_{(\ev^B_1,\ev^B_1)}
 \mc{M}_{1,l_2}(X,A_2)^{\phi,\tau}/G) \times \R \to  \ov{\mc{M}}(X,A)^{\phi,\tau}
$$
given by smoothing domain with respect to the corresponding gluing parameter $\ep$, is an oriented isomorphism. Therefore, the orientation of ${\mc{M}}(X,A)^{\phi,\tau}$ extends across the disk-bubbling codimension one strata.\qed

\begin{remark}
A  relative spin structure $[V,\si]$ on $(X,L)$ also determines an orientation 
on every moduli space $\mc{M}_l^{\disk}(X,L,\beta)$ are still orientable;
see \cite[Theorem~8.1.1]{FOOO}.
These orientations descend to $\mc{M}_l(X,A)^{\phi,\tau}$ if
$$\frac12 \lr{c_1(TX),A}\equiv \lr{w_2(V),A} \mod2;$$
see \cite[Corollary~5.9]{XCapsSigns}.
The conclusion of Proposition~\ref{flip-orientation} is still true.
For example, $(\P^{4m+1},\R\P^{4m+1})$ is not spin, but is relatively spin; 
each choice of the two homotopy classes of relative spin structures determines 
an orientation on $\ov{\mc{M}}_l(X,A)^{\phi,\tau}$.
In the more general setting of Pin structures, 
analogous sign computations are carried out in \cite[Proposition 2.12]{S}.
\end{remark}

\begin{proposition}\label{matching-orientation_prp}
Let $(X,\om,\phi)$ be a symplectic manifold with a real structure such that $4|c_1(TX)$. Given a compatible pair of a real square root for $K_X$ and a spin structure on $L=\Fix(\phi)$, 
if $A,B\in H_2(X)$ are such that $A=B-\phi_*B$, then via the gluing maps
\begin{equation}\label{glue-map}
 (\mc{M}_{1}(X,B)\times_{\ev_1} L) \times \R^+  
\to\mc{M}(X,A)^{\phi,\tau},\mc{M}(X,A)^{\phi,\eta}
\end{equation}
the induced orientations on the first component of 
the left-hand side as the boundary of $\ov{\mc{M}}(X,A)^{\phi,\tau}$ and $\ov{\mc{M}}(X,A)^{\phi,\eta}$ are reverse of each other.
\end{proposition}

\begin{proof}
By definition, the induced orientations on $L$ via the given spin structure and the given real square root are inverse of each other.
Without lose of generality, we may assume that the orientation of $L$ coincides with the induced orientation of the given spin structure and is reverse of the one given by real square root.
A curve in the common boundary of these two moduli spaces is of the form 
$$f=[u,\Si=\P^1_{\top} \cup_q \P^1_{\bot}],$$ 
with the involution $c$ over $\Si$ having one fixed point, the node $q$.
We replace each such $f$ with the unstable map 
$$\tilde{f}=[\tilde{u},\tilde{\Si}=\P^1_{\top} \cup \P^1_0 \cup \P^1_{\bot}],$$
with $\tilde{u}$ restricting to the constant $u(q)$ over the central part $\P^1_0$.  
We can view $\tilde{f}$ as an element of $\partial\ov{\mc{M}}(X,A)^{\phi,\tau}$ by extending 
the involution to $\P^1_0$ via $c|_{\P^1_0}=\tau$ and as an element of
$\partial\ov{\mc{M}}(X,A)^{\phi,\eta}$ by extending the involution to $\P^1_0$ via 
$c|_{\P^1_0}=\eta$. 
The real automorphism group of $\tilde{f}$ restricted to the middle component is $S^1$.
 
First, lets consider $\tilde{f}$ with $c|_{\P^1_0}=\tau$. In this case we can divide $\tilde{f}$ into two nodal $J$-holomorphic disks; let 
$$\tilde{f}_\top=[\tilde{u},\tilde{D}=\P^1_{\top} \cup D_0],$$
be the half including $\P^1_\top$ (the final conclusion is independent of the particular choice). 
For simplicity, we may assume that $\tilde{f}_\top$ can be glued to a $J$-holomorphic disk $f_\ep$ over the glued domain $D_\ep\cong D$; otherwise, we need to consider the obstruction bundle.
In order to understand the induced orientation on $T_f\partial\ov{\mc{M}}(X,A)^{\phi,\tau}$, we need to understand the orientation on $T_{f_{\ep}}\mc{M}^\disk(X,L,\beta)$ and extend it to the sphere-bubbling boundary. 
Following the orientation and gluing argument in \cite[Section 8.3]{FOOO} and 
\cite[Section 7.4.1]{FOOO}, in order to orient $T_{f_{\ep}}\mc{M}^\disk(X,L,\beta)$, we orient the tangent bundle of the parametrized $J$-holomorphic disks $T_{u_{\ep}}\mc{P}^\disk(X,L,\beta)$ and then consider the quotient orientation on $T_{f_{\ep}}\mc{M}^\disk(X,L,\beta)$ for which 
$$ 
T_{f_{\ep}}\mc{M}^\disk(X,L,\beta)\oplus T_{\tn{id}}G_\tau^0 \cong T_{u_{\ep}}\mc{P}^\disk(X,L,\beta)
$$
is an oriented isomorphism of vector spaces; see Lemma~\ref{new-lemma} for the negative sign.
In order to orient $T_{u_{\ep}}\mc{P}^\disk(X,L,\beta)$, we trivialize $u_{\ep}|_{\partial D_\ep}^*TL$ via the given spin structure, and degenerate $u_{\ep}^*TX$ into a bundle over $\tilde{D}$, such that over the central part the induced bundle is trivial. 
The path $\{u_\ep\}_{\ep\to 0}$ exactly describes such a degeneration. 
Over the central part $\tilde{u}|_{D_0}$ is trivial and by assumption, the orientation of $T_{u(q)}L$ coincides with the one given by the spin structure;
therefore, the space of real sections of $\tilde{u}|_{D_0}TX$ is orientably isomorphic to $T_{u(q)}L$. 

Now, lets consider $\tilde{f}$ with $c|_{\P^1_0}=\eta$.
For simplicity, we may again assume that $\tilde{f}$ can be glued to an $(\phi,\eta)$-real $J$-holomorphic map $f_\ep$ over the glued domain $\Si_\ep\cong \P^1$. 
Following the orientation procedure of the proof of Lemma~\ref{trivialization-to-orientation}, in order to orient $T_{f_{\ep}}{\mc{M}}(X,A)^{\phi,\eta}$, we orient the tangent bundle of the parametrized $J$-holomorphic spheres $T_{u_{\ep}}{\mc{P}}(X,A)^{\phi,\eta}$ and then consider the quotient orientation as above. 
In order to orient $T_{u_{\ep}}{\mc{P}}(X,A)^{\phi,\eta}$,
we fixed an admissible trivialization of $u_{\ep}^*TX$ over $\P^1-\{0,\infty\}$ given by the real square root and degenerated $u_{\ep}^*TX$ into a bundle over $\tilde{\Si}$, such that over the central part the induced bundle is admissibly trivial. 
The path $\{u_\ep\}_{\ep\to 0}$ exactly describes such a degeneration. 
Once again, over the central part $\tilde{u}|_{\P^1_0}$ is trivial and by assumption, the orientation of $T_{u(q)}L$ is reverse of the one induced from the real square root;
therefore, the space of real sections of $\tilde{u}|_{\P^1_0}TX$ is orientably isomorphic to $T_{u(q)}L$ with the reverse orientation. 

Similar to Section~\ref{kuranishi-structure}, for $\tilde{f}\in \partial \ov{\mc{M}}(X,A)^{\phi,\tau},\partial\ov{\mc{M}}(X,A)^{\phi,\eta}$, let
$E_{\tilde{u}}\equiv \tilde{u}^*TX$ and $E_{\tilde{u}}^{0,1}\equiv E_{u_\top}^{0,1}\oplus E_{u_0}^{0,1} \oplus E_{u_\bot}^{0,1}$, where $u_\top$, $u_0$, and $u_\bot$ are the restrictions of $u$ to the corresponding components, respectively. 
The deformation theory of $\ov{\mc{P}}(X,A)^{\phi,\tau}$ and $\ov{\mc{P}}(X,A)^{\phi,\eta}$ at $f$ is described by the linearization of the Cauchy-Riemann operator
$$
L_{J,\tilde{u}}\colon W^{k,p}(E_{\tilde{u}})\to W^{k-1,p}(E_{\tilde{u}}^{0,1}), ~~~p>2,k\geq 1,
$$
where $L_{J,\tilde{u}}=L_{J,u_\top}\oplus L_{J,u_0}\oplus L_{J,u_\bot}$.
Let $H^0(E_{\tilde{u}})_\R$ and $H^1(E_{\tilde{u}})_\R$ be the kernel and cokernel, respectively, of the restricted operator,
$$
L_{J,\tilde{u}}\colon W^{k,p}(E_{\tilde{u}})_\R\to W^{k-1,p}(E_{\tilde{u}}^{0,1})_\R.
$$
Assuming $H^1(E_{\tilde{u}})_\R=0$ (otherwise we need to work modulo obstruction bundle), via the gluing maps $f \to \{f_\ep\}_{\ep>0}$, 
\begin{equation}\label{equ:orient+}
H^0(E_{\tilde{u}})_\R\cong H^0(E_{u_\ep})_\R\cong T_{u_\ep}\mc{P}(X,A)^{\phi,\tau}
\end{equation}
and 
\begin{equation}\label{equ:orient-}
H^0(E_{\tilde{u}})_\R\cong H^0(E_{u_\ep})_\R\cong T_{u_\ep}\mc{P}(X,A)^{\phi,\eta}.
\end{equation}
As in the proof of Lemma~\ref{trivialization-to-orientation}, 
the orientation of $H^0(E_{\tilde{u}})_\R$ is canonically determined by the orientation of $T_{u(q)}L$. 
By the argument of the past two paragraphs, the orientation on the left-hand side of (\ref{equ:orient+}), via the gluing map, gives the orientation on the right-hand side determined by the spin structure; and the orientation on the left-hand  side of (\ref{equ:orient-}) (again via the gluing map) gives the reverse orientation on the right-hand side determined by the real square root.  

Finally, in order to complete the comparison, it remains to compare the automorphism groups of domains before and after two different gluings. 
Let 
$$G_0=\tn{Aut}_{\R}([\P^1_{\top}\cup_q\P^1_{\down}])\cong \tn{Aut}(\P^1_{\top},q)\subset \tn{PSL}(2,\C)$$
be the identity component of  the real automorphism group of $\Si$. 
This is a real $2$-dimensional complex Lie group which has a canonical orientation (although we do not care about its orientation; see Remark~\ref{fiber-prod-ori}). 
After replacing $\Si$ with $\tilde{\Si}$, the real automorphism group of the domain increases by a factor of $S^1$ and the gluing parameter of the domain takes values in $\C$. 
To kill the extra $S^1$-action in both the automorphism group and the gluing parameter, as in the statement of the lemma, we consider the gluing parameter to be positive real (absolute value of the complex one) and restrict to a real $4$-dimensional section of $G_0\times S^1$, given by
$$G_0'=\{ (g,e^{\mf{i}\theta})\subset G_0 \times S^1:~~ dg|_{T_0\P^1_{\top}}\in R^{+} e^{-\mf{i}\theta}\},$$
which is canonically isomorphic to $G_0$. 
Let
$$\mc{C}=\{ (z_t,z_0,z_b,\ep)\in\P^1\times \P^1\times \P^1 \times \R^{\geq 0}| z_tz_0=\ep, z_b\ov{c(z_0)}=\ep, ~~\ep\in \R^{\geq0}\}.$$
This is a real one-parameter family of genus zero real curves over $\R^{\geq 0}$,
$$(z_t,z_0,z_b,\ep) \to \ep,$$
with the fiber-preserving involution
$$ (z_t,z_0,z_b,\ep)\to (\bar{z_b},c(z_0),\bar{z_t},\ep),$$
that describes a gluing of the singular real curve $\tilde{\Si}$ into smooth real curves.

Over $\mc{C}$, consider the group $G_0^{\mc{C}}$ generated by the following set of maps
\begin{equation}
\begin{split}
&R_{r}\colon (z_t,z_0,z_b,\ep) \to (rz_t,z_0,rz_b,r\ep),\quad r \in \R^{+}~~\tn{close to 1}, \\
&R_{\theta}\colon (z_t,z_0,z_b,\ep) \to (e^{-\mf{i}\theta}z_t,e^{\mf{i}\theta}z_{0},e^{\mf{i}\theta}z_b,\ep),\quad e^{\mf{i}\theta}\in S^1,\\
&T_{a}\colon (z_t,z_0,z_b,\ep) \to (\frac{z_t+(-1)^{|c|}\ep^2 \bar{a}}{1+az_t}, \frac{z_0+a\ep}{1+(-1)^{|c|}\ep \bar{a}z_0}, \frac{ z_b+ (-1)^{|c|} a\ep^2}{1+\bar{a}z_b}), \quad a \in \C,
\end{split}
\end{equation}
with $|c|$ defined as in $(\ref{absphi})$. This group extends the action of $G_0'$ to the whole family. Restricted to each fiber $\mc{C}_{\ep}$, $\ep\neq 0 $, $R_{\theta}$ and $T_a$ generate the 3-dimensional real automorphism group of the fiber,  $G_{\eta}$ or $G^0_{\tau}$, depending on $c$. 
Let 
$$v_1 = \frac{d}{dr}R_r|_{r=1}, \quad v_2=\frac{d}{d\theta}R_{\theta}|_{\theta=0}, \quad v_3+\mf{i}v_4= \frac{d}{da}T_{a}|_{a=0}.$$
Restricted to $\mc{C}_0$, $-v_1,v_2,v_3,v_4$ form an oriented basis of $T_{\tn{id}}G'_0$. 
With the conventions of Section~\ref{ch:realcurves}, the restriction of  
$v_2,v_3,v_4$ to $\mc{C}_{\ep}$, $\ep \neq 0$, forms an oriented basis of $T_{\tn{id}}G_\eta$ or $T_{\tn{id}}G^0_\tau$. Finally, $v_1$ (after some positive rescaling) is a lift of the inward normal vector field $\frac{\partial}{\partial\ep}$ to the family. 
Let $ (T_f\mc{M}_1(X,B)\times_{\ev_1}L)^{[\partial \tau]}$ and $(T_f\mc{M}_1(X,B)\times_{\ev_1}L)^{[\partial \eta]}$ denote $T_f\mc{M}_1(X,B)\times_{\ev_1}L$, oriented as the boundary of $(\phi,\tau)$-space and $(\phi,\eta)$-space, respectively.
Then for each choice of $c=\tau,\eta$
$$
T_f(\mc{M}_1(X,B)\times_{\ev_1}L)^{[\partial c]}\oplus \R\cdot v_1 \cong T_f\ov{\mc{M}}(X,A)^{\phi,c}
$$ 
is an oriented isomorphism of vector spaces.
Adding the oriented basis $v_2,v_3,v_4$ to both sides, we find that 
$$
T_f(\mc{M}_1(X,B)\times_{\ev_1}L)^{[\partial c]} \oplus \bigoplus_{i=1}^{4}\R\cdot v_i\cong T_f\ov{\mc{M}}(X,A)^{\phi,c}
 \oplus \bigoplus_{i=2}^{4} \R\cdot v_i\cong T_{f}(\ov{\mc{P}}(X,A)^{\phi,c})$$
 is an oriented isomorphism of vector spaces.
 Finally, replacing the right-hand side by $H^0(E_{\tilde{u}})_\R$, since the induced orientation on $H^0(E_{\tilde{u}})_\R$ via the spin structure and the real square root are reverse of each other,
we conclude that the orientations $T_f(\mc{P}_1(X,B)\times_{\ev_1}L)^{[\partial \tau]}$ and $T_f(\mc{P}_1(X,B)\times_{\ev_1}L)^{[\partial \eta]}$ are reverse of each other. 
\end{proof}
\noindent
This finishes the proof of Theorem~\ref{thm:main}.
\qed
\end{proofofboundary2}

Thus, if $K_X$ has a real square root, $L$ is spin, and $4|c_1(TX)$, choosing the spin structure and the square root compatibly, $\ov{\mc{M}}_l(X,A)^{\phi}$ is closed and oriented. In this case, for $\theta_1,\ldots,\theta_l\in H^*(X)$, we define real GW invariants by 
$$
N_A^{\phi}(\theta_1,\ldots,\theta_l)=\int_{[\ov{\mc{M}}_l(X,A)^{\phi}]^{\tn{vir}}}\ev_1^*(\theta_1)\wedge \cdots \wedge \ev_{l}^*(\theta_l).
$$

\begin{remark}\label{fiber-prod-ori}
In the proof of Proposition~\ref{matching-orientation_prp}, 
we did not calculate the fiber product orientation on the left-hand side of (\ref{glue-map}); we just showed that the induced boundary orientations are reverse of each other.
After fixing a fiber product orientation convention, it is not hard to compare the induced and fiber product orientations directly. 
\end{remark}

\section{Proof of Theorem~\ref{zeroness}}\label{sec:zeroes}
By \cite[Proposition 2.1]{MF1}, there exists a symplectic degeneration  $\pi\colon\mc{X}\to \De$ of $X$ with an induced real structure $\phi_{\mc{X}}$ over a disk $\De\subset \C$ (which we can assume to be the unit disk) such that the central fiber $X_0=\pi^{-1}(0)$ is a simple normal crossing symplectic manifold with real structure $\phi_{\mc{X}}|_{X_0}$,
$$
X_0= X_- \cup_D X_+,\quad D\cong \P^1\times \P^1, \quad \phi_{\pm}:=\phi_{\mc{X}}|_{X_{\pm}},
$$
where $(X_-,\phi_-)$ is symplectomorphic to real quadratic hypersurface in $\P^4$ given by 
$$
x_0^2-\sum_{i=1}^{4}x_i^2=0,\qquad D=(x_0=0),
$$
and $\Fix(\phi_+)=\emptyset$. 
Moreover, the fibers over $\De^*$ are smooth and symplectically isotopic to $(X,\phi)$. Note that $\mc{N}_DX_+\cong \mc{O}(-1)|_D$; therefore, all curves inside $D$, as curves in $X_+$, have negative intersection with $D$.

\vskip.1in
Via the symplectic sum procedure, every almost complex structure $J_0$ on $X_0$, i.e. a union of two almost complex structures $J_+$ and $J_-$ on $X_+$ and $X_-$, respectively, such that both preserve $D$, extends to an almost complex structure $J_{\mc{X}}$ on $\mc{X}$. If $J_\pm$ are $(\om_\pm,\phi_\pm)$-compatible, the resulting $J_{\mc{X}}$ is also $(\om_{\mc{X}},\phi_{\mc{X}})$-compatible.
Each fiber of $\mc{X}$ over $[0,1]\subset \De$ is invariant under $\phi_{\mc{X}}$. For $t\neq (0,1]$, 
$$
(X_t=\pi^{-1}(t),\om_t=\om_{\mc{X}}|_{X_t},\phi_t=\phi_{\mc{X}}|_{X_t})
$$ 
is isomorphic to $(X,\om,\phi)$; therefore, we can think of $\{J_t=J_{\mc{X}}|_{X_t}\}_{(0,1]}$ as a family of compatible almost complex structures on $(X,\om,\phi)$ converging to the singular almost complex structure $J_0$.

\vskip.1in
Set
\begin{equation}\label{family-moduli}
\ov{\mc{M}}(X,A,\left\{J_t\right\}_{t\in (0,1]})^{\phi}\equiv \bigcup_{t\in(0,1]} \ov{\mc{M}}(X,A,J_t)^{\phi}.
\end{equation}
Let $\ov{\mc{M}}(X,A,\left\{J_t\right\}_{t\in [0,1]})^{\phi}$ be the relative stable map compactification of (\ref{family-moduli}), as in \cite[Section 3.2]{MF1}, which includes ``stable" real maps into $X_0$. 
Every element $(u,\Si)$ of $\ov{\mc{M}}(X,A,\left\{J_t\right\}_{t\in [0,1]})^{\phi}$ with image in  $X_0$ belongs to a fiber product of  real relative moduli spaces over $X_{-}$ and $X_{+}$ with matching intersections along $D$, i.e. 
\begin{equation}\label{equ:fiber-prod}
\ov{\mc{M}}(X_-,D,\rho,\Gamma_{-})^{\phi_-}\times_{(\ev_{\xi^-},\ev_{\xi^+})} \ov{\mc{M}}(X_+,D,\rho,\Gamma_{+})^{\phi_+},
\end{equation}
where $\ov{\mc{M}}(X_-,D,\rho,\Gamma_{-})^{\phi_-}$ and $\ov{\mc{M}}(X_+,D,\rho,\Gamma_{+})^{\phi_+}$ are the relative moduli spaces of real curves, possibly with disconnected domains, with the same intersection pattern $\rho$. 
Here $\xi^\pm$ are the contact points with $D$, $\tn{ev}_{\xi^\pm}$ are the evaluation maps at $\xi^\pm$, and $\Gamma_{\pm}$ encodes the data corresponding to the topological types of the domain and image; see \cite[Section 4]{MF1} for more details on the definition.
  
The moduli space $\ov{\mc{M}}(X,A,\left\{J_t\right\}_{t\in [0,1]})^{\phi}$ gives a cobordism between the moduli space of real curves $\ov{\mc{M}}(X,J_1,A)^{\phi}$ over a smooth fiber and the moduli space of real curves in the singular fiber. By \cite[Proposition 2.1]{MF1}, $c_1(TX_+)=-PD(D)$. Therefore, if the image of the maps in $\ov{\mc{M}}(X_+,D,\rho,\Gamma_{+})^{\phi_+}$ have homology class $B\in H_2(X_+,\Z)$ with $B\cdot D> 0$, then 
$$
\dim^{\vir}(\ov{\mc{M}}(X_+,D,\rho,\Gamma_{+})^{\phi_+}) \leq c_1(TX_+)(B) < 0.
$$
This implies that for ``generic" $J_0$, we should expect that the limit maps in $X_0$ to lie entirely in $X_+\setminus D$ or $X_-\setminus D$. On the other hand, note that the degree (or symplectic area) of every $J_-$-holomorphic map in $X_-$ is proportionate to its intersection number with $D$, because $D$ is the hyperplane class in $X_-$; therefore, every non-trivial $J_-$-holomorphic map in $X_-$ has non-trivial intersection with $D$. We conclude that for such generic $J_0$, the only non-empty terms in (\ref{equ:fiber-prod}) correspond to real $J_+$-holomorphic maps inside $X_+$. Since $\phi_+$ has no fixed points,  the proof then follows from the Gromov compactness theorem.

\vskip.1in
The argument for finding such generic $J_0$ is almost identical to that of \cite[Theorem 3.1.5]{MS2}. We provide the necessary adjustments. Let $\mc{J}_{D,w_+,\phi_+}$ be the space of almost complex structures in $\mc{J}_{w_+,\phi_+}$ which preserve\footnote{We can impose more regularity condition along $D$ and it does not affect the argument below. In fact we may even assume that $J_+$ is holomorphic in a neighborhood of $D$.} $TD$. Similar to the argument in the middle of \cite[Page 47]{MS2}, the tangent space  
$$
T_{J_+}\mc{J}_{D,w_+,\phi_+}
$$
consists of those $Y\in \tn{End}(TX_+)$ where
\begin{equation}\label{tangent-condition}
\aligned
&YJ_++J_+Y=0,\quad \om_+(Yv,w)+\om_+(v,Yw)=0, \\
 &\phi_+^*Y=-Y,\qquad \tn{and}\quad Y(TD)\subset TD.
 \endaligned
\end{equation}
The extra conditions in the second row correspond to compatibility with $\phi_+$ and $D$, respectively. In the following argument, we consider two types of moduli spaces. Fix some $B\in H_2(X_+,\Z)$ such that $B\cdot D>0$. For every $J_+\in \mc{J}_{D,w_+,\phi_+}$, let $\mc{M}^*(X_+,B,J_+)$ be the (ordinary) moduli space of somewhere injective $J_+$-holomorphic spheres of degree $B$ whose image, as a set in $X_+$, is ``not" invariant under the action of $\phi_+$. Every element of $\mc{M}^*(X_+,B,J_+)$ intersects $D$ in finite set of points with total multiplicity $B\cdot D$.  
Similarly, let $\mc{M}^*(X_+,B,J_+)^{\phi_+,\eta}$ be the moduli space of degree $B$ somewhere injective $(\phi_+,\eta)$-maps $u\colon (\P^1,\eta)\to(X_+,\phi_+)$.  
By adjusting the proof of \cite[Theorem 3.1.5]{MS2}, we prove the following proposition.

\begin{proposition}\label{empty-lemma}
For every $B\in H_2(X_+,\Z)$ with $B\cdot D>0$, there exists a  set of second category $\mc{J}^*_{D,w_+,\phi_+}\subset \mc{J}_{D,w_+,\phi_+}$ such that for every $J_+\in \mc{J}^*_{D,w_+,\phi_+}$, the moduli spaces $\mc{M}^*(X_+,B,J_+)$ and $\mc{M}^*(X_+,B,J_+)^{\phi_+,\eta}$ are empty.
\end{proposition}

\begin{proof}
In order to prove this proposition, we show that the proof of \cite[Proposition 3.2.1]{MS2} can be adjusted to the smaller set of almost complex structures $\mc{J}_{D,w_+,\phi_+}$ considered here.  
 
Set 
\begin{equation}\label{J+map}
\mc{M}^*(X_+,B,\mc{J}_{D,w_+,\phi_+})\equiv\bigcup_{J_+\in \mc{J}_{D,w_+,\phi_+}} \mc{M}^*(X_+,B,J_+).
\end{equation}
Let $u\colon \P^1\to X_+$ be a $J_+$-holomorphic map in  $\mc{M}^*(X_+,B,\mc{J}_{D,w_+,\phi_+})$. 
With notation similar to the proof of \cite[Proposition 3.2.1]{MS2}, we have to show that
 $$
\tn{D}\bar{\partial}_{u,J_+}\colon W^{1,p}(u^*TX_+)\times T_{J_+}\mc{J}_{D,w_+,\phi_+} \to W^0((T^*\P^1)^{0,1}\otimes  u^*TX_+)
 $$
 is surjective.
Assume, by contradiction, that there exists a non-trivial 
$$
\gamma \in L^q((T^*\P^1)^{0,1}\otimes  u^*TX_+),
$$  
with $1/p+1/q=1$, which annihilates the image of $\tn{D}\bar{\partial}_{u,J_+}$. Then $\gamma$ is of class $W^{1,p}$ as well. By assumption, the set of points $z\in \P^1$ where $u$ is injective (i.e. (\ref{swi}) holds), $u(z)\notin D$, and 
$$
\phi_+(u(z))\not\in \tn{image}(u)
$$
is an open dense subset of $\P^1$. Choose one such point $z\in \P^1$ such that $\gamma(z)\neq 0$. Then, as in the proof of \cite[Proposition 3.2.1]{MS2}, there exists some $Y\in \tn{End}(TX_+)$, supported in a  neighborhood $U\subset X_+$ of $u(z)$, such that
$$
U\cap D =\emptyset , \quad \phi_+(U) \cap \tn{image}(u)=\emptyset ,\quad
\int_{u^{-1}(U)}\left\langle \gamma, Y\circ \tn{d}u \circ j\right\rangle >0,
$$
and $Y$ satisfies the first two conditions of (\ref{tangent-condition}); here $j$ is the complex structure of $\P^1$. We replace $Y$ over $\phi_+(U)$ by $-\phi_+^*Y$ and denote the result by $Y'$. Then, $Y'$ satisfies all the conditions in 
(\ref{tangent-condition}) and
$$
\int_{\P^1}\left\langle \gamma, Y'\circ \tn{d}u \circ j\right\rangle =\int_{u^{-1}(U)} \left\langle \gamma, Y\circ \tn{d}u \circ j\right\rangle>0.
$$
Thus, $Y'$ is an element of $T_{J_+}\mc{J}_{D,w_+,\phi_+}$ which is not annihilated by $\gamma$; this is a contradiction.

Next, we consider
\begin{equation}\label{J+map2}
\mc{M}^*(X_+,B,\mc{J}_{D,w_+,\phi_+})^{\phi_+,\eta}=\bigcup_{J_+\in \mc{J}_{D,w_+,\phi_+}} \mc{M}^*(X_+,B,J_+)^{\phi_+,\eta}.
\end{equation} 
Let $u\colon (\P^1,\eta)\to (X_+,\phi_+)$ be a real $J_+$-holomorphic map in  $\mc{M}^*(X_+,B,\mc{J}_{D,w_+,\phi_+})^{\phi_+,\eta}$. The proof is similar, but involves the real version of Banach spaces considered above. With notation as in Section~\ref{kuranishi-structure}, we have to show that
 $$
\tn{D}\bar{\partial}_{u,J_+}\colon W^{1,p}(u^*TX_+)_\R \times T_{J_+}\mc{J}_{D,w_+,\phi_+} \to W^0((T^*\P^1)^{0,1}\otimes  u^*TX_+)_\R
 $$
 is surjective. Assume, by contradiction, that there exists some non-trivial $\gamma \in L^q((T^*\P^1)^{0,1}\otimes  u^*TX_+)_\R$, which annihilates the image of $\tn{D}\bar{\partial}_{u,J_+}$. By assumption, the set of points $z\in \P^1$ where $u$ is  injective and $u(z)\notin D$ is an open dense subset of $\P^1$. Choose one such point $z\in \P^1$ such that $\gamma(z)\neq 0$. Then, as in the proof of \cite[Proposition 3.2.1]{MS2}, there exists some $Y\in \tn{End}(TX_+)$, supported in a  neighborhood $U\subset X_+$ of $u(z)$, such that
$$
U\cap D =\emptyset , \quad \phi_+(U) \cap U=\emptyset ,\quad
\int_{u^{-1}(U)} \left\langle \gamma, Y\circ du \circ j\right\rangle >0,
$$
and $Y$ satisfies the first two conditions of (\ref{tangent-condition}). We replace $Y$ over $\phi_+(U)$ by $-\phi_+^*Y$ and denote the result by $Y'$. Then, $Y'$ satisfies all the conditions in 
(\ref{tangent-condition}) and
$$
\int_{\P^1}\left\langle \gamma, Y'\circ du \circ j\right\rangle = 2\int_{u^{-1}(U)} \left\langle \gamma, Y\circ du \circ j\right\rangle >0.
$$
Thus, $Y'$ is an element of $T_{J_+}\mc{J}_{D,w_+,\phi_+}$ which is not annihilated by $\gamma$; this is a contradiction.
 \end{proof}
 
 \begin{lemma}\label{This-or-that}
For some $J_+\in \mc{J}_{D,w_+,\phi_+}$, let $u\colon \P^1 \to X_+$ be a somewhere injective $J_+$-holomorphic map whose image, as a set, is invariant under the action of $\phi_+$. Then there exists an antiholomorphic involution $c$ on $\P^1$, conjugate to $\eta$, such that $u$ is a $(\phi_+,c)$-real map.
\end{lemma}

This lemma implies that every degree $B$ somewhere injective $J_+$-holomorphic sphere $u$, where $B\cdot D>0$, either belongs to  $\mc{M}^*(X_+,B,J_+)$ or can be enhanced to an element of $\mc{M}^*(X_+,B,J_+)^{\phi_+,\eta}$.

\begin{proof}
By assumption, outside a finite set of points $S\subset \P^1$, every $z\in \P^1\setminus S$ is a somewhere injective point and $u(\P^{1}\setminus S)$ is $\phi_+$-invariant. The involution  $\phi_+$ canonically lifts to an antiholomorphic  involution $c$ on $\P^{1}\setminus S$ with no fixed point. Then, every such involution has a unique extension across entire $\P^1$ which, after a reparametrization, is isomorphic to $\eta$.
\end{proof}

Let us come back to the proof of Theorem~\ref{zeroness}. By Gromov compactness theorem and in the light of Proposition~\ref{empty-lemma}, for every $E>0$, there exists $J_+\in \mc{J}_{D,w_+,\phi_+}$ such that for all $B\in H_2(X_+,\Z)$ with $B\cdot D>0$ and $\om(B)<E$, the moduli spaces $\mc{M}^*(X_+,B,J_+)$ and $\mc{M}^*(X_+,B,J_+)^{\phi_+,\eta}$ are empty.
For such $J_+$, assume by contradiction that there exists a non-trivial element $f$ in $\ov{\mc{M}}(X_+,D,\rho,\Gamma_{+})^{\phi_+}$ with homology class $B$ and $\om(B)<E$. This element should have a smooth component, i.e. a map over some $\P^1$, which is a multiple cover of some somewhere injective map with non-trivial homology class $B'$, $\om(B')>0$ and $B'\cdot D>0$. This somewhere injective map either belongs to $\mc{M}^*(X_+,B',J_+)$, or it belongs to $\mc{M}^*(X_+,B',J_+)^{\phi_+,\eta}$, or by Lemma~\ref{This-or-that}, it can be enhanced to an element of $\mc{M}^*(X_+,B',J_+)^{\phi_+,\eta}$. This is a contradiction to the assumption on $J_+$.

\vskip.1in
Starting from a $J_+$ as in Lemma~\ref{empty-lemma} and extending it to $J_{\mc{X}}$ on $\mc{X}$, in the light of Gromov compactness theorem, and the fact that a limit of $\tau$-maps has non-trivial components in $X_-$, the conclusion of previous paragraph implies that for some $t_0>0$, all the moduli spaces  $\{\ov{\mc{M}}(X,A,J_t)^{\phi,\eta}\}_{0<t<t_0}$, where $A$ is non-trivial and $\om(A)<E$, should be empty. This finishes the proof of Theorem~\ref{zeroness}. 
\qed

\part{Odd-dimensional projective spaces \tn{(joint with A.~Zinger)}}
\section{Orientations for the moduli spaces}\label{orientPn_sec}

We give an explicit description of real maps from~$\P^1$ to~$\P^{2m-1}$
in Section~\ref{Podd} and use it in Section~\ref{AGorient_subs}
to endow the moduli spaces of such maps with orientations.
In Section~\ref{deg1_subs}, we show that the sign of the diffeomorphism
$$\ev_1\!: \mc{M}_1(\P^{2m-1},1)^{\phi,c}\lra\P^{2m-1},\quad
\big[u,(z^+,z^-)\big]\lra u(z^+),$$
is $(-1)^{m-1}$ with the respect to the algebraic orientation of Section~\ref{AGorient_subs}
on the domain and the complex orientation in the target
whenever
$$(\phi,c)=(\tau_{2m-1},\tau),(\eta_{2m-1},\eta);$$
otherwise, the moduli space above is empty.
This is also the sign of the real line through a pair of conjugate points
with respect to these orientations.
In Section~\ref{evendeg_subs}, we focus on the even-degree maps and show 
the conclusion of Proposition~\ref{matching-orientation_prp} applies to 
the algebraic orientations of Section~\ref{AGorient_subs}; see Proposition~\ref{algbnd_prp}.
In Section~\ref{orient_subs}, we describe the canonical real square root structure on $K_{\P^{4m-1}}$
and spin structure on $\R\P^{4m-1}$ induced by the exact sequence~(\ref{Pnses_e})
and used to define the numbers~(\ref{numsPndfn_e}). 

The algebraic orientations and the orientations on the moduli spaces
arising from the structures of Section~\ref{orient_subs} are compared in
Corollary~\ref{comporient_crl};
its conclusions are summarized at the end of Section~\ref{AGorient_subs}.
Along with this corollary, Proposition~\ref{algbnd_prp} provides 
a direct verification of the claim of Proposition~\ref{matching-orientation_prp}
for $(\P^{4m-1},\tau_{4m-1})$ with
the  real square root structure and spin structure of Section~\ref{orient_subs}.

For the remainder of the paper, $c\!=\!\tau,\eta$ and $\phi\!=\!\tau_{2m-1},\eta_{2m-1}$.
Define
\begin{equation}\label{absphi}
|c|=\begin{cases}0,&\hbox{if}~c\!=\!\tau;\\
1,&\hbox{if}~c\!=\!\eta; \end{cases}
 \qquad
|\phi|=\begin{cases}0,&\hbox{if}~\phi\!=\!\tau_{2m-1};\\
1,&\hbox{if}~\phi\!=\!\eta_{2m-1}. \end{cases} 
\end{equation}
We identify $0\!\in\!\C$ and $\infty$ with $[1,0]\!\in\!\P^1$ and $[0,1]\!\in\!\P^1$,
respectively.

\subsection{Spaces of parametrized maps}\label{Podd}

For $m,d\!\in\!\Z^+$, the  space $\mc{P}_0(\P^{2m-1},d)^{\phi,c}$ 
of (parametrized) $(\phi,c)$-real  degree~$d$ holomorphic maps  
consists of maps of the~form 
\begin{equation}\label{P1polyn_e}
u\!:\P^1 \lra\ \P^{2m-1}, \quad [x,y]\lra\big[p_1(x,y),q_1(x,y),\ldots,p_m(x,y),q_m(x,y)\big],
\end{equation}
where $p_1,q_1,\ldots,p_m,q_m$ are degree~$d$ homogeneous  polynomials 
in two variables without common factor which satisfy some compatibility properties. 
Suppose 
$$p_i(x,y)=A_i\!\prod_{r=1}^d\!\big(a_{i;r}x\!+\!(-1)^{|c|}b_{i;r}y\big), \quad
 q_i(x,y)=\bar{B}_i\!\prod_{r=1}^d\!\big(a_{i;r}'x\!+\!b_{i;r}'y\big).$$
The condition $u\!\circ\!c\!=\!\phi\!\circ\!\tau$ is then equivalent 
to the existence of $\zeta\!\in\!\C^*$ such~that 
\begin{equation*}\begin{split}
(-1)^{|\phi|}B_i\big(\bar{a}_{i;r}',\bar{b}_{i;r}'\big)&=
\zeta\cdot(-1)^{|c|d}A_i\big(b_{i;r},a_{i;r}\big),\\
\bar{A}_i\big(\bar{a}_{i;r},(-1)^{|c|}\bar{b}_{i;r}\big)&=
\zeta\cdot \bar{B}_i\big(b_{i;r}',(-1)^{|c|}a_{i;r}'\big)
\end{split}\end{equation*}
for all~$i$ and~$r$.
These two requirements are in turn equivalent to
\begin{equation}\label{cphicond_e}|\phi|\!+\!|c|d \in 2\Z, \quad |\zeta|=1, \quad
\ze\bar{B}_i\big(a_{i;r}',b_{i;r}'\big)=\bar{A}_i\big(\bar{b}_{i;r},\bar{a}_{i;r}\big)
~~\forall~i,r.
\end{equation}

\begin{proof}[Proof of Lemma~\ref{new-lemma}]
For $c\!=\!\eta$, the first condition in~(\ref{cphicond_e}) becomes $|\phi|\!+\!d\!\in\!2\Z$.
If it is not satisfied, the space $\mc{P}_0(\P^{2m-1},d)^{\phi,\eta}$ of parametrized maps is empty.
This immediately implies that the stratum of the moduli space $\mc{M}_l(\P^{2m-1},d)^{\phi,\eta}$
consisting of smooth maps is empty.
The claims of Lemma~\ref{new-lemma} are then obtained by observing that any map in a boundary stratum
contains a real map from $(\P^1,\eta)$ with the degree of the same parity as~$d$.
\end{proof}

From~(\ref{cphicond_e}),  we obtain the following observation.

\begin{lemma}\label{realcond_lmm}
Suppose $c\!=\!\tau,\eta$, $\phi\!=\!\tau_{2m-1},\eta_{2m-1}$, and $d\!\in\!\Z^+$
are such that $|\phi|\!+\!|c|d\!\in\!2\Z$.
The map~(\ref{P1polyn_e}) is $(\phi,c)$-real if and only if
$$p_i(x,y)=A_i\!\prod_{r=1}^d\!\big(a_{i;r}x\!+\!(-1)^{|c|}b_{i;r}y\big), \quad
 q_i(x,y)=\bar{B}_i\!\prod_{r=1}^d\!\big(\bar{b}_{i;r}x\!+\!\bar{a}_{i;r}y\big),$$
for some $A_i,B_i\!\in\!\C$ and $[a_{i;r},b_{i;r}]\!\in\!\P^1$ such that 
$$|A_i|=|B_i|~~~\forall\,i=1,\ldots,m, \qquad [A_1,\ldots,A_m]=[B_1,\ldots,B_m]\in\P^{m-1},$$
i.e.~$(B_1,\ldots,B_m)\!=\!\zeta(A_1,\ldots,A_m)$ for some $\zeta\!\in\!S^1\!\subset\!\C$.
\end{lemma}

For $a,b\!\in\!\C$, define
$$p_{a,b}\!:\C^2\lra\C, \qquad p_{a,b}(x,y)= ax\!+\!by\,.$$
Taking $m\!=\!1$, $\phi\!=\!c$, and $d\!=\!1$ in Lemma~\ref{realcond_lmm}, 
we find that 
\begin{alignat*}{2}
\big\{(a,b)\!\in\!\C^2\!:\,|a|\!\neq\!|b|\big\}/\R^*&\lra G_{\tau}\equiv\tn{Aut}(\P^1,\tau), &\quad
[a,b]&\lra\big[p_{a,b},p_{\bar{b},\bar{a}}\big],\\
\big\{(a,b)\!\in\!\C^2\!-\!0\big\}/\R^*&\lra G_{\eta}\equiv\tn{Aut}(\P^1,\eta), &\quad
[a,b]&\lra\big[p_{a,-b},p_{\bar{b},\bar{a}}\big].
\end{alignat*}
are diffeomorphisms.
In particular, $G_{\tau}$ has two topological components, with the automorphism 
$$\vt\!: \P^1\lra\P^1, \qquad  \vt\big([x,y]\big)=[y,x],$$
contained in the non-identity component.

\subsection{Algebraic orientations}
\label{AGorient_subs}

For $m,d\!\in\!\Z^+$, let $\De_{m,d}^{\eta}\!=\!\emptyset$ and
\begin{equation*}\begin{split}
\De_{m,d}^{\tau}&=
\big\{([b_{1;1},\ldots,b_{1;d}],\ldots,[b_{m;1},\ldots,b_{m;d}])\in(\tn{Sym}^d\C)^m\!:\\
&\hspace{1.8in}S^1\cap\bigcap_{i=1}^m\{b_{i;r}\!:\,r\!=\!1,\ldots,d\}\neq\emptyset\big\}.
\end{split}\end{equation*}
We identify $\R\P^{2m-1}$ with $(\C^m\!-\!\{0\})/\R^*$, viewing the $i$-th complex coordinate
as the $(2i\!-\!1)$ and $2i$-th real components.

Suppose $c\!=\!\tau,\eta$, $\phi\!=\!\tau_{2m-1},\eta_{2m-1}$, and $d\!\in\!\Z^+$
are such that $|\phi|\!+\!|c|d\!\in\!2\Z$.
By Lemma~\ref{realcond_lmm}, the~map
\begin{gather*}
\Th_c\!:
\big((\tn{Sym}^d\C)^m\!-\!\De_{m,d}^c\big)\times \R\P^{2m-1}\lra\mc{P}_0(\P^{2m-1},d)^{\phi,c}\,,\\
\begin{split}
&\big([b_{1;1},\ldots,b_{1;d}],\ldots,[b_{m;1},\ldots,b_{m;d}],[A_1,\ldots,A_m]\big)\\
&\qquad\lra\bigg[A_1\!\prod_{r=1}^d\!p_{1,(-1)^{|c|}b_{1;r}},
\bar{A}_1\!\prod_{r=1}^d\!p_{\bar{b}_{1;r},1},\ldots, 
A_m\!\prod_{r=1}^d\!p_{1,(-1)^{|c|}b_{m;r}}, \bar{A}_m\!\prod_{r=1}^d\!p_{\bar{b}_{m;r},1}\bigg],
\end{split}\end{gather*}
is a diffeomorphism over the open subset of $\mc{P}_0(\P^{2m-1},d)^{\phi,c}$ 
consisting of maps~$u$ such that $u([1,0])$ does not lie in any of 
the coordinate subspaces of~$\P^{2m-1}$.
Since the complement of this subspace is of codimension~2, $\Th_c$ induces
an orientation on  $\mc{P}_0(\P^{2m-1},d)^{\phi,c}$.
The~map
\begin{gather*}
\Th_c'\!:
\big((\tn{Sym}^d\C)^m\!-\!\De_{m,d}^c\big)\times \R\P^{2m-1}\lra\mc{P}_0(\P^{2m-1},d)^{\phi,c}\,,\\
\begin{split}
&\big([a_{1;1},\ldots,a_{1;d}],\ldots,[a_{m;1},\ldots,a_{m;d}],[B_1,\ldots,B_m]\big)\\
&\qquad\lra \bigg[\bar{B}_1\!\prod_{r=1}^d\!p_{\bar{a}_{1;r},(-1)^{|c|}},
B_1\!\prod_{r=1}^d\!p_{1,a_{1;r}},\ldots,
\bar{B}_m\!\prod_{r=1}^d\!p_{\bar{a}_{m;r},(-1)^{|c|}},
B_m\!\prod_{r=1}^d\!p_{1,a_{m;r}}\bigg],
\end{split}\end{gather*}
is also a diffeomorphism over this open subset of $\mc{P}_0(\P^{2m-1},d)^{\phi,c}$.
The  two diffeomorphisms induce
the same orientation on $\mc{P}_0(\P^{2m-1},d)^{\phi,c}$ 
if and only if \hbox{$(d\!+\!1)m\!\in\!2\Z$}.
In particular, the two orientations are the same if $d\!\not\in\!2\Z$.

The action of the automorphism $\vt$ of $(\P^1,\tau)$ lifts over $\Th$ and~$\Th'$~as 
\begin{gather*}
\big((\tn{Sym}^d\C)^m\!-\!\De_{m,d}^{\tau}\big)\times \R\P^{2m-1}\lra
\big((\tn{Sym}^d\C)^m\!-\!\De_{m,d}^{\tau}\big)\times \R\P^{2m-1}\,,\\
\begin{split}
&\big([b_{1;1},\ldots,b_{1;d}],\ldots,[b_{m;1},\ldots,b_{m;d}],[A_1,\ldots,A_m]\big)\\
& \quad\lra\big([b_{1;1}^{-1},\ldots,b_{1;d}^{-1}],\ldots,[b_{m;1}^{-1},\ldots,b_{m;d}^{-1}],
[A_1b_{1;1}\ldots b_{1;d},\ldots,A_mb_{m;1}\ldots b_{m;d}]\big).
\end{split}\end{gather*}
This lift is orientation-preserving.
Since the group~$G_{\tau}$ has two topological components, with~$\vt$ contained in the non-identity component,
and the group~$G_{\eta}$ is connected, it follows that the above orientations descend
to orientations on the quotient
$$\mc{M}_0(\P^{2m-1},d)^{\phi,c}=\mc{P}_0(\P^{2m-1},d)^{\phi,c}/G_c\,.$$ 
This implies that the moduli space $\mc{M}_l(\P^{2m-1},d)^{\phi,c}$ is orientable for all
$c,\phi$ and~$m,d,l$.

Let $\pi\!:\C^m\!-\!0\!\lra\!\R\P^{2m-1}$ be the projection map.
The standard action of $\R^*$ on $\C^m\!-\!0$ determines an isomorphism
$$\La_{\R}^{\tn{top}}\C^m\big|_{\C^m-0}\approx 
\pi^*\La_{\R}^{\tn{top}}\big(\R\P^{2m-1}\big)\otimes_{\R}\R$$
of real line bundles over $\C^m\!-\!0$.
We orient $\R\P^{2m-1}$ from the standard orientations of $\C^m$
and~$\R^+$ via this isomorphism.
Thus, $v_1,\ldots,v_{2m-1}\in\!T_w\C^m$ descends to an oriented basis for $\R\P^{2m-1}$
if $v_1,\ldots,v_{2m-1},w$ is an oriented basis for~$\C^m$.
For example, an oriented pair of vectors in each of  $m\!-\!1$ of the complex components of~$\C^m$
and the negative rotation in the remaining component determine 
an oriented basis on $\R\P^{2m-1}$.
In particular, the covering projection
$$S^1\lra \R\P^1\!=\!\C^*/\R^*, \qquad \ze\lra [\ze],$$
is orientation-reversing with respect to the standard orientation on $S^1\!\subset\!\C$
and our orientation on~$\R\P^1$.

We will call the orientations on $\mc{P}_0(\P^{2m-1},d)^{\phi,c}$  and
$\mc{M}_l(\P^{2m-1},d)^{\phi,c}$ induced by~$\Th_c$
\textsf{the algebraic orientations};
they agree with the orientations induced by~$\Th_c'$ unless $d\!\in\!2\Z$ and 
$m\!\not\in\!2\Z$.
By Corollary~\ref{comporient_crl}, the algebraic orientation of $\mc{M}_l(\P^{2m-1},d)^{\phi,c}$~is
\begin{enumerate}[label=$\bullet$,leftmargin=*]

\item the opposite of the orientation induced by
the spin structure on $\R\P^{2m-1}$ described in Section~\ref{orient_subs}
if $m\!\in\!2\Z$, $d\!\not\in\!2\Z$, and  \hbox{$(\phi,c)\!=\!(\tau_{2m-1},\tau)$},

\item the opposite of the orientation induced by
the real square root of $K_{\P^{2m-1}}$  described in Section~\ref{orient_subs}
if $m\!\in\!2\Z$, $d\!\not\in\!2\Z$, and  $(\phi,c)\!=\!(\eta_{2m-1},\eta)$,

\item the same as the orientation induced by
the relative spin structure on $\R\P^{2m-1}$ described at the end of 
Section~\ref{signpf_subs}
if $m\!\not\in\!2\Z$, $d\!\not\in\!2\Z$, and  $(\phi,c)\!=\!(\tau_{2m-1},\tau)$,

\item  the same as the orientation induced by
the spin sub-structure on the real line bundle $K_{\P^{2m-1}}$  
described at the end of Section~\ref{signpf_subs}
if $m\!\not\in\!2\Z$, $d\!\not\in\!2\Z$, and  \hbox{$(\phi,c)\!=\!(\eta_{2m-1},\eta)$}.
\end{enumerate}

\subsection{Moduli spaces of degree 1 maps}
\label{deg1_subs}

We will next  note some properties of the algebraic orientation on 
$\mc{M}_l(\P^{2m-1},1)^{\phi,c}$.

\begin{lemma}\label{orientGc_lmm}
Let $c\!=\!\tau,\eta$. 
\begin{enumerate}[label=(\arabic*),leftmargin=*]

\item The algebraic orientation on $G_c\!=\!\mc{P}_0(\P^1,1)^{c,c}$ is
the opposite of the canonical orientation specified 
at the beginning of Section~\ref{ch:realcurves}.
 
\item With respect to the algebraic orientation, $\mc{M}_0(\P^1,1)^{c,c}$ 
is a single positive point. 

\end{enumerate}
\end{lemma}

\begin{proof}
For $m\!=\!1$ and $d\!=\!1$, the map~$\Th_c$ determining
the algebraic orientation reduces~to
\begin{gather*}
\big(\C\!-\!\De_{1,1}^c\big)\!\times\!\R\P^1\lra \mc{P}_0(\P^1,1)^{c,c},\\
\big(b,[\ne^{\fI\th}]\big)\lra \frac{\ne^{-\fI\th}}{\ne^{\fI\th}}\cdot
\frac{\bar{b}x\!+\!y}{x\!+\!(-1)^{|c|}by}
=\ne^{-2\fI\th}\cdot \frac{z\!+\!\bar{b}}{1\!+\!(-1)^{|c|}bz}\,.
\end{gather*}
At $(0,[1])$, the left-hand side above is oriented by the complex orientation of~$\C$
and the negative $\th$-direction.
The right-hand side is oriented by the complex orientation of~$a$
and the positive $\th$-direction in 
$$\big(a,\ne^{\fI\th}\big)\lra \ne^{\fI\th}\cdot \frac{z\!+\!a}{1\!+\!(-1)^{|c|}\bar{a}z}\,;$$
see the beginning of Section~\ref{ch:realcurves}.
Thus, the first map above is orientation-reversing
(orientation-reversing on~$\C$ and orientation-preserving on~$\th$). 
This establishes the first claim.
The second claim of this lemma follows from the first and Lemma~\ref{ptorient_lmm}.
\end{proof}

With $c$, $\phi$, and $m$ as above, let 
$$\ev_0\!: \mc{P}_0(\P^{2m-1},1)^{\phi,c}\lra \P^{2m-1}, \qquad u\lra u(0),$$
denote the evaluation at $0\!\in\!\P^1$.
Let
$$\cN^m\P\equiv \frac{T\P^{2m+1}\big|_{T\P^{2m-1}}}{T\P^{2m-1}}
\quad\hbox{and}\quad
\cN^m\mc{P}\equiv 
\frac{T\mc{P}_0(\P^{2m+1},1)^{\phi,c}\big|_{\mc{P}_0(\P^{2m-1},1)^{\phi,c}}}
{T\mc{P}_0(\P^{2m-1},1)^{\phi,c}}$$
denote the normal bundle of $\P^{2m-1}$ in $\P^{2m+1}$ and the normal bundle of 
$\mc{P}_0(\P^{2m-1},1)^{\phi,c}$ in $\mc{P}_0(\P^{2m+1},1)^{\phi,c}$,
respectively.
The complex orientations on the projective spaces induce an orientation on~$\cN^m\P$.
If $|\phi|\!=\!|c|$, the algebraic orientations on the spaces of parametrized maps
induce an orientation on~$\cN^m\mc{P}$.
The differential of~$\ev_0$ descends to an isomorphism
\begin{equation}\label{cNisom_e}
\tn{d}\ev_0\!: \cN^m\mc{P} \lra \ev_0^*\cN^m\P\,.
\end{equation}

\begin{lemma}\label{cNisom_lmm}
Let $c\!=\!\tau,\eta$ and $\phi\!=\!\tau_{2m-1},\eta_{2m-1}$.
If $|\phi|\!=\!|c|$, the isomorphism~(\ref{cNisom_e}) is orientation-reversing
with respect to the algebraic orientation on the domain and the complex orientation
on the target.
\end{lemma}

\begin{proof}
It is sufficient to establish the claim near the image~$u_0$ of 
$$\big(0,\ldots,0,[1,0,\ldots,0]\big) \in\C^m\!\times \R\P^{2m-1}$$
under~$\Th_c$.
The left-hand side in~(\ref{cNisom_e}) is then oriented by the complex orientations
of~$A_{m+1}$ and $b_{m+1}\!\equiv\!b_{m+1;1}$.
Near~$u_0$, the map~$\ev_0$ between the normal neighborhoods can be written~as
$$\C^2\lra \C^2, \qquad \big(A_{m+1},b_{m+1}\big)\lra \big(A_{m+1},\bar{A}_{m+1}\bar{b}_{m+1}\big).$$
This map is orientation-reversing near~$u_0$. 
\end{proof}

The moduli spaces 
\begin{equation*}\begin{split}
\ov{\mc{M}}_0(\P^{2m-1},1)^{\tau_{2m-1}}&=\mc{M}_0(\P^{2m-1},1)^{\tau_{2m-1},\tau} \qquad\hbox{and}\\
\ov{\mc{M}}_0(\P^{2m-1},1)^{\eta_{2m-1}}&=\mc{M}_0(\P^{2m-1},1)^{\eta_{2m-1},\eta}
\end{split}\end{equation*}
are compact manifolds.
Using the algebraic orientations on these spaces, we can thus define the numbers
$$\wt{N}_1^{\phi}(2m\!-\!1)=\int_{[\ov{\mc{M}}_1(\P^{2m-1},1)^{\phi}]} \ev_0^*H^{2m-1}$$
for $\phi\!=\!\tau_{2m-1},\eta_{2m-1}$;
they are signed counts of real lines in $\P^{2m-1}$ passing through a pair of conjugate points.

\begin{corollary}\label{alNums_crl}
If $\phi\!=\!\tau_{2m-1},\eta_{2m-1}$, then 
$$\wt{N}_1^{\phi}(2m\!-\!1)=(-1)^{m-1}.$$
\end{corollary}

\begin{proof}
It is sufficient to show that 
\begin{equation}\label{alNums_e}
\wt{N}_1^{\phi}(1)=1, \qquad  \wt{N}_1^{\phi}(2m\!+\!1)=-\wt{N}_1^{\phi}(2m\!-\!1).
\end{equation}
The map~$u_0$ in the proof of Lemma~\ref{cNisom_lmm} is the only element of
$\mc{M}_0(\P^{2m-1},1)^{\phi}$ passing through the point 
$$P_1\equiv[1,0,\ldots,0]\in \P^{2m-1}.$$
The sign of this element is the sign of the isomorphism
\begin{equation}\label{alNums_e2}
\tn{d}_{[u_0,0]}\ev_1\!: 
T_{[u_0,0]}\mc{M}_1(\P^{2m-1},1)^{\phi}\lra T_{P_1}\P^{2m-1}.
\end{equation}
The orientation on the domain of this map is obtained via the exact sequence
\begin{equation}\label{alNums_e3}
0\lra T_0\P^1\lra T_{[u_0,0]}\mc{M}_1(\P^{2m-1},1)^{\phi}
\lra T_{[u_0]}\mc{M}_0(\P^{2m-1},1)^{\phi}\lra0
\end{equation}
from the algebraic orientation of $\mc{M}_0(\P^{2m-1},1)^{\phi}$ and 
the complex orientation of~$\P^1$.
By the second statement of Lemma~\ref{orientGc_lmm}, the first arrow in~(\ref{alNums_e3})
is an orientation-preserving isomorphism if $m\!=\!1$.
Since its composition with~(\ref{alNums_e2}) is the identity if $m\!=\!1$, 
the first claim in~(\ref{alNums_e}) holds.

Let $\cN^m\mc{M}$ denote the normal bundle of $\mc{M}_1(\P^{2m-1},1)^{\phi}$ 
in $\mc{M}_1(\P^{2m+1},1)^{\phi}$.
The differential~(\ref{alNums_e2}) induces a commutative diagram
$$\xymatrix{ 0 \ar[r]& T_{[u_0,0]}\mc{M}_1(\P^{2m-1},1)^{\phi} \ar[r]\ar[d]^{\tn{d}_{[u_0,0]}}&
T_{[u_0,0]}\mc{M}_1(\P^{2m+1},1)^{\phi} \ar[r]\ar[d]^{\tn{d}_{[u_0,0]}}& 
\cN_{[u_0]}^m\mc{M}\ar[r]\ar[d]^{\tn{d}_{[u_0,0]}}& 0\\
0 \ar[r]&  T_{P_1}\P^{2m-1}\ar[r]&  T_{P_1}\P^{2m+1} \ar[r]& \cN_{P_1}^m\P\ar[r]& 0.}$$
The second   claim in~(\ref{alNums_e}) is equivalent to the isomorphism given by 
the last vertical arrow above being orientation-reversing.
The projection
$$\mc{P}_0(\P^{2m+1},1)^{\phi}\lra  \mc{M}_1(\P^{2m+1},1)^{\phi}, \qquad
u\lra[u,0],$$
pulls back this isomorphism to the isomorphism~~(\ref{cNisom_e}) at~$[u_0]$.
The latter is orientation-reversing by Lemma~\ref{cNisom_lmm}.
\end{proof}

\subsection{Moduli spaces of even degree  maps}
\label{evendeg_subs}

Since $\tn{Fix}(\eta_{2m-1})\!=\!\emptyset$,
$$\mc{M}_l(\P^{2m-1},d)^{\eta_{2m-1},\tau}=\emptyset \qquad\forall~d\!\in\!\Z.$$
By Lemma~\ref{new-lemma},  
$$\mc{M}_l(\P^{2m-1},d)^{\eta_{2m-1},\eta}=\emptyset \qquad\forall~d\!\in\!2\Z.$$
On the other hand, Lemma~\ref{realcond_lmm} implies that the moduli spaces
\begin{equation}\label{tauspace_e}
\mc{M}_l(\P^{2m-1},2d)^{\tau_{2m-1},\tau} \qquad\hbox{and}\qquad
\mc{M}_l(\P^{2m-1},2d)^{\tau_{2m-1},\eta}
\end{equation}
are both non-empty for all $d\!\in\!\Z^+$.
They have common codimension-one boundary  
\begin{equation}\label{tauspace_e2}
\partial^1 \ov{\mc{M}}_l(\P^{2m-1},2d)^{\tau_{2m-1},\tau}
=\partial^1\ov{\mc{M}}_l(\P^{2m-1},2d)^{\tau_{2m-1},\eta}
\end{equation}
consisting of real maps from a wedge of two copies of~$\P^1$ interchanged by 
an orientation-reversing involution;
the corresponding image curves then have an isolated real node.
The first moduli space also has a boundary component 
consisting of real maps from a wedge of two copies of~$\P^1$ with an involution
preserving each copy;
the corresponding image curves then have a non-isolated real node.
By the next statement, the conclusion of Proposition~\ref{matching-orientation_prp} applies 
to the algebraic orientations on the moduli spaces~(\ref{tauspace_e}). 

\begin{proposition}\label{algbnd_prp}
For all $d\!\in\!\Z^+$, the orientations on the codimension~1 boundary~(\ref{tauspace_e2})
induced by the algebraic orientations on the moduli spaces~(\ref{tauspace_e})
are the same.
\end{proposition}

\begin{proof} It is sufficient to consider the case $l\!=\!1$.
Let 
$$\P_{2m}^{2m-1}\equiv \big\{[Z_1,\ldots,Z_{2m}]\!\in\!\P^{2m-1}\!:\,Z_{2m}\!=\!0\big\}.$$
For $c\!=\!\tau,\eta$, define 
$$\mc{M}_{\bu}(\P^{2m-1},2d)^{\tau_{2m-1},c}=
\big\{[u,0]\!\in\!\mc{M}_1(\P^{2m-1},2d)^{\tau_{2m-1},c}\!:\,u(0)\!\in\!\P_{2m}^{2m-1}\big\}$$
and let 
$$\mc{P}_{\bu}(\P^{2m-1},2d)^{\tau_{2m-1},c}\subset \mc{P}_0(\P^{2m-1},2d)^{\tau_{2m-1},c}$$
be the preimage of $\mc{M}_{\bu}(\P^{2m-1},2d)^{\tau_{2m-1},c}$ under the projection
$$ \mc{P}_0(\P^{2m-1},2d)^{\tau_{2m-1},c}\lra  \mc{M}_1(\P^{2m-1},2d)^{\tau_{2m-1},c},
\qquad u\lra [u,0].$$
The action of the subgroup $S^1\!\subset\!G_c$ of rotations around~0 restricts to 
an action  on $\mc{P}_{\bu}(\P^{2m-1},2d)^{\tau_{2m-1},c}$ and
\begin{equation}\label{wtcMc_e}\begin{split}
\mc{M}_{\bu}(\P^{2m-1},2d)^{\tau_{2m-1},c}&=
\mc{P}_{\bu}(\P^{2m-1},2d)^{\tau_{2m-1},c}\big/S^1.
\end{split}\end{equation}

Let $\mc{P}_0(\P^{2m-1},d)$ denote the space of  (parametrized)  degree~$d$ 
holomorphic maps~$u$ $\P^1\!\lra\!\P^{2m-1}$.
Define 
\begin{equation*}\begin{split}
\mc{M}_{\bu}(\P^{2m-1},d)_{\R}=
\big\{[u,0,\infty]\!\in\!\mc{M}_2(\P^{2m-1},d)\!:~&u(0)\!\in\!\P^{2m-1}_{2m},\\
&u(\infty)\!\in\!\tn{Fix}(\tau_{2m-1})\big\}
\end{split}\end{equation*}
and let
$$\mc{P}_{\bu}(\P^{2m-1},d)_{\R}\subset \mc{P}_0(\P^{2m-1},d)$$
be the preimage of $\mc{M}_{\bu}(\P^{2m-1},d)_{\R}$ under the projection
\begin{equation}\label{PCproj_e}
\mc{P}_0(\P^{2m-1},d)\lra  \mc{M}_2(\P^{2m-1},d), \qquad u\lra [u,0,\infty].
\end{equation}
Thus,
\begin{equation}\label{Z2qu_e}\begin{split}
&\partial^1 \ov{\mc{M}}_{\bu}(\P^{2m-1},2d)^{\tau_{2m-1},\tau},
\partial^1\ov{\mc{M}}_{\bu}(\P^{2m-1},2d)^{\tau_{2m-1},\eta}
= \mc{M}_{\bu}(\P^{2m-1},d)_{\R}\,.
\end{split}\end{equation}
The actions of the subgroup $S^1\!\subset\!\tn{PSL}(2,\C)$ of rotations around~0
and the subgroup $\R^+\!\subset\!\tn{PSL}(2,\C)$ of scaling from~$\infty$
restrict to  actions  on $\mc{P}_{\bu}(\P^{2m-1},d)_{\R}$ and
$$\mc{M}_{\bu}(\P^{2m-1},d)_{\R}=
\mc{P}_{\bu}(\P^{2m-1},d)_{\R}\big/(\R^+\!\times\!S^1).$$
Let
\begin{gather*}
\fL^+=\mc{P}_{\bu}(\P^{2m-1},d)_{\R}\!\times_{\R^+}\!\ov{\R^+}\lra 
\mc{P}_{\bu}(\P^{2m-1},d)_{\R}\big/\R^+, \\
\fL^+(1)= 
\big\{\big[u,\ep\big]\!\in\!\fL^+\!:\,\ep\!<\!1\big\}, ~~
\fL_0^+(1)= 
\big\{\big[u,\ep\big]\!\in\!\fL^+\!:\,\ep\!\in\!(0,1)\big\}.
\end{gather*}

Let 
$$[d]=\big\{1,\ldots,d\big\}.$$
For $\ep\!\in\!\R$ and $\b\!\in\!(\tn{Sym}^d\C)^{2m}$, let $\ep\b$ be the element
of $(\tn{Sym}^d\C)^{2m}$ obtained from~$\b$ by  multiplying all coordinates of~$\b$ by~$\ep$.
The group $\R^+$ acts on $(\tn{Sym}^d\C)^{2m}$~by 
$$\R^+\!\times\!\big(\tn{Sym}^d\C\big)^{2m} \lra\big(\tn{Sym}^d\C\big)^{2m}, \qquad
(\ze,\b)\lra \ze^{-1}\b.$$
Let 
$$\wt\cL^+=\big(\tn{Sym}^d\C\big)^{2m}\!\times_{\R^+}\!\ov{\R^+}\lra 
\big(\tn{Sym}^d\C\big)^{2m}/\R^+\,.$$
We define
\begin{equation*}\begin{split}
\R\P^{2m-1}_{\bu}&=(\C^*)^m/\R^*\subset \R\P^{2m-1}\,,\\
(\tn{Sym}^d\C)_{\bu}^{2m}&=
\big\{\big[(b_{i;r})_{r\in[d]}\big]_{i\in[2m]}\!\in\!(\tn{Sym}^d\C)^{2m}\!:
b_{2m;m+1}\!\ldots\!b_{2m;d}\!=\!0,\\
&\hspace{.2in}
\big|b_{i;r}\big|\!<\!1~\forall\,i,r,~
\bigcap_{i=1}^m\{b_{i;r}\!:r\!\in\![d]\}\cap
\bigcap_{i=1}^m\{b_{m+i;r}\!:r\!\in\![d]\}=\emptyset\big\}.
\end{split}\end{equation*}
The standard action of $S^1\!\subset\!\C$ on~$\C$ and the trivial action on $\R\P^{2m-1}_{\bu}$
induce an action~on
$$W_{\bu}\equiv (\tn{Sym}^d\C)_{\bu}^{2m}\!\times\!\R\P^{2m-1}_{\bu}\,.$$
Given an element $(\b,[\A])$ of~$W_{\bu}$ and $\ze\!\in\!\R^+$ sufficiently close to~1
(depending on $(\b,[\A])$), we define $\ze\!\cdot\!(\b,[\A])$ to be the element 
obtained by multiplying the components of~$\b$ by~$\ze^{-1}$.
Let $W_{\bu}'$ be the quotient of $W_{\bu}$ by the resulting equivalence relation.
Define
\begin{gather*}
\wt\fL^+=\wt\cL^+\!\times\!\R\P^{2m-1}_{\bu}\big|_{W_{\bu}'}, \qquad
\wt\fL_0^+=\wt\fL^+-W_{\bu}'\,,\\
\Phi\!:\wt\fL_0^+(1)\!\equiv\!
\big\{\big[\b,[\A],\ep\big]\!\in\!\wt\fL_0^+\!:\,\ep\!<\!1\big\}\lra W_{\bu}, 
\quad \Phi\big(\big[\b,[\A],\ep\big]\big)=\big(\ep\b,[\A]\big).
\end{gather*}

For $c\!=\!\tau,\eta$, the $S^1$-equivariant map
\begin{gather*}
\wh\Th_c\!:W_{\bu} \lra\mc{P}_{\bu}(\P^{2m-1},2d)^{\phi,c}\,,\\
\begin{split}
&\big(\big[(b_{i;r})_{r\in[d]}\big]_{i\in[2m]},\big[(A_i)_{i\in[m]}\big]\big)\\
&\hspace{.5in}\lra\bigg[\bigg(\!\!
A_i\!\prod_{r=1}^d\!p_{b_{i;r},1}p_{1,(-1)^{|c|}\bar{b}_{m+i;r}},
\bar{A}_i\!\prod_{r=1}^d\!p_{1,(-1)^{|c|}\bar{b}_{i;r}}p_{b_{m+i;r},1}
\!\!\bigg)_{\!\!i=1,\ldots,m}\bigg],
\end{split}\end{gather*}
is a diffeomorphism onto an open subset of the target and induces an orientation
on the latter from the canonical orientation of~$W_{\bu}$.
The induced orientation on the left-hand side of~(\ref{wtcMc_e})
is the orientation induced by the algebraic orientation on    
$\mc{M}_1(\P^{2m-1},2d)^{\tau_{2m-1},c}$ and the complex orientation on~$\P^{2m-1}_{2m}$. 
We will call this orientation the algebraic orientation as well.

The smooth map
\begin{gather*}
\wh\Th\!:  W_{\bu}\lra\mc{P}_{\bu}(\P^{2m-1},d)_{\R}\,,\\
\begin{split}
&\big(\big[(b_{i;r})_{r\in[d]}\big]_{i\in[2m]},\big[(A_i)_{i\in[m]}\big]\big)
\lra\bigg[\bigg(\!\!
A_i\!\prod_{r=1}^d\!p_{b_{i;r},1},\bar{A}_i\!\prod_{r=1}^d\!p_{b_{m+i;r},1}
\!\!\bigg)_{\!\!i=1,\ldots,m}\bigg],
\end{split}\end{gather*}
commutes with the $S^1$-actions and the local $\R^+$-actions.
Thus, it lifts to an $S^1$-equivariant diffeomorphism
$$\wt\Th\!:  \wt\fL^+\lra \fL^+\big|_{\wh\Th(W_{\bu})}$$
which is a linear isometry  on the fibers.
For $c\!=\!\tau,\eta$, define
\begin{gather*}
\Psi_c\!: \fL_0^+(1)\big|_{\wh\Th(W_{\bu})}\lra  \mc{P}_{\bu}(\P^{2m-1},2d)^{\tau_{2m-1},c}  \qquad\hbox{by}\\
\Psi_c\big(\wt\Th\big(\big[\b,[\A],\ep\big]\big)\big)=
\wh\Th_c\big(\Phi\big(\big[\b,[\A],\ep\big]\big)\big).
\end{gather*}
Thus, the diagram of~$S^1$-equivariant maps
$$\xymatrix{  & \wt\fL_0^+(1) \ar@/_2pc/[ldd]_{\Phi}  \ar@/^2pc/[rdd]^{\wt\Th} &\\
 & \mc{P}_{\bu}(\P^{2m-1},2d)^{\tau_{2m-1},\tau}& \\
W_{\bu} \ar[ru]_{\wh\Th_{\tau}} \ar[rd]^{\wh\Th_{\eta}}&& 
\fL_0^+(1)|_{\wh\Th(W_{\bu})}   \ar[lu]^{\Psi_{\tau}} \ar[ld]_{\Psi_{\eta}}\\
&  \mc{P}_{\bu}(\P^{2m-1},2d)^{\tau_{2m-1},\eta} }$$
commutes.

For each $(\b,[\A])\!\in\!W_{\bu}$, the sequence of the equivalence classes of maps
$$\big(\id,\wh\Th_c(\ep_k\b,[\A])\big)\!: \P^1\lra \P^1\!\times\!\P^{2m-1}$$
with $\ep_k\!\ra\!0$ converges to the equivalence class of the~map
$$u_{\top}\!\cup\!u_0\!\cup\!u_{\bot}\!:
\P_{\top}^1\!\cup\!\P_0^1\!\cup\!\P_{\bot}^1\lra \P^1\!\times\!\P^{2m-1}$$
as in the proof of Proposition~\ref{matching-orientation_prp}.
The second component of~$u_0$ is mapped to the point $[\A]$ in $\tn{Fix}(\tau_{2m-1})$,
while the second component of~$u_{\top}$ is the image of $\wh\Th(\b,[\A])$ 
under~(\ref{PCproj_e}).
Contracting $\P_0^1$, we obtain the image of $\wh\Th(\b,[\A])$
under the identification~(\ref{Z2qu_e}).
This implies that the~map
\begin{equation*}\begin{split}
\fL^+(1)/S^1 &\lra  \ov{\mc{M}}_{\bu}(\P^{2m-1},2d)^{\tau_{2m-1},c}, \\
\big[u,\ep\big]&\lra \begin{cases} 
\big[\Psi_c\big([u,\ep\big]\big)\big],&\hbox{if}~\ep\!\neq\!0;\\
\big[u],&\hbox{if}~\ep\!=\!0;
\end{cases}
\end{split}\end{equation*}
is a homeomorphism onto a neighborhood of 
$$\wh\Th(W_{\bu})/S^1\subset \partial^1\ov{\mc{M}}_{\bu}(\P^{2m-1},2d)^{\tau_{2m-1},c}$$
in $\ov{\mc{M}}_{\bu}(\P^{2m-1},2d)^{\tau_{2m-1},c}$.

Since the boundary~(\ref{tauspace_e2}) is connected, 
it is sufficient to establish the claim of the proposition
at the boundary elements contained in the image of $\wh\Th(W_{\bu})$ under 
the identification~(\ref{Z2qu_e}).
The substance of this claim is that the bottom diffeomorphisms in the commutative diagram
$$\xymatrix{  \mc{P}_{\bu}(\P^{2m-1},2d)^{\tau_{2m-1},\tau} \ar[d]&
\fL_0^+(1)|_{\wh\Th(W_{\bu})} \ar[l]_<<<<<{\Psi_{\tau}}  \ar[r]^<<<<<{\Psi_{\eta}} \ar[d] & 
\mc{P}_{\bu}(\P^{2m-1},2d)^{\tau_{2m-1},\tau} \ar[d]\\
 \mc{M}_{\bu}(\P^{2m-1},2d)^{\tau_{2m-1},\eta}&
\fL_0^+(1)|_{\wh\Th(W_{\bu})}/S^1 \ar[l]  \ar[r] &
\mc{M}_{\bu}(\P^{2m-1},2d)^{\tau_{2m-1},\eta}}$$
induce the same orientation on their domain from the algebraic orientations 
on the targets.
The latter are induced from the orientations of the targets of 
$\Psi_{\tau}$ and~$\Psi_{\eta}$ induced by the diffeomorphisms~$\wh\Th_{\tau}$ and~$\wh\Th_{\eta}$.
Thus, the claim is equivalent to the orientations on the domain of $\Psi_{\tau}$ and~$\Psi_{\eta}$
induced by these orientations of their target being the same.
This is immediate from the commutativity of the preceding diagram.
\end{proof}

\subsection{The canonical orientations}\label{orient_subs}

The first homomorphism $f\!=\!(f_1,\ldots,f_n)$ in~(\ref{Pnses_e})
is described by
$$\big\{f_i(\ell,\la)\big\}(a_1,\ldots,a_n)=\la a_i
\quad\forall\,(a_1,\ldots,a_n)\!\in\!\ell.$$
With $\mc{U}_i\!=\!\{[Z_1,\ldots,Z_n]\!\in\!\P^{n-1}\!:\,Z_i\!\neq\!0\}$,  let 
$$\mathbf{z}_i\!=\!(z_{i1},\ldots,z_{in}):\mc{U}_i\ra\C^n, 
\qquad\hbox{where}\quad z_{ij}=\frac{Z_j}{Z_i}\,.$$
The second homomorphism in~(\ref{Pnses_e}) over $\mc{U}_i$ is described by 
$$(p_1,\ldots,p_n)\lra
\sum_{j\neq i} \big(p_j(\mathbf{z}_i(\ell))-z_{ij}p_i(\mathbf{z}_i(\ell))\big)
\frac{\partial}{\partial z_{ij}}
\quad\forall\,p_j\in\mc{O}_{\P^{n-1}}(1)|_{\ell}.$$
It is straightforward to check that this homomorphism is independent
of the choice of~$i$ and the sequence~(\ref{Pnses_e}) is indeed exact.
This short exact sequence gives rise to a natural isomorphism
\begin{equation}\label{Pnses_e2}
\Lambda_{\C}^{\top}\big(n\mc{O}_{\P^{n-1}}(1)\big) \approx 
\Lambda_{\C}^{\top}\big(\P^{n-1}\!\times\!\C\big)\otimes
\Lambda_{\C}^{\top}(T\P^{n-1})\approx K_{\P^{n-1}}^*.
\end{equation}

We define conjugations $\fc_{\phi}^{\pm}$ on~$\C^{2m}$ by 
$$\fc_{\phi}^{\pm}\big(x_1,\ldots,x_{2m}\big)= (\pm1)^{|\phi|}
\big((-1)^{|\phi|}\bar{x}_2,\bar{x}_1,\ldots,(-1)^{|\phi|}\bar{x}_{2m},\bar{x}_{2m-1})\big).$$
The involution $\phi\!=\!\eta_{2m-1},\tau_{2m-1}$ lifts to an involution
\begin{gather*}
\Phi\!: \mc{O}_{\P^{2m-1}}(-1)\!\oplus\!\mc{O}_{\P^{2m-1}}(-1)
\lra \mc{O}_{\P^{2m-1}}(-1)\!\oplus\!\mc{O}_{\P^{2m-1}}(-1),\\
\Phi\big(\ell,x,y)=\big(\phi(\ell),\fc_{\phi}^-(y),\fc_{\phi}^+(x)\big).
\end{gather*}
In turn, this involution induces an involution on the dual bundle,
\begin{equation}\label{cO1inv_e}\begin{split}
\Phi\!: \mc{O}_{\P^{2m-1}}(1)\!\oplus\!\mc{O}_{\P^{2m-1}}(1)
&\lra \mc{O}_{\P^{2m-1}}(1)\!\oplus\!\mc{O}_{\P^{2m-1}}(1),\\
\big\{\Phi(\ell,\alpha_1,\alpha_2)\big\}(\phi(\ell),x,y)
&=\ov{\big\{(\ell,\alpha_1,\alpha_2)\big\}\big(\Phi(\phi(\ell),x,y)\big)},
\end{split}\end{equation}
and thus involutions~$\Phi$ on
\begin{equation}\label{O2mRS_e}\begin{split}
\mc{O}_{\P^{2m-1}}(2)&\approx 
\Lambda_{\C}^{\top}\big(\mc{O}_{\P^{2m-1}}(1)\!\oplus\!\mc{O}_{\P^{2m-1}}(1)\big),\\
2m\mc{O}_{\P^{2m-1}}(1)&\approx m\big(\mc{O}_{\P^{2m-1}}(1)\!\oplus\!\mc{O}_{\P^{2m-1}}(1)\big),\\
\mc{O}_{\P^{2m-1}}(2m)&\approx\Lambda_{\C}^{\top}\big(2m\mc{O}_{\P^{2m-1}}(1)\big)\approx
\mc{O}_{\P^{2m-1}}(2)^{\otimes m}
\end{split}\end{equation}
lifting~$\phi$.
The last two lifts commute with the homomorphisms in~(\ref{Pnses_e}) and 
the isomorphism~(\ref{Pnses_e2}), when $n\!=\!2m$ is even.

The isomorphisms~(\ref{Pnses_e2}) and~(\ref{O2mRS_e}) determine a real square root structure
on~$K_{\P^{4m-1}}$, as needed for orienting the moduli spaces 
$\ov{\mc{M}}_l(\P^{4m-1},d)^{\phi,\eta}$.
We describe it below.
For $i=1,2,\ldots,2m$, we define
$$\bar{i}=\begin{cases}i\!+\!1,&\hbox{if}~2\!\!\not|i;\\
i\!-\!1,&\hbox{if}~2|i.\end{cases}$$

A spin structure on $\R\P^{4m-1}\!=\!\Fix(\tau_{4m-1})$ is determined by a trivialization~of 
$T\R\P^{4m-1}\!=\!\Fix(d\tau_{4m-1})$
over any one of the $m$~circles
$$\R\P^1_i=\R\P^1_{\bar{i}}\equiv\big\{[Z_1,\ldots,Z_{4m}]\!\in\!\R\P^{4m-1}\!:
\,Z_j\!=\!0~\forall\, j\!\neq\!i,\bar{i}\big\},$$
with $i\!=\!1,2,\ldots,2m$.
Via the real part (the fixed loci of the involutions) of the short exact sequence~(\ref{Pnses_e}),
such a trivialization induces a trivialization of 
\begin{equation*}\begin{split}
\big(4m\mc{O}_{\P^{4m-1}}(1)\big)^{\R}
&\equiv\Fix\big(\Phi\!:4m\mc{O}_{\P^{4m-1}}(1)\lra 4m\mc{O}_{\P^{4m-1}}(1)\big)\\
&\approx 2m\big(2\mc{O}_{\P^{4m-1}}(1)\big)^{\R},
\end{split}\end{equation*}
with the first trivializing section being $f(\cdot,1)$.
The homotopy class of the resulting trivialization is independent of the lifts
of the $4m\!-\!1$ trivializing sections of $T\R\P^{4m-1}$ over the homomorphism~$g$
in~(\ref{Pnses_e}) and depends only on the homotopy class of the trivialization of
$T\R\P^{4m-1}$.
Furthermore, this induces a bijective correspondence between the homotopy classes
of trivializations of the two bundles.
On the other hand, any trivialization of $(2\mc{O}_{\P^{4m-1}}(1))^{\R}$ over $\R\P^1_i$
induces a trivialization of $2(2\mc{O}_{\P^{4m-1}}(1))^{\R}$, the homotopy class of 
which is independent of the choice of the first trivialization.
Therefore, there is a canonical homotopy class of trivializations of $(4m\mc{O}_{\P^{4m-1}}(1))^{\R}$
over~$\R\P^1_i$, which in turn determines a homotopy class of trivializations of
$T\R\P^{4m-1}$ over~$\R\P^1_i$ and thus a spin structure on $\R\P^{4m-1}$
(which is independent of the choice of~$i$).
This spin structure determines an orientation on $\ov{\mc{M}}_l(\P^{4m-1},d)^{\tau_{4m-1},\tau}$.

Since we trivialize the summands $(2\mc{O}_{\P^{4m-1}}(1))^{\R}$ in the same way, 
the orientations~on 
$$\Lambda_{\R}^{\top}\big(T\R\P^{4m-1}\big)=\big(\Lambda_{\C}^{\top}T\P^{4m-1}\big)^{\R}
\approx \mc{O}_{\P^{4m-1}}(2m)^{\R}\otimes\mc{O}_{\P^{4m-1}}(2m)^{\R} $$
and thus on $\R\P^{4m-1}$ induced by the canonical square root and spin structure 
are the same.
The canonical square root and spin structure are therefore {\it not} compatible
in the sense of Definition~\ref{dfn:comp-spin-sqroot}.
By Proposition~\ref{matching-orientation_prp}, 
we must thus flip the canonical orientation of either
 $\ov{\mc{M}}_l(\P^{4m-1},d)^{\tau_{4m-1},\tau}$ or $\ov{\mc{M}}_l(\P^{4m-1},d)^{\tau_{4m-1},\eta}$ 
when orienting the moduli spaces $\ov{\mc{M}}_l(\P^{4m-1},d)^{\tau_{4m-1}}$
as in Section~\ref{ch:mixed-invariants}. 
For the purposes of Sections~\ref{IntroPn_subs} and~\ref{EquivLocal_sec},
we flip the orientation of the $\eta$ moduli space.
Thus, the chosen orientation of  $\ov{\mc{M}}_l(\P^{4m-1},d)^{\tau_{4m-1}}$ agrees
 with the canonical orientation on its $\tau$-subspace and 
is the reverse of the canonical orientation on its $\eta$-subspace.

\section{Equivariant localization}\label{EquivLocal_sec}

In this section, we use equivariant localization to prove Theorem~\ref{equalGW}
by summing over the fixed loci of a torus action on $\ov{\mc{M}}_l(\P^{4m-1},d)^{\phi}$.
As in \cite[Sections~7,8]{KatzLiu} and \cite[Section~3]{PSW}, 
these loci are described by graphs with one half-edge.
The contribution of the complement of the half-edge to the normal bundle of the corresponding
locus is standard.
Proposition~\ref{CenContr_prp} determines the key contribution of the half-edge to 
the normal bundle and is thus analogous to \cite[(3)]{KatzLiu} and \cite[Lemma~6]{PSW},
though our arguments are rather different from~\cite{KatzLiu} and~\cite{PSW}.

We describe the fixed loci of a natural action of
$$\T\equiv (S^1)^m\equiv \big\{(\ze_1,\ldots,\ze_m)\!\in\!\C^m\!:\,|\ze_k|\!=\!1\big\}$$
on $\ov{\mc{M}}_l(\P^{2m-1},d)^{\phi}$ in Section~\ref{fixedloci_subs} and their normal bundles 
in Section~\ref{NormBndl_subs}.
In Section~\ref{applic_subs}, we prove Theorem~\ref{equalGW} and compute some low-degree
real invariants.
Proposition~\ref{CenContr_prp}  is proved in Section~\ref{signpf_subs}.

\subsection{Fixed loci}
\label{fixedloci_subs}

The $m$-torus $\T$ acts on $\P^{2m-1}$ by
$$(\zeta_1,\ldots,\zeta_m)\cdot[Z_1,\ldots,Z_{2m}\big] 
= [\zeta_1Z_1,\zeta_1^{-1}Z_2,\ldots,\zeta_m Z_{2m-1},\zeta_m^{-1}Z_{2m}].$$
This action commutes with the involutions $\phi\!=\!\tau_{2m-1},\eta_{2m-1}$ and 
has $2m$ fixed points,
$$P_1=[1,0, \ldots,0],\qquad\ldots\qquad P_{2m} = [0,\ldots,1].$$
We note that $\phi(P_i)=P_{\bar{i}}$.
By composition on the left, $\T$ also acts on
${\mc{M}}_l(\P^{2m-1},d)^{\phi,c}$, where $c\!=\!\tau,\eta$.

\begin{lemma}[{\cite[Lemma 3.1]{PZ}}]
The irreducible $\T$-fixed curves in $\P^{2m-1}$ are the lines $L_{ij}$ connecting
the points $p_i$ and $p_j$ with $i\!\neq\!j$.
Moreover, the irreducible $\phi$- and $\T$-fixed curves in $\P^{2m-1}$ are the lines~$L_{i\bar{i}}$.
\end{lemma}

The above $\T$-action on $\P^{2m-1}$ naturally lifts to the tautological line bundle
$$\mc{O}_{\P^{2m-1}}(-1)\lra \P^{2m-1}$$
and thus to the line bundle $\mc{O}_{\P^{2m-1}}(a)$ for every $a\!\in\!\Z$.
Let $\la_i\!\in\!H_{\T}^*$ be the equivariant first Chern class of 
$\mc{O}_{\P^{2m-1}}(1)|_{P_i}$.
Thus, 
$$\la_{\bar{i}}=-\la_i, \qquad
H_{\T}^*=\Q[\la_1,\la_3,\ldots,\la_{2m-1}].$$

Let $[f,(z_k^+,z_k^-)_k]$ be an element of  ${\mc{M}}_l(\P^{2m-1},d)^{\phi,c}$ fixed by 
the $\T$-action.
Since there are no $\T$-fixed points in $\P^{2m-1}$ that are also fixed by~$\phi$,
the domain~$\Si$ of~$f$ contains a central component $\Si_0$, while the remaining irreducible
components come in conjugate pairs. 
Furthermore, $f_0\equiv f|_{\Si_0}$ is a cover of some line
$L_{i\bar{i}}$ of some degree 
$d_0\!\in\!\Z^+$ which is branched only over~$P_i$ and~$P_{\bar{i}}$.
Every nodal and marked point of~$\Si$ and branched point of~$f$
is mapped to a fixed point~$P_j$.
If $d_0\!<\!d$ or $l\!>\!1$,
the complement of $\Si_0$ in $\Si$ consists of two nodal curves 
$\Si'$ and~$\Si''$, each with $l\!+\!1$ marked points $(x_k)_{k=0}^l$ so that
$x_0$ corresponds to the node shared with~$\Si_0$ and 
each of the remaining points is decorated by a sign~$s_k$, $+$ or~$-$, 
depending on whether
it is the first or the second point in the pair~$(z^+,z_k^-)$.

Similarly to \cite[Section~27.3]{Clay},
every fixed locus of such maps can be modeled on  a labeled tree, $\Gamma$, 
symmetric about the mid-point of a distinguished edge~$e_0$, 
which corresponds to the central component~$\Si_0$ of the $\T$-fixed maps in the locus.
Every edge~$e$ of~$\Gamma$ is labeled by some $d_e\!\in\!\Z^+$, indicating
the degree of the corresponding map; these labels 
are preserved by the reflection symmetry of~$\Gamma$.
Every vertex~$v$ is labeled by some $j_v\!=\!1,2,\ldots,2m$ 
in such a way that the reflection symmetry takes a vertex labeled~$j$ to 
a vertex labeled~$\bar{j}$.
The graph~$\Gamma$ also contains open edges which correspond to the marked points of the
domain~$\Si$; we denote by $v(k)$ the vertex to which the $k$-th marked point is attached.
Figure~\ref{decorated-graph}(a) shows one such graph describing a $\T$-locus in 
$\ov{\mc{M}}_7(\P^{2m-1},[7])^{\phi,c}$.
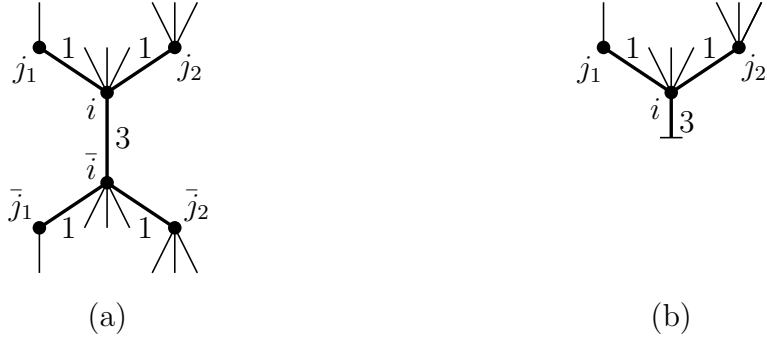
\begin{figure}
\begin{pspicture}(-3,-2.5)(11,2.4)
\psset{unit=.3cm}
\pscircle*(0,2){.3}\pscircle*(0,-2){.3}\psline[linewidth=.15](0,2)(0,-2)
\psline[linewidth=.07](0,2)(0,4)\psline[linewidth=.07](0,2)(1,4)\psline[linewidth=.07](0,2)(-1,4)
\psline[linewidth=.07](0,-2)(0,-4)\psline[linewidth=.07](0,-2)(1,-4)\psline[linewidth=.07](0,-2)(-1,-4)
\psline[linewidth=.15](0,2)(3,4)\pscircle*(3,4){.3}\psline[linewidth=.07](3,4)(3,6)
\psline[linewidth=.07](3,4)(4,6)\psline[linewidth=.07](3,4)(2,6)
\psline[linewidth=.15](0,2)(-3,4)\pscircle*(-3,4){.3}\psline[linewidth=.07](-3,4)(-3,6)
\psline[linewidth=.15](0,-2)(3,-4)\pscircle*(3,-4){.3}\psline[linewidth=.07](3,-4)(3,-6)\psline[linewidth=.07](3,-4)(4,-6)\psline[linewidth=.07](3,-4)(2,-6)
\psline[linewidth=.15](0,-2)(-3,-4)\pscircle*(-3,-4){.3}\psline[linewidth=.07](-3,-4)(-3,-6)
\rput(.7,0){$3$}\rput(-1.7,4){$1$}\rput(1.7,4){$1$}
\rput(-1.7,-4){$1$}\rput(1.7,-4){$1$}
\rput(-.7,1.3){$i$}\rput(-.7,-1.3){$\bar{i}$}
\rput(-3.6,3){$j_1$}\rput(3.7,3.1){$j_2$}
\rput(-3.8,-3){$\bar{j}_1$}\rput(4,-3){$\bar{j}_2$}

\pscircle*(25,2){.3}\psline[linewidth=.15](25,2)(25,0)\psline[linewidth=.07](24.5,0)(25.5,0)
\psline[linewidth=.07](25,2)(25,4)\psline[linewidth=.07](25,2)(26,4)\psline[linewidth=.07](25,2)(24,4)
\psline[linewidth=.15](25,2)(28,4)\pscircle*(28,4){.3}\psline[linewidth=.07](28,4)(28,6)
\psline[linewidth=.07](28,4)(29,6)\psline[linewidth=.07](28,4)(29,6)
\psline[linewidth=.07](28,4)(27,6)
\psline[linewidth=.15](25,2)(22,4)\pscircle*(22,4){.3}\psline[linewidth=.07](22,4)(22,6)
\rput(25.7,0.7){$3$}\rput(23.3,4){$1$}\rput(26.7,4){$1$}
\rput(24.3,1.3){$i$}
\rput(21.4,3){$j_1$}\rput(28.7,3.1){$j_2$}

\rput(0,-8){(a)}\rput(25,-8){(b)} 

\end{pspicture}	
\caption{A decorated graph on the left and one of its halves on the right.}\label{decorated-graph}
\end{figure}

Removing $e_0$ from $\Gamma$, we get a disconnected graph $\Gamma'\!\sqcup\!\bar\Gamma'$,
with $\bar\Gamma''$  obtained from $\Gamma'$ by replacing each vertex label~$j$
by~$\bar{j}$.
Choose one of the connected subgraphs, e.g.~$\Gamma'$, and add
the corresponding half-edge in place of the central edge; see Figure~\ref{decorated-graph}(b). 
We denote the total half graph by $\Gamma_{\half}$.
All calculations below are based on this half-graph; 
it is straightforward to check that the result is independent of which half we choose.

For each vertex $v$ in $\Gamma_{\half}$, let 
$$\ov{\mc{M}}_v=\ov{\mc{M}}_{0,\val(v)},$$
where $\val(v)$ is the valence of~$v$, i.e.~the number of edges and open edges in~$\Gamma$
leaving~$v$.
If  $\val(v)\!=\!1,2$, we take $\ov{\mc{M}}_{0,\val(v)}$ to be a point.
Let 
$$\ov{\mc{M}}_{\Gamma_{\half}}=\prod_v\ov{\mc{M}}_v\,, \qquad
D_{\Gamma_{\half}}=\big|\tn{Aut}(\Gamma_{\half})\big|\, d_0\prod_ed_e,$$
where the products are taken over the vertices $v$ and edges $e$ in~$\Gamma_{\half}$
and $\tn{Aut}(\Gamma_{\half})$ denotes the group of automorphisms of~$\Gamma_{\half}$.

\subsection{Normal bundles}\label{NormBndl_subs}

For every flag $F\!=\!(v,e)$, let $j_F\!=\!j_v$.
For every element $[f,(z^+,z^-)_k]$ in the fixed locus corresponding to~$\Gamma$,
there is an exact sequence
\begin{equation*}\begin{split}
0 \lra \Aut\big(\Si,(z_k^+,z_k^-)_k\big)_{\R}  \lra \Def(f)_{\R} &\lra 
\Def\big(f,(z_k^+,z_k^-)_k\big)_{\R} \\
&\lra \Def(\Si,(z_k^+,z_k^-)_k)_{\R} \lra 0,
\end{split}\end{equation*}
where $\Si$ is the domain of~$f$. Thus,
\begin{equation}\label{NGa_e}\begin{split}
e(N_{\Gamma})&=e\big(\Def\big(f,(z_k^+,z_k^-)_k\big)_{\R}^{\mov}\big)\\
&=\frac{e(\Def(f)_{\R}^{\mov})e(\Def(\Si,(z_k^+,z_k^-)_k)_{\R}^{\mov})}
{e(\Aut(\Si,(z_k^+,z_k^-)_k)_{\R}^{\mov})},
\end{split}\end{equation}
where ``mov" means the moving part (the part with the nonzero $\T$-weights)
and $e(\cdot)$ denotes the equivariant Euler class.
Following \cite[Section~27.4]{Clay}, we now determine the three terms appearing
on the right-hand side of~(\ref{NGa_e}). 

For each edge $e$ of $\Gamma_{\half}$, $\Aut(\Si,(z_k^+,z_k^-)_k)_{\R}$ contains 
a $\T$-fixed one-dimensional complex subspace of infinitesimal automorphisms
of the corresponding non-contracted component $\Si_e$ which fix the two branch points 
of $f_e\!\equiv\!f|_{\Si_e}$; this subspace cancels with a similar piece in $\Def(f_e)_{\R}$.
The space $\Aut(\Si,(z_k^+,z_k^-)_k)_{\R}$ also contains a $\T$-fixed one-dimensional real 
subspace of infinitesimal automorphisms of the central component~$\Si_0$;
this subspace cancels with a similar piece in $\Def(f_0)_{\R}$,
up to sign taken into account by Proposition~\ref{CenContr_prp}.
The remaining automorphisms, none of which is $\T$-fixed, correspond to 
the vertices~$v$ in~$\Gamma_{\half}$ of valence~1;
they describe the infinitesimal automorphisms moving the branch point~$x_v$ of~$f_e$,
where $e$ is the unique edge containing~$v$, that lies over~$j_v$.
Thus, similarly to \cite[Section~27.4]{Clay},
\begin{equation}\label{eAut_e}\begin{split}
e\big(\Aut(\Si,(z_k^+,z_k^-)_k)_{\R}^{\mov}\big)
&=\prod_{\begin{subarray}{c}v\in e\\ \val(v)=1\end{subarray}}\!\!\!\!\!e(T_{x_v}\Si_e)
=\prod_{\begin{subarray}{c}v\in e\\ \val(v)=1\end{subarray}}\!\!\!\!\!w_{(v,e)}\,,\\
\hbox{where}\qquad
w_{(v,\{v,v'\})}&=\frac{\la_{j_v}-\la_{j_{v'}}}{d_{\{v,v'\}}}\,.
\end{split}
\end{equation}

A deformation of a contracted component of the domain (as a marked curve) is $\T$-fixed.
The moving deformations come from smoothing (conjugate pairs) of nodes of~$\Si$. 
For each node~$x$ of $\Si$ corresponding to $\Gamma_{\half}$,
$\Def(\Si,(z_k^+,z_k^-)_k)_{\R}^{\mov}$ contains the complex one-dimensional space isomorphic
to the tensor product of the tangent spaces of the two components of~$\Si$ sharing~$x$. 
There are two possibilities. 
Each $v\!\in\!\Gamma_{\half}$ shared by two edges contributes $w_{F_1}\!+\!w_{F_2}$, 
where $F_1$ and $F_2$ are the two flags containing~$v$.
Each flag $F\!=\!(v,e)$ with $v\!\in\!\Gamma_{\half}$ and $\val(v)\!\ge\!3$ 
contributes $w_F\!-\!\psi_F$, where $\psi_F\!\in\!H^2(\ov{\mc{M}}_v)$ is 
the first Chern class of the universal cotangent bundle on $\ov{\mc{M}}_v$ corresponding to
the marked point determined by~$F$ on the contracted curve determined by the vertex~$v$.
Thus,
\begin{equation}\label{eDefSi_e}\begin{split}
e\big(\Def(\Si,(z_k^+,z_k^-)_k)_{\R}^{\mov}\big)&=
\prod_{\begin{subarray}{c}\val(v)=2\\ v\in e_1,e_2\\ e_1\neq e_2\end{subarray}}
\!\!\!\!\big(w_{(v,e_1)}\!+\!w_{(v,e_2)}\big)\\
&\hspace{.5in}\times
\prod_{\val(v)\ge3}\prod_{v\in e}\big(w_{(v,e)}\!-\!\psi_{(v,e)}\big).
\end{split}\end{equation}

Finally, there is an exact sequence 
\begin{equation*}\begin{split}
0&\lra \Def(f)_{\R}  \\
&\lra H^0(\Si_{e_0}, f_0^*T\P^{2m-1})_{\R} \oplus 
\bigoplus_{e\neq e_0} \!H^0(\Si_e,f_e^*T\P^{2m-1}) 
\oplus \bigoplus_v \! T_{p_{j_v}}\P^{2m-1}  \\
&\lra \bigoplus_F \!T_{p_{j_F}} \P^{2m-1} \lra 0,
\end{split}\end{equation*}
where the direct sums are taken over the vertices $v$, edges $e$, and flags $F$ in $\Gamma_{\half}$.
Thus,
\begin{equation}\label{eDefMap_e}\begin{split}
&e\big(\Def(f)_{\R}^{\mov}\big)=\prod\limits_v\prod\limits_{j\neq j_v}\!\!(\la_{j_v}\!-\!\la_j)\\
&\hspace{.4in}\times
\frac{e(H^0(\Si_{e_0}, f_0^*T\P^{2m-1})_{\R}^{\mov})
\prod\limits_{e\neq e_0}\!\!\! e(H^0(\Si_e,f_e^*T\P^{2m-1})^{\mov})}
{\prod\limits_F\prod\limits_{j\neq j_F}\!\!(\la_{j_F}\!-\!\la_j)}\,.
\end{split}\end{equation}
The contribution of $e\!\neq\!e_0$ is standard and given by
\begin{equation}\label{eContr1_e}\begin{split}
&e\big(H^0(\Si_e,f_e^*T\P^{2m-1})^{\mov}\big)\\
&\qquad= (-1)^{d_e} \frac{d_e!^2}{d_e^{2d_e}} 
\big(\la_{j_1}\!-\!\la_{j_2}\big)^{2d_e} \prod_{r=0}^{d_e}\prod_{k\neq j_1,j_2}\!\!\! 
\left(\frac{r\la_{j_1}\!+\!(d_e\!-\!r)\la_{j_2}}{d_e}- \la_k\right),
\end{split}\end{equation}
where $j_1$ and $j_2$ are the two vertex labels of the edge $e$; 
see \cite[Section 27.4]{Clay}.
The contribution of the half-edge $e_0$ is described by the next lemma,
which is proved in Section~\ref{signpf_subs}.

\begin{proposition}\label{CenContr_prp}
Let $\phi\!=\!\tau_{4m-1},\eta_{4m-1}$, $c\!=\!\tau,\eta$, and
$$f_0\!:\big(\P^1,[1,0],[0,1]\big)\ra\big(\P^{4m-1},P_i,P_{\bar{i}}\big)$$ 
be the degree~$d_0$  
cover of a line $L_{i\bar{i}}$ branched over only $P_i$ and $P_{\bar{i}}$
and intertwining the involutions $c$ and~$\phi$. 
With respect to the canonical orientation of the moduli space 
$\ov{\mc{M}}_0(\P^{4m-1},d_0)^{\phi,c}$ as in Section~\ref{orient_subs},
\begin{equation}\label{eContr2_e}\begin{split}
&e\big(H^0(\Si_{e_0},f_0^*T\P^{4m-1})_{\R}^{\mov}\big)\\
&\hspace{.5in}=
(-1)^{d_0}d_0!\bigg(\frac{2\la_i}{d_0}\bigg)^{\!\!d_0}
\!\!\!\!\!\!\!\!
\prod_{\begin{subarray}{c}1\le j\le 4m\\ j\neq i,\, 2|(j-i)\end{subarray}}
\!\!\prod_{r=0}^{d_0}
\!\!\bigg(\!\frac{d_0\!-\!2r}{d_0}\la_i-\la_j\!\bigg).
\end{split}\end{equation}
\end{proposition}

\subsection{Applications}\label{applic_subs}

By the classical localization theorem of \cite{ABo},
\begin{equation}\label{ABform_e}\begin{split}
N_d^{\phi}(t_1,\ldots,t_l)&=\sum_{\Gamma\triangleright\tau}
\frac{1}{D_{\Gamma_{\half}}}\int_{{\mc{M}}_{\Gamma_{\half}}}\!\!\!\!\!\!
\frac{\prod\limits_{k=1}^ls_k^{t_k+1}\la_{j_{v(k)}}^{t_k}}{e(N_{\Gamma})}\\
&\hspace{.7in}-\sum_{\Gamma\triangleright\eta}
\frac{1}{D_{\Gamma_{\half}}}\int_{{\mc{M}}_{\Gamma_{\half}}}\!\!\!\!\!\!
\frac{\prod\limits_{k=1}^ls_k^{t_k+1}\la_{j_{v(k)}}^{t_k}}{e(N_{\Gamma})}\,,
\end{split}\end{equation}
where the first and second sums are taken over the graphs $\Gamma$ corresponding to
the fixed loci~in 
\begin{equation}\label{Mphi_e}
\ov{\mc{M}}_l(\P^{4m-1},d)^{\phi,\tau} \qquad\hbox{and}\qquad 
\ov{\mc{M}}_l(\P^{4m-1},d)^{\phi,\eta},
\end{equation}
respectively.
By Section~\ref{Podd}, such graphs satisfy $|\phi|\!+\!|c|d_0\!\in\!2\Z$
with $c\!=\!\tau,\eta$, respectively.
The negative sign in~(\ref{ABform_e}) arises due to the fact that we flip the orientation~of 
the second moduli space above when gluing it to the first; 
see the last paragraph of Section~\ref{orient_subs}.
Along with (\ref{NGa_e})-(\ref{eContr2_e}), (\ref{ABform_e}) provides an explicit way 
of computing the numbers~(\ref{numsPndfn_e}).

\begin{proof}[Proof of Theorem~\ref{equalGW}]
Suppose $t_k\!\in\!2\Z$ for some $k$. 
Given any graph $\Gamma$ corresponding to a fixed locus in either moduli space in~(\ref{Mphi_e}),
let $\Gamma'$ be the graph obtained from $\Gamma$ by changing the sign of the $k$-th marked point.
By~(\ref{ABform_e}), the contribution of $\Gamma'$ to $N_d^{\phi}(t_1,\ldots,t_l)$
is the negative of the contribution of~$\Gamma$.
Thus,  $N_d^{\phi}(t_1,\ldots,t_l)$ vanishes.

Suppose $d\!\in\!2\Z$. 
The graphs $\Gamma$ describing  the fixed loci in the  spaces~(\ref{Mphi_e}) 
with $\phi\!=\!\tau_{4m-1}$ are the same.
By~(\ref{ABform_e}), this implies that $N_d^{\tau_{4m-1}}(t_1,\ldots,t_l)$ vanishes.
By Lemma~\ref{new-lemma}, both  spaces~(\ref{Mphi_e}) are empty for $\phi\!=\!\eta_{4m-1}$.
Thus, $N_d^{\eta_{4m-1}}(t_1,\ldots,t_l)$ also vanishes.

Suppose $d\!\not\in\!2\Z$. 
By Lemma~\ref{new-lemma}, the second  space in~(\ref{Mphi_e}) for $\phi\!=\!\tau_{4m-1}$ 
and the first  space in~(\ref{Mphi_e}) for $\phi\!=\!\eta_{4m-1}$ are empty.
The graphs $\Gamma$ corresponding to the fixed loci in
 the first  space in~(\ref{Mphi_e}) for $\phi\!=\!\tau_{4m-1}$ 
and the second  space in~(\ref{Mphi_e}) for $\phi\!=\!\eta_{4m-1}$ are the same.
Along with~(\ref{ABform_e}), this implies~(\ref{numsPn_e}).
\end{proof}

If $d,t_1,\ldots,t_l$ are odd, (\ref{ABform_e}) gives
\begin{equation}\label{ABform_e2}
N_d^{\tau_{4m-1}}(t_1,\ldots,t_l) =-N_d^{\eta_{4m-1}}(t_1,\ldots,t_l)
=2^{l-1}\sum_{\Gamma_{\half}}
\frac{1}{D_{\Gamma_{\half}}}\int_{{\mc{M}}_{\Gamma_{\half}}}\!\!
\frac{\prod\limits_{k=1}^l\la_{j_{v(k)}}^{t_k}}{e(N_{\Gamma})},
\end{equation}
where the sum is taken over all half-graphs $\Gamma_{\half}$ corresponding to 
$\ov{\mc{M}}_l(\P^{4m-1},d)^{\tau_{4m-1},\tau}$ containing marked points 
with the $+$ sign only.

\begin{example}[$d\!=\!1$]\label{d1_eg}
We now establish (\ref{d1N_e1}).
For $d\!=\!1$ and $t_1,\ldots,t_l\!\in\!\Z^+$ odd, 
(\ref{ABform_e2}) and Proposition~\ref{CenContr_prp} give
$$N_1^{\tau_{4m-1}}(t_1,\ldots,t_l) 
=-2^l\sum_{i=1}^{2m} \frac{\la_{i}^{t_1+\ldots+t_l}}
{2\la_{i}\prod\limits_{\begin{subarray}{c}1\le j\le 2m\\ j\neq i\end{subarray}}
\!\!\!\!\!(\la_{j}^2\!-\!\la_{i}^2)}(2\la_{i})^{-(l-1)},$$
after formally replacing $(\la_1,\la_3,\ldots)$ by  $(\la_1,\la_2,\ldots)$. 
Using the second condition in~(\ref{d1N_e2}) and the residue theorem on $S^2$, we obtain
\begin{equation}\label{d1eg_e2}
N_1^{\tau_{4m-1}}(t_1,\ldots,t_l) 
=\sum_{i=1}^{2m}\underset{z=\la_{i}^2}\Res\frac{z^{2m-1}dz}{\prod\limits_{j=1}^{2m}(z\!-\!\la_{j}^2)}
=-\underset{z=\infty}\Res\frac{z^{2m-1}dz}{\prod\limits_{j=1}^{2m}(z\!-\!\la_{j}^2)}
=1.
\end{equation}
\end{example}

\begin{example}[$d\!=\!3$]\label{d3_eg}
We now establish~(\ref{d3N_e1})   using Pandharipande's trick of twisting by 
the equivariant weights to reduce the number of the contributing torus fixed loci;
the restrictions of the integrand to the remaining loci vanish. 
Let 
$$ J =\big\{1,3,\ldots,4m\!-\!1\big\}-\{1,3\}$$
By the last condition in~(\ref{d3N_e2}), $J\!=\!J_1\!\sqcup\!\!J_2$
for some $J_1,J_2\!\subset\!J$ with \hbox{$|J_i|\!=\!(t_i\!-\!3)/2$}.
Set
\begin{equation*}\begin{split}
\alpha_k&=(H\!+\!\la_1)(H^2\!-\!\la_3^2) \prod_{j\in J_k} (H^2\!-\!\la_j^2),\quad k\!=\!1,2,\\
\alpha_3&=(H\!+\!\la_3)(H^2\!-\!\la_1^2) \prod_{j\in J} (H^2\!-\!\la_j^2).
\end{split}\end{equation*}
We now apply the equivariant localization theorem of~\cite{ABo} to compute
\begin{equation}\label{d3twist_e}
N_3^{\tau_{4m-1}}(t_1,t_2,4m\!-\!1)= \int_{\ov{\mc{M}}_3(\P^{4m-1},3)^{\tau_{4m-1}}} 
\!\!\!\tn{ev}_1^*\alpha_1\,\tn{ev}_2^*\alpha_2\,\tn{ev}_3^*\alpha_3 \,.
\end{equation}
The restriction of $\ev_3^*\alpha_3$ to a torus fixed locus vanishes unless the marked point~3 is
sent to~$P_3$.
For $k\!=\!1,2$, 
the restriction of $\ev_k^*\alpha_k$ to a fixed locus vanishes unless the marked point~$i$ is
sent to~$P_j$ with $j\!=\!1$, or $j\!\in\!J\!-\!J_k$, or $\bar{j}\!\in\!R\!-\!J_k$.
Since any half-graph has at most two vertices in this case and $J\!\subset\!J_1\!\cup\!J_2$,
the restriction of the integrand in~(\ref{d3twist_e}) to a torus fixed locus vanishes 
unless the marked points~1 and~2 are sent to~$P_1$.
Thus,  Figure~\ref{fig:2pairs} shows all half-graphs contributing to~(\ref{d3twist_e}).
From~(\ref{ABform_e2}) and Proposition~\ref{CenContr_prp}, we thus~obtain 
\begin{equation*}\begin{split}
N_3^{\tau_{4m-1}}(t_1,t_2,4m\!-\!1)&= 
 \frac{(3\la_1\!-\!\la_3)(\la_1\!+\!\la_3)}{2\la_1(\la_1\!-\!\la_3)}
+\frac{(3\la_1\!+\!\la_3)(\la_1\!-\!\la_3)}{2\la_1(\la_1\!+\!\la_3)}\\
&\hspace{.7in}+
\frac{\la_1(\la_3\!+\!\la_1)}{\la_3(\la_3\!-\!\la_1)}-
\frac{\la_1(\la_3\!-\!\la_1)}{\la_3(\la_3\!+\!\la_1)}=-1.
\end{split}\end{equation*}
\end{example}

\begin{figure}
\begin{pspicture}(-3,-2)(11,1.7)
\psset{unit=.3cm}
\pscircle*(-8,2){.3}
\psline[linewidth=.15](-8,2)(-8,0)\psline[linewidth=.1](-8,2)(-9,3.5)\psline[linewidth=.1](-8,2)(-9.5,3)
\psline[linewidth=.15](-8,2)(-3,2)\psline[linewidth=.1](-3,2)(-1.5,3.5)
\pscircle*(-3,2){.3}
\psline[linewidth=.1](-8.7,-0)(-7.3,0)
\rput(-7.7,3.2){$1$}\rput(-3.2,3){$3$}\rput(-5.4,3){$1$}
\rput(-9.3,4.3){\scriptsize{$2^+$}}\rput(-10.2,3){\scriptsize{$1^+$}}\rput(-1,4){\scriptsize{$3^+$}}
\rput(-9.1,0){$1$}

\pscircle*(3,2){.3}
\psline[linewidth=.15](3,2)(3,0)\psline[linewidth=.1](3,2)(2,3.5)\psline[linewidth=.1](3,2)(1.5,3)
\psline[linewidth=.15](3,2)(8,-2)
\pscircle*(8,-2){.3}
\psline[linewidth=.1](2.3,-0)(3.7,0)\psline[linewidth=.1](8,-2)(9.5,-3.5)
\rput(3.3,3.2){$1$}
\rput(2,4.3){\scriptsize{$2^+$}}\rput(0.8,3){\scriptsize{$1^+$}}\rput(10,-4){\scriptsize{$3^-$}}
\rput(7.8,-3){$4$}\rput(2,0){$1$}
\rput(6,0.6){$1$}

\pscircle*(16,2){.3}
\psline[linewidth=.15](21,2)(21,0)\psline[linewidth=.1](16,2)(15,3.5)\psline[linewidth=.1](16,2)(14.5,3)
\psline[linewidth=.15](16,2)(21,2)\psline[linewidth=.1](21,2)(22.5,3.5)
\pscircle*(21,2){.3}
\psline[linewidth=.1](20.3,-0)(21.7,0)
\rput(16.3,3.2){$1$}\rput(20.8,3){$3$}
\rput(15,4.3){\scriptsize{$2^+$}}\rput(13.8,3){\scriptsize{$1^+$}}
\rput(23.2,4.2){\scriptsize{$3^+$}}
\rput(18.6,3){$1$}

\pscircle*(27,-2){.3}
\psline[linewidth=.15](32,2)(32,0)
\psline[linewidth=.1](32,2)(33.5,3.5)
\pscircle*(32,2){.3}
\psline[linewidth=.15](27,-2)(32,2)\psline[linewidth=.1](27,-2)(26,-3.5)\psline[linewidth=.1](27,-2)(25.5,-3)
\psline[linewidth=.1](31.3,-0)(32.7,0)
\rput(31.8,3){$3$}
\rput(34.2,4.2){\scriptsize{$3^+$}}
\rput(27.3,-3.2){$2$}
\rput(25.7,-4.3){\scriptsize{$2^-$}}\rput(24.8,-3){\scriptsize{$1^-$}}
\rput(19.9,0){$1$}\rput(33.1,0){$1$}\rput(29,0.6){$1$}

\rput(-5.5,-5.5){$(1)$}
\rput(5.5,-5.5){$(2)$}
\rput(18.5,-5.5){$(3)$}
\rput(29.5,-5.5){$(4)$}
\end{pspicture}	

\caption{The half-graphs contributing to the localization computation of 
$N_3^{\tau_{2m-1}}(t_1,t_2,4m\!-\!1)$ 
with the constraints  $\alpha_1,\alpha_2,\alpha_3$.}\label{fig:2pairs}
\end{figure}
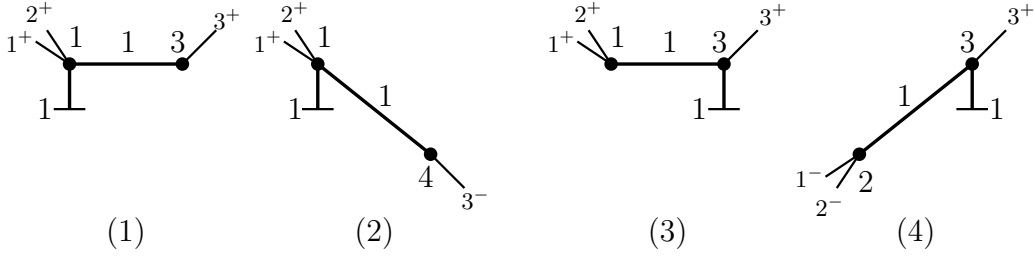

\subsection{Comparisons of orientations}\label{signpf_subs}

Most of this section is dedicated to establishing Proposition~\ref{CenContr_prp}.
We then compare the orientation on $\ov{\mc{M}}_l(\P^{4m-1},d)^{\phi,c}$
induced by the real square root and the spin structure of Section~\ref{orient_subs}
with the algebraic orientation defined in Section~\ref{AGorient_subs}; see Corollary~\ref{comporient_crl}.

\begin{proof}[Proof of Proposition~\ref{CenContr_prp}]
For $i\!=\!1,\ldots,2m$, let
$$L_{i\bar{i}}^{\bu}\equiv L_{i\bar{i}}-\big\{P_i,P_{\bar{i}}\big\}$$
denote the complement of the torus fixed points of the line $L_{i\bar{i}}$.
The involution~$\Phi$ on $\mc{O}_{\P^{2m-1}}(1)\!\otimes\!\mc{O}_{\P^{2m-1}}(1)$ 
induced by~(\ref{cO1inv_e}) is given~by
\begin{equation*}
\big\{\Phi(\ell,\alpha)\big\}
\big(\phi(\ell),x\!\otimes\!y\big)
=-\ov{\alpha\big(\fc_{\phi}^-(y)\!\otimes\!\fc_{\phi}^+(x)\big)}.
\end{equation*}
The restriction of 
$$\mc{O}_{\P^{2m-1}}(1)\otimes\mc{O}_{\P^{2m-1}}(1)
\approx\Lambda_{\C}^{\top}\big(\mc{O}_{\P^{2m-1}}(1)\!\oplus\!\mc{O}_{\P^{2m-1}}(1)\big)$$
to $L_{i\bar{i}}^{\bu}$ is trivialized by the homomorphism
$$(\ell,\alpha)\ra \big(\ell,\mf{i} z_{i\bar{i}}^{-1}
\alpha(\mathbf{z}_i(\ell)\!\otimes\!\mathbf{z}_i(\ell))\big).$$
Via this trivialization, the above involution
on $\mc{O}_{\P^{2m-1}}(1)\!\otimes\!\mc{O}_{\P^{2m-1}}(1)$ 
corresponds to the standard involution on 
$L_{i\bar{i}}^{\bu}\!\times\!\C$
lifting the restriction of~$\phi$.
Thus, the trivialization
\begin{equation}\label{dettriv_e}\begin{split}
\mc{O}_{\P^{2m-1}}(2m)\big|_{L_{i\bar{i}}^{\bu}}
&\lra L_{i\bar{i}}^{\bu}\!\times\!\C, \\
(\ell,\alpha)&\lra \big(\ell,\mf{i}^m z_{i\bar{i}}^{-m}\alpha(\mathbf{z}_i(\ell)^{\otimes 2m})\big),
\end{split}\end{equation}
is an admissible trivialization of $(\mc{O}_{\P^{2m-1}}(2m),\Phi)$ induced by a trivialization
of its real square root via the canonical isomorphism 
$$\mc{O}_{\P^{2m-1}}(2m)\approx \mc{O}_{\P^{2m-1}}(m)\otimes \mc{O}_{\P^{2m-1}}(m)$$
if $m\!\in\!2\Z$.

In the case $\phi\!=\!\tau_{2m-1}$,
\begin{equation*}\begin{split}
\big(2\mc{O}_{\P^{2m-1}}(1)\big)^{\R}
=\big\{(\alpha_1,\alpha_2)\!\in\!2\mc{O}_{\P^{2m-1}}(1)|_{\ell}\!:~&\ell\!\in\!\R\P^{2m-1},\\
&\alpha_2(x)=\ov{\alpha_1(\fc_{\phi}^+(x))}~\forall\,x\!\in\!\ell\big\}.
\end{split}\end{equation*}
Thus, we can trivialize $(2\mc{O}_{\P^{2m-1}}(1))^{\R}$ over $\R\P^1_i$ by 
\begin{equation}\label{cO2triv_e}
(\ell,\alpha_1,\alpha_2)\ra \big(\ell,\alpha_1(\mathbf{z}_i(\ell))\big)\in \R\P^1_i\!\times\!\C.
\end{equation}
Therefore, by  Section~\ref{orient_subs},  the trivialization
\begin{equation}\label{Psidfn_e}\begin{split}
\Psi\!:\big(2m\mc{O}_{\P^{2m-1}}(1)\big)^{\R}&\ra  \R\P^1_i\!\times\!\C^{\{j:\,2|(j-i)\}},\\
\Psi(\ell,\alpha_1,\ldots,\alpha_{2m})_j&=\alpha_j(\mathbf{z}_i(\ell)),
\end{split}\end{equation}
determines the canonical spin structure on $\R\P^{2m-1}$ if $m\!\in\!2\Z$.

The standard coordinate vector fields on $\mc{U}_i\!\subset\!\P^{2m-1}$
as in Section~\ref{orient_subs} induce a trivialization of $T\P^{2m-1}$ 
along~$L_{i\bar{i}}^{\bu}$.
Let 
$$\Phi_i,\wt\phi\!:  L_{i\bar{i}}^{\bu}\!\times\!\C^{2m-1}\lra
L_{i\bar{i}}^{\bu}\!\times\!\C^{2m-1}$$
be the conjugation induced by this trivialization and 
the lift of~$\phi$ to the standard conjugation, respectively.
The composition 
$$\tilde{\phi}\!\circ\!\Phi_i\!:
L_{i\bar{i}}^{\bu}\lra  \tn{GL}(2m\!-\!1,\C)$$
is given~by
\begin{equation}\label{Conjtriv_e}
\big(\tilde{\phi}\!\circ\!\Phi_i(z_{i\bar{i}})\big)_{j_1j_2}=
\begin{cases}
(-1)^{|\phi|+1}z_{i\bar{i}}^{-2}\,,&\hbox{if}~j_1,j_2=\bar{i};\\
(-1)^{|\phi|(i+j_1)}z_{i\bar{i}}^{-1}\,,&\hbox{if}~j_1\!=\!\bar{j}_2\neq i,\bar{i};\\
0,&\hbox{otherwise}.
\end{cases}\end{equation}
Define
$$A_i\!: L_{i\bar{i}}^{\bu}\lra\tn{GL}(2m\!-\!1,\C)$$
by
\begin{equation}\label{Adfn_e}
\big(A_i(z_{i\bar{i}})\big)_{j_1j_2}=
\begin{cases}
\mf{i} a z_{i\bar{i}}^{-1}\,,&\hbox{if}~j_1,j_2=\bar{i};\\
(-\mf{i})^{j_1-j_2-1},&\hbox{if}~j_1\!\in\!\{j_2,\bar{j}_2\},2|(j_2\!-\!i);\\
\mf{i}^{j_1+j_2}z_{i\bar{i}}^{-1},&\hbox{if}~j_1\!\in\!\{j_2,\bar{j}_2\},2|(j_2\!-\!\bar{i});\\
0,&\hbox{otherwise}.
\end{cases}\end{equation}
By~(\ref{Conjtriv_e}), the composition of the above trivialization of~$T\P^{2m-1}$
over $L_{i\bar{i}}^{\bu}$ with $(\id,A_i)$
intertwines $\tn{d}\phi$
with~$\tilde\phi$ whenever $a\!\in\!\R^*$.

We order the standard coordinate vector fields along $L_{i\bar{i}}^{\bu}$
so that $\frac{\partial}{\partial z_{i\bar{i}}}$ is listed first, followed by the pairs
consisting of 
$\frac{\partial}{\partial z_{ij}}$ and $\frac{\partial}{\partial z_{i\bar{j}}}$
with $j\!\neq\!i$ and $2|(j\!-\!i)$.
The corresponding element of $\La_{\C}^{\top}(T\P^{2m-1})$ is then mapped to 
$$\det\big(A_i(z_{i\bar{i}})\big)=(-1)^{(m-1)\bar{i}}
\mf{i}^m 2^{m-1}az_{i\bar{i}}^{-m}$$
under the trivialization of $\La_{\C}^{\top}(T\P^{2m-1})$ over 
$L_{i\bar{i}}^{\bu}$ induced 
by the composite trivialization of~$T\P^{2m-1}$.
On the other hand, the image of this $(2m\!-\!1)$-tensor under the canonical 
isomorphism~(\ref{Pnses_e2}) followed by the trivialization~(\ref{dettriv_e})
is $(-1)^{m\bar{i}}\mf{i}^m z_{i\bar{i}}^{-m}$.
Thus, the two trivializations of $(K_{\P^{n-1}}^*,\Phi)$ over 
$L_{i\bar{i}}^{\bu}$ are homotopic in the sense
of Definition~\ref{admissible} if and only if $(-1)^{\bar{i}}a\!>\!0$.

In the case $\phi\!=\!\tau_{2m-1}$, the above composite trivialization restricts to 
a trivialization of $T\R\P^{2m-1}$ over~$\R\P_i^1$.
The first trivializing section is $-\mf{i} a^{-1}z_{i\bar{i}}\frac{\partial}{\partial z_{i\bar{i}}}$,
followed~by 
\begin{equation}\label{coordvect_e}
\frac{1}{2}\bigg(-\mf{i}\frac{\partial}{\partial z_{ij}}
+\mf{i} z_{i\bar{i}}\frac{\partial}{\partial z_{i\bar{j}}} \bigg)
\qquad\hbox{and}\qquad
\frac{(-1)^{\bar{i}}}{2}\bigg(\frac{\partial}{\partial z_{ij}}
+z_{i\bar{i}}\frac{\partial}{\partial z_{i\bar{j}}} \bigg)
\end{equation}
with $j\!\neq\!i$ and $2|(j\!-\!i)$.
Lifting these sections over the homomorphism~$h$ in the real part of the short exact 
sequence~(\ref{Pnses_e}) and combining with the image of~$f$, we obtain a trivialization
of $(2m\mc{O}_{\P^{2m-1}}(1))^{\R}$.
The composition of this trivialization with the trivialization~(\ref{Psidfn_e}) sends 
the two standard real basis elements in one $\C$-factor to~1 and $\mf{i}/2a$
and in $(m\!-\!1)$ $\C$-factors to $-\mf{i}/2$ and~$(-1)^{\bar{i}}/2$.
This transformation is orientation-preserving if and only if $(-1)^{(m-1)\bar{i}}a\!>\!0$.
If $m\!\in\!2\Z$,  the trivialization of $(2m\mc{O}_{\P^{2m-1}}(1))^{\R}$ over~$\R\P_i^1$ 
just discussed thus differs from a canonical one by a constant matrix-valued function which is
orientation-preserving if and only if $(-1)^{\bar{i}}a\!>\!0$.
Therefore, the trivialization of $T\R\P^{2m-1}$ over~$\R\P_i^1$ induced by the above composite
trivialization corresponds to the chosen spin structure on~$\R\P^{2m-1}$ if
$(-1)^{(m-1)\bar{i}}a\!>\!0$.

In summary, the orientation of $H^0(\Si_0,f_0^*T\P^{2m-1})_{\R}$ induced
by the above composite trivialization is the orientation induced~by
\begin{enumerate}[label=$\bullet$,leftmargin=*]

\item the chosen square root  if $c\!=\!\eta$ and  $(-1)^{\bar{i}}a\!>\!0$,

\item the chosen spin structure if $c\!=\!\tau$ and $(-1)^{(m-1)\bar{i}}a\!>\!0$.

\end{enumerate}

By (\ref{Conjtriv_e}), the components of a section $s\!\in\!H^0(\Si_0,f_0^*T\P^{2m-1})_{\R}$
with respect to the trivialization of $f_0^*T\P^n$ over $\P^1\!-\!\{0,\infty \}$
induced by the coordinate tangent vectors along $L_{i\bar{i}}^{\bu}$ 
satisfy
\begin{alignat}{2}
\label{newcoeff_e1}
s_{\bar{i}}(z)&=z^{d_0}\sum_{r=-d_0}^{d_0}s_{\bar{i};r}z^r,&\quad
s_{\bar{i};r}&=(-1)^{1+|c|r}\ov{s_{\bar{i};-r}};\\
\label{newcoeff_e2}
s_j(z)&=\sum_{r=0}^{d_0}s_{j;r}z^r,&\quad s_{j;r}&=(-1)^{|\phi|(i+j)+|c|r}\ov{s_{\bar{j};d_0-r}}     
\qquad\forall\,j\!\neq\!i,\bar{i}.
\end{alignat}
Therefore, the complex coefficients 
$$s_{\bar{i};-r}~~\hbox{with}~~r\!=\!1,\ldots,d_0, \quad\hbox{and}\quad
s_{j;r}~~\hbox{with}~~r\!=\!0,1,\ldots,d_0,~j\!\neq\!\bar{i},~2|(j\!-\!\bar{i}),$$ 
and the real coefficient $\mf{i} s_{\bar{i};0}$ give coordinates on $H^0(\Si_0,f_0^*T\P^{2m-1})_{\R}$.
With $A_i$ as in~(\ref{Adfn_e}), 
\begin{equation*} 
(A_is)_j(z)=\sum_{r=-d_0}^{d_0} \!\! b_{j;r}z^r,\quad 
\end{equation*}
where $b_{\bar{i};r}\!=\!\mf{i} as_{\bar{i};r}$ and 
\begin{gather*}
b_{j;0}=2
\begin{cases}
\tn{Im}(s_{\bar{j};d_0}),&\hbox{if}~j\!\neq\!i,\,2|(j\!-\!i);\\
(-1)^{\bar{i}}\tn{Re}(s_{j;d_0}),&\hbox{if}~j\!\neq\!\bar{i},\,2|(j\!-\!\bar{i});
\end{cases}\\
(b_{j;r},b_{j;-r})=
\begin{cases}
\mf{i}(s_{j;r},-s_{\bar{j};d_0-r}),&\hbox{if}~r\!\ge\!0,\,j\!\neq\!i,\,2|(j\!-\!i);\\
(-1)^{\bar{i}}(s_{\bar{j};r},s_{j;d_0-r}),&\hbox{if}~r\!\ge\!0,j\!\neq\!\bar{i},\,2|(j\!-\!\bar{i}).
\end{cases}
\end{gather*}
We note that $b_{j;r}\!=\!(-1)^{|c|r}\ov{b_{j;-r}}$ for $j\!\neq\!i$, as expected.

The weights of the $\T$-actions on the coordinate function $z_{i\bar{i}}$ 
and the coordinate vector fields~$\frac{\partial}{\partial z_{ij}}$
are $-2\la_i/d_0$ and $\la_i\!-\!\la_j$, respectively.
Thus, the weights of the $\T$-actions on the sections 
\begin{alignat*}{2}
&z^{d_0}\cdot s_{\bar{i};-r}z^{-r}\frac{\partial}{\partial z_{i\bar{i}}}
&\quad &\hbox{with}\quad r\!=\!1,\ldots,d_0, \qquad\hbox{and}\\
&s_{j;d_0-r}z^{d_0-r}\frac{\partial}{\partial z_{ij}}
&\quad &\hbox{with}\quad r\!=\!0,1,\ldots,d_0,~j\!\neq\!i,\bar{i},
\end{alignat*}
are given by
$$\bigg(1-d_0\frac{1}{d_0}+\frac{r}{d_0}\bigg)2\la_i \qquad\hbox{and}\qquad
\la_i\!-\!\la_j-
\frac{d_0\!-\!r}{d_0}2\la_i\,,$$
respectively.
Under the collapsing procedures of Lemma~\ref{trivialization-to-orientation} 
and \cite[Proposition 8.1.4]{FOOO}, 
the parts of $A_is$ involving negative and positive powers of~$z$ correspond 
to the holomorphic sections on $\Si_{\top}$ and~$\Si_{\bot}$, respectively.
Since we use the complex orientation of sections on~$\Si_{\top}$, these parts contribute
\begin{equation*}\begin{split}
&d_0!\bigg(\frac{2\la_i}{d_0}\bigg)^{\!\!d_0}\!\!\!\!\!\!\!\!
\prod_{\begin{subarray}{c}1\le j\le 2m\\ j\neq\bar{i},\, 2|(j-\bar{i})\end{subarray}}
\!\!\prod_{r=1}^{d_0}\!\bigg(\!\frac{2r\!-\!d_0}{d_0}\la_i-\la_j\!\!\bigg)\\
&\hspace{1.2in}=(-1)^{(m-1)d_0} d_0!\bigg(\frac{2\la_i}{d_0}\bigg)^{\!\!d_0}\!\!\!\!\!\!\!\!
\prod_{\begin{subarray}{c}1\le j\le 2m\\ j\neq i,\, 2|(j-i)\end{subarray}}
\!\!\prod_{r=1}^{d_0}\!\bigg(\!\frac{d_0\!-\!2r}{d_0}\la_i-\la_j\!\!\bigg)
\end{split}\end{equation*}
to $e(H^0(\Si_{e_0},f_0^*T\P^{4m-1})_{\R}^{\mov})$.

The parts of $A_is$ constant in~$z$ correspond to holomorphic sections on $\Si_0$ commuting
with the involution and constitute the direct sum of the trivial representation of~$\T$ 
on the space of sections $\{b_{\bar{i};0}z^0\!:\,b_{\bar{i};0}\!\in\!\R\}$ and 
of the two-dimensional representations of weight 
$$(-1)^i\big(\!-\!\la_i\!-\!\la_j\big)=(-1)^{\bar{i}}\big(\la_i\!-\!\la_{\bar{j}}\big)$$ 
with $j\!\neq\!\bar{i}$ and $2|(j\!-\!\bar{i})$.
Combining with the previous displayed expression, we find~that
\begin{equation*}\begin{split}
&e\big(H^0(\Si_{e_0},f_0^*T\P^{2m-1})_{\R}^{\mov}\big)\\
&\hspace{1in}=(-1)^{(m-1)(\bar{i}+d_0)} d_0!\bigg(\frac{2\la_i}{d_0}\bigg)^{\!\!d_0}\!\!\!\!\!\!\!\!
\prod_{\begin{subarray}{c}1\le j\le 2m\\ j\neq i,\, 2|(j-i)\end{subarray}}
\!\!\prod_{r=0}^{d_0}\!\bigg(\!\frac{d_0\!-\!2r}{d_0}\la_i-\la_j\!\!\bigg)
\end{split}\end{equation*}
if the number $a\!\in\!\R^*$ in~(\ref{Adfn_e}) has the correct sign 
and the space of sections $\{b_{\bar{i};0}z^0\!\}$
is oriented by the positive direction of $b_{\bar{i};0}\!\in\!\R$.

The characteristic vector field~$s$ for the action of $S^1\!\subset\!G_c$
which fixes the point $z\!=\!0$ in $\C\!\subset\!\P^1$ is given~by
\begin{equation}\label{charVF_e}
(A_is)_{\bar{i}}= \mf{i} a\, z^{-d_0}\frac{d}{d\theta}\big(e^{-\mf{i}\theta}z\big)^{d_0}\Big|_{\theta=0}
=ad_0, \quad (A_is)_j=0~~\forall\,j\neq i,\bar{i}.
\end{equation}
Combining this with the bullet points above, we  conclude that
\begin{equation}\label{TeulerClass_e}\begin{split}
&e\big(H^0(\Si_{e_0},f_0^*T\P^{2m-1})_{\R}^{\mov}\big)\\
&\hspace{.5in}=(-1)^{(m-1)d_0+m|c|\bar{i}} d_0!\bigg(\frac{2\la_i}{d_0}\bigg)^{\!\!d_0}\!\!\!\!\!\!\!\!
\prod_{\begin{subarray}{c}1\le j\le 2m\\ j\neq i,\, 2|(j-i)\end{subarray}}
\!\!\prod_{r=0}^{d_0}\!\bigg(\!\frac{d_0\!-\!2r}{d_0}\la_i-\la_j\!\!\bigg).
\end{split}\end{equation}
Taking $m\!\in\!2\Z$, we obtain~(\ref{eContr2_e}).
\end{proof}

\begin{remark}\label{comporient_rmk2a}
For $m\!\in\!\Z^+$ and $i\!=\!1,\ldots,2m$, let $L_{i\bar{i}}^+\!\subset\!L_{i\bar{i}}$
denote the disk cut out by $\R\P^1_i$ that contains~$P_i$.
The projection
$$\big(2\mc{O}_{\P^{2m-1}}(1)\big)^{\R} \lra \mc{O}_{\P^{2m-1}}(1)\big|_{\R\P^{2m-1}}$$
to the first component is an isomorphism.
Thus, the trivialization~(\ref{cO2triv_e}) induces a trivialization~$\Psi_i'$ of 
$\mc{O}_{\P^{2m-1}}(1)$ over~$\R\P^1_i$.
It extends over~$L_{i\bar{i}}^+$ by the same formula.
The trivialization obtained from~(\ref{cO2triv_e}) by evaluating~$\alpha_2$,
instead of~$\alpha_1$, differs from this trivialization~by 
the orientation and spin:
$$\alpha_2\big(\z_i(\ell)\big)=\ov{\alpha_1\big(\fc_{\phi}^+(\z_i(\ell))\big)}
=\ov{\alpha_1\big(\ov{z_{i\bar{i}}(\ell)}\z_i(\ell)\big)}
=z_{i\bar{i}}(\ell)\ov{\alpha_1\big(\z_i(\ell)\big)}.$$
For $m\!\not\in\!2\Z$, the reasoning of Section~\ref{orient_subs} determines a canonical spin
structure~on  
$$T\R\P^{2m-1}\oplus \mc{O}_{\P^{2m-1}}(1)\big|_{\R\P^{2m-1}} \approx
T\R\P^{2m-1}\oplus  \big(2\mc{O}_{\P^{2m-1}}(1)\big)^{\R} $$
and thus a \textsf{relative spin structure} on $\R\P^{2m-1}$.
If $i$ is odd, the restriction of this trivialization to~$\R\P^1_i$ is equivalent to the direct sum 
of the trivialization~$\Psi$ used in the proof of Proposition~\ref{CenContr_prp}
and the trivialization~$\Psi_i'$.
If $i$ is even, this restriction differs from the direct sum by the orientation and spin.
\end{remark}

\begin{remark}\label{comporient_rmk2b}
For $m,d\!\in\!\Z^+$, 
let $\mc{M}_0^{\disk}(\P^{2m-1},d)$ denote the moduli space of
holomorphic disk maps to~$\P^{2m-1}$ with boundary on $\R\P^{2m-1}$
that double to degree $d$ holomorphic maps.
Thus, 
\begin{equation}\label{comporient_e2b}
\mc{M}_0(\P^{2m-1},d)^{\tau_{2m-1},\tau} = 
\mc{M}_0^{\disk}(\P^{2m-1},d)\big/\tau_{\mc{M}}\,.
\end{equation}
For $m\!\not\in\!2\Z$, the relative spin structure of Remark~\ref{comporient_rmk2a} 
determines an orientation on the disk space in~(\ref{comporient_e2b});
see \cite[Theorem~8.1.1]{FOOO}.
By \cite[Corollary~5.9]{XCapsSigns}, this orientation descends to 
the left-hand side in~(\ref{comporient_e2b}).
By Remark~\ref{comporient_rmk2a}, the last orientation is the orientation determined 
in the proof of Proposition~\ref{CenContr_prp} if $i$ is odd.
If $i$ is even and $d_0$ is odd, the trivialization of $f_0^*T\R\P^{2m-1}$ over~$S^1$
in this proof
differs from the trivialization induced by the relative spin structure of Remark~\ref{comporient_rmk2a}
by the orientation and the spin.
Each of these changes by itself would reverse the induced orientation on~(\ref{comporient_e2b});
the two of them together preserve~it.
If $i$ and $d_0$ are even, the trivialization of $f_0^*T\R\P^{2m-1}$ over~$S^1$
in the proof 
differs from the  relative spin trivialization only by the orientation;
this change reverses the induced orientation on~(\ref{comporient_e2b}).
In summary, (\ref{TeulerClass_e}) describes the orientation on
the moduli space in~(\ref{comporient_e2b}) induced by the relative spin structure
of Remark~\ref{comporient_rmk2a}
 unless $d_0$ and~$i$ are even.
\end{remark}

\begin{remark}\label{comporient_rmk2c}
Let $\fa\!:S^1\!\lra\!S^1$ denote the antipodal involution.
A \textsf{spin substructure} on $(K_{\P^{2m-1}},d\phi)$ in the sense of 
the paragraph above \cite[Corollary~5.10]{XCapsSigns} consists of a trivialization
of this real bundle pair over every real loop 
$$\alpha\!:(S^1,\fa)\lra(\P^{2m-1},\phi)$$ 
so that these trivializations extend over homotopies of such loops.
Any such substructure orients $\mc{M}_0(\P^{2m-1},d)^{\phi,\eta}$.
It can be specified by the trivialization over~$\R\P^1_i$ viewed as the boundary 
of~$L_{i\bar{i}}^+$ to be given by~(\ref{dettriv_e}).
If $m\!\not\in\!2\Z$ and $\phi\!=\!\eta_{2m-1}$, the proof of \cite[Lemma~3.4]{XCapsSigns}
implies that the induced trivialization over $\R\P^1_{\bar{i}}\!=\!\R\P^1_{i}$ 
viewed as the boundary of~$L_{\bar{i}i}^+$ is then the {\it opposite} of~(\ref{dettriv_e}).
Since (\ref{dettriv_e}) is invariant under the interchange of~$i$ and~$\bar{i}$,
this interchange thus changes the spin substructure used to orient $\mc{M}_0(\P^{2m-1},d)^{\phi,\eta}$
in Lemma~\ref{trivialization-to-orientation} and the proof of Proposition~\ref{CenContr_prp}.
Thus, the interchange of~$i$ and~$\bar{i}$ changes the orientation of the moduli space
if $m\!\not\in\!2\Z$, $\phi\!=\!\eta_{2m-1}$, and $d\!\not\in\!2\Z$
(if $d\!\in\!2\Z$, this moduli space is empty).
If  $\phi\!=\!\tau_{2m-1}$, then $(K_{\P^{2m-1}},d\phi)$ admits 
a real square.
By \cite[Corollary~2.4(2)]{XCapsSetup}, reversing the orientation of a real loop~$\alpha$
does not change the trivialization in a spin substructure.
Thus, the interchange of~$i$ and~$\bar{i}$ preserves the orientation of the moduli space
if $\phi\!=\!\tau_{2m-1}$ and $d\!\in\!2\Z$
(if $d\!\not\in\!2\Z$, this moduli space is empty).
\end{remark}

\noindent
We note that the right-hand side of~(\ref{TeulerClass_e}) has the expected behavior 
if~$i$ is replaced by~$\bar{i}$.
This interchange changes the sign of the right-hand side of~(\ref{TeulerClass_e}) 
if and only~if
$$\de_m^c(d_0)\equiv m\big(d_0\!+\!1\!+\!|c|\big)$$
is even.
If the target of the fibration
\begin{equation}\label{TeulerClass_e2}
\ov{\mc{M}}_1(\P^{2m-1},d)^{\phi,c}\lra\ov{\mc{M}}_0(\P^{2m-1},d)^{\phi,c}
\end{equation}
is oriented at~$[f_0]$ using~(\ref{dettriv_e}) if $c\!=\!\eta$ and~(\ref{Psidfn_e}) if $c\!=\!\tau$,
then~(\ref{TeulerClass_e}) describes the corresponding orientation 
of  the domain of~(\ref{TeulerClass_e2}) at the map~$f_0$ with 
the positive marked point sent to~$P_i$.
Interchanging~$i$ and~$\bar{i}$ reverses the orientation of the fiber of~(\ref{TeulerClass_e2}) 
over~$[f_0]$ and thus of the domain of~(\ref{TeulerClass_e2}).
Thus, $\de_m^c(d_0)$ should be even if and only if
the orientation on the target in~(\ref{TeulerClass_e2}) does not change under this interchange.
If $m\!\in\!2\Z$, (\ref{dettriv_e}) and~(\ref{Psidfn_e}) are the restrictions
of a real square root over~$\P^{2m-1}$ and of a spin structure on~$\R\P^{2m-1}$,
respectively, which orient the target in~(\ref{TeulerClass_e2});
in this case $\de_m^c(d_0)$ is even.

Suppose $m\!\not\in\!2\Z$. 
If $c\!=\!\tau$, (\ref{TeulerClass_e}) corresponds to the orientation of
the target in~(\ref{TeulerClass_e2}) induced by the relative spin structure
of Remark~\ref{comporient_rmk2a} unless~$d_0$ and~$i$ are even;
see Remark~\ref{comporient_rmk2b}.
Thus,  replacing~$i$ by~$\bar{i}$ preserves its orientation 
if $d_0\!\not\in\!2\Z$ (when $\de_m^c(d_0)$ is even)
and reverses it if $d_0\!\in\!2\Z$ (when $\de_m^c(d_0)$ is odd).
If $c\!=\!\eta$ and $d\!\in\!2\Z$ (and thus $\phi\!=\!\tau_{2m-1}$),
(\ref{TeulerClass_e}) corresponds to the orientation of
the target in~(\ref{TeulerClass_e2}) induced by a spin substructure on $(K_{\P^{2m-1}},d\phi)$;
see Remark~\ref{comporient_rmk2c}.
In this case, $\de_m^c(d_0)$ is indeed even.
If $c\!=\!\eta$ and $d\!\not\in\!2\Z$ (and thus $\phi\!=\!\eta_{2m-1}$), 
 replacing~$i$ by~$\bar{i}$  changes the spin substructure used to orient 
the target in~(\ref{TeulerClass_e2}) and reverses its orientation.
In this case, $\de_m^c(d_0)$ is indeed odd.

Let $c\!=\!\tau,\eta$ and $\phi\!=\!\tau_{2m-1},\eta_{2m-1}$ be such that 
$2|(|\phi|\!-\!|c|d_0)$. 
Suppose
$$f_0\!:\P^1\!=\!\Si_{e_0}\lra \P^{2m-1}$$
is as above.
Along with the characteristic vector field of the $S^1$-action on the parameter space,
(\ref{TeulerClass_e}) determines an orientation on 
\begin{equation}\label{cMcP_e}
 H^0(\P^1,f_0^*T\P^{2m-1})_{\R} =T_{f_0}\mc{P}_0(\P^{2m-1},d)^{\phi,c}.
\end{equation}
Another orientation on this space is described in Section~\ref{AGorient_subs}.

\begin{corollary}\label{comporient_crl}
The two orientations on the vector space in~(\ref{cMcP_e}) are the same if 
and only if $m((d_0\!+\!1)i\!+\!1\!+\!|c|\bar{i})\!\not\in\!2\Z$.
\end{corollary}

\begin{proof}
The algebraic orientation on~(\ref{cMcP_e}) is induced by the diffeomorphism~$\Th_c$ 
in Section~\ref{AGorient_subs}.
The orientation induced by the diffeomorphism~$\Th_c'$ differs from this orientation
by $(-1)^{m(d_0+1)}$.
Let~$\Th_i$ denote the first diffeomorphism if $i\!\not\in\!2\Z$ and the second one if $i\!\in\!2\Z$.
The comparison below is thus made with the algebraic orientation multiplied by
$(-1)^{m(d_0+1)\bar{i}}$.

Let $1_i\!\in\!\C^m$ denote the unit coordinate vector for the component $\flr{(i\!+\!1)/2}$.
It is sufficient to establish the claim near the image~$f_0$ of 
$$\big([0,\ldots,0],\ldots,[0,\ldots,0],[1_i]\big)
\in\big(\tn{Sym}^{d_0}\C\big)^m\!\times \R\P^{2m-1}$$
under~$\Th_i$.
The deformations of the coefficients of $z^r$ in 
$$A_j\prod_{r=1}^{d_0}(1\!+\!(-1)^{|c|}b_{j;r}z)\bigg/
A_i\prod_{r=1}^{d_0}\big(1\!+\!(-1)^{|c|}b_{i;r}z\big)$$
for $j\!\neq\!i$ and $2|(j\!-\!i)$ correspond to the coefficients $s_{j;r}$ in~(\ref{newcoeff_e2}), 
up to a complex multiple.
The deformations of the coefficients of~$z^{d_0+r}$ in 
$$ \bar{A}_i\prod_{r=1}^{d_0}(\bar{b}_{i;r}\!+\!z) 
\bigg/A_i\prod_{r=1}^{d_0}(1\!+\!(-1)^{|c|}b_{i;r}z)$$
with $r\!>\!0$  correspond to the coefficients $s_{\bar{i};d_0+r}$ in~(\ref{newcoeff_e1}).
Thus, the complex orientations on $(\tn{Sym}^{d_0}\C)^m$ and on $A_j\!\in\!\C$ 
with $j\!\neq\!i$ and $2|(j\!-\!i)$ give 
\begin{equation}\label{TeulerClass_e4}\begin{split}
&e\big(H^0(\Si_{e_0},f_0^*T\P^{2m-1})_{\R}^{\mov}\big)\\
&\hspace{.5in}
=(-1)^{d_0}d_0!\bigg(\frac{2\la_i}{d_0}\bigg)^{\!\!d_0}\!\!\!\!\!\!\!\!
\prod_{\begin{subarray}{c}1\le j\le 2m\\ j\neq i,\, 2|(j-i)\end{subarray}}
\!\!\prod_{r=0}^{d_0}\!\bigg(\!\frac{d_0\!-\!2r}{d_0}\la_i-\la_j\!\!\bigg).
\end{split}\end{equation}

The characteristic vector field corresponding to the $S^1$ part of the variation of~$A_i$ 
is $-2\mf{i}z^{d_0}\frac{\partial}{\partial z_{i\bar{i}}}$,
i.e.~a positive multiple of the characteristic  vector field of the $S^1$-action on the parameter space.
Our orientation of $\R\P^{2m-1}$ is determined by the complex orientations on 
$A_j\!\in\!\C$, with $j\!\neq\!i$ and $2|(j\!-\!i)$,
and the {\it negative} characteristic vector field of the $S^1$-action on~$A_i$.
Comparing~(\ref{TeulerClass_e}) and~(\ref{TeulerClass_e4}), we then obtain the claim.
\end{proof}

\noindent
The orientations on the moduli spaces~(\ref{cmeq_e})
induced by the canonical spin structure on $\R\P^{4m-1}$
and the canonical real square root of $(K_{\P^{4m-1}},\eta_{4m-1})$
described in Section~\ref{orient_subs} are thus the opposite of 
the algebraic orientations of Section~\ref{AGorient_subs}.
This is reflected in the opposite signs for the line counts of Example~\ref{d1_eg}
and Corollary~\ref{alNums_crl} with $m$ replaced by~$2m$.
By Corollary~\ref{comporient_crl} and Remark~\ref{comporient_rmk2b}, 
the orientation on $\ov{\mc{M}}_l(\P^{4m+1},d)^{\tau_{4m+1},\tau}$
induced by the canonical relative spin structure of Remark~\ref{comporient_rmk2a}
is the same as the algebraic orientation if $d\!\not\in\!2\Z$ and 
the opposite if $d\!\in\!2\Z$.
The orientation on $\ov{\mc{M}}_l(\P^{4m+1},d)^{\tau_{4m+1},\eta}$
induced by the  spin substructure on $(\P^{4m+1},{\tau_{4m+1}})$ of Remark~\ref{comporient_rmk2c}
is the opposite of  the algebraic orientation.
The orientation on $\ov{\mc{M}}_l(\P^{4m+1},d)^{\eta_{4m+1},\eta}$
induced by the  spin substructure on $(\P^{4m+1},{\eta_{4m+1}})$ 
determined by the trivialization~(\ref{dettriv_e}) over $\R\P^1_1$ oriented 
as the boundary of~$L_{12}^+$
is the same of  the algebraic orientation.

\begin{remark}\label{comporient_rmk}
Corollary~\ref{comporient_crl} implies that Proposition~\ref{CenContr_prp}
with the leading sign exponent changed to $(-1)^{(d_0+1)(1+m\bar{i})}$ applies
 to the moduli spaces
\begin{equation}\label{comporient_e}
\ov{\mc{M}}_l(\P^{2m-1},d)^{\phi,\tau} \qquad\hbox{and}\qquad 
\ov{\mc{M}}_l(\P^{2m-1},d)^{\phi,\eta}
\end{equation}
with the algebraic orientations of Section~\ref{AGorient_subs} for $d\!\not\in\!2\Z$,
whether $m$ is odd or even.
The resulting half-edge contribution then changes sign when~$i$ is replaced by~$\bar{i}$,
as expected.
With $4m$ replaced by~$2m$ in Example~\ref{d1_eg}, 
this change introduces the sign of $(-1)^{m-1}$ in~(\ref{d1eg_e2}) and recovers
Corollary~\ref{alNums_crl}.
The remaining considerations of Sections~\ref{fixedloci_subs} and~\ref{NormBndl_subs} 
apply to the algebraic orientations of these moduli spaces,
with one important difference if $m\!\not\in\!2\Z$.
By \cite[Lemma~3.1]{GZ2}, the correct orientation for the smoothing of the disk node
differs from the complex one by $(-1)^{md^+}$, where $d_+\!=\!(d\!-\!d_0)/2$ 
is the degree of the graph~$\Gamma'$ attached to the disk.
This is also consistent with the sentence preceding this remark and 
\cite[Remark~2.6]{RealEnum}.
This introduces two new signs into the computation Example~\ref{d1_eg},
resulting in the same answer.
The statement and proof of Theorem~\ref{equalGW} apply to the algebraic orientations
of the moduli spaces~(\ref{comporient_e}), after reversing the orientation on the second space
(as needed for $d\!\in\!2\Z$ by Proposition~\ref{algbnd_prp}).
\end{remark}



\end{document}